\documentclass[11pt]{article}
\usepackage{amssymb}
\usepackage{amsmath}
\usepackage{mathrsfs}
\usepackage{enumerate}
\usepackage{graphics}
\usepackage{graphicx}
\usepackage{dsfont}
\usepackage{subfigure}
\usepackage[T1]{fontenc}
\usepackage{latexsym,amssymb,amsmath,amsfonts,amsthm}
\usepackage{txfonts}
\newcommand{\R}{{\mathbb R}}
\newcommand{\Z}{{\mathbb Z}}
\newcommand{\N}{{\mathbb N}}

\newcommand{\Sp}{{\mathbb S}}
\newcommand{\ds}{\displaystyle}
\newcommand{\no}{\nonumber}
\newcommand{\be}{\begin{eqnarray}}
\newcommand{\ben}{\begin{eqnarray*}}
\newcommand{\en}{\end{eqnarray}}
\newcommand{\enn}{\end{eqnarray*}}

\newcommand{\pa}{\partial}

\newcommand{\ov}{\overline}

\newcommand{\Rt}{{\rm Re}}

\newcommand{\G}{\Gamma}

\newcommand{\ol}{\overline}

\newcommand{\half}{\frac{1}{2}}

\newtheorem{theorem}{Theorem}[section]

\newtheorem{remark}[theorem]{Remark}

\newtheorem{algorithm}{Algorithm}[section]

\begin{document}
\renewcommand{\theequation}{\arabic{section}.\arabic{equation}}
\begin{titlepage}
\title{\bf Imaging of locally rough surfaces from intensity-only far-field or near-field data}
\author{Bo Zhang and Haiwen Zhang\\
LSEC and Institute of Applied Mathematics, AMSS\\
Chinese Academy of Sciences, Beijing 100190, China\\
(emails: {\sf b.zhang@amt.ac.cn} and {\sf zhanghaiwen@amss.ac.cn})
}
\date{}
\end{titlepage}
\maketitle

\vspace{.2in}

\begin{abstract}
This paper is concerned with a nonlinear imaging problem, which aims to reconstruct a locally perturbed,
perfectly reflecting, infinite plane from intensity-only (or phaseless) far-field or near-field data.
A recursive Newton iteration algorithm in frequencies is developed to reconstruct the locally rough surface
from multi-frequency intensity-only far-field or near-field data, where the fast integral equation solver
developed in \cite{ZhangZhang2013} is used to solve the direct scattering problem in each iteration.
For the case with far-field data, a main feature of our work is that the incident field is taken as
a superposition of two plane waves with different directions rather than one plane wave, so the location
and shape of the local perturbation of the infinite plane can be reconstructed simultaneously from
intensity-only far-field data with multiple wave numbers.
This is different from previous work on inverse scattering from phaseless far-field data,
where only the shape reconstruction was considered due to the translation invariance property of the
phaseless far-field pattern corresponding to one plane wave as the incident field.
Finally, numerical examples are carried out to demonstrate that our reconstruction algorithm is
stable and accurate even for the case of multiple-scale profiles.
\end{abstract}

\section{Introduction}

We consider problems of scattering of time-harmonic electromagnetic waves by a locally
perturbed infinite plane (which is called a locally rough surface).
Such problems occurs in many applications such as radar, remote sensing, geophysics, medical
imaging and nondestructive testing (see, e.g. \cite{SS01,V99,WC01}).

In this paper, we are restricted to the two-dimensional transverse electric (TE) case for a perfectly reflecting,
locally rough surface by assuming that the local perturbation is invariant in the $x_3$ direction.
Denote by $\G:=\{(x_1,x_2)\;:\;x_2=h_\G(x_1),x_1\in\R\}$ the locally rough surface with a smooth surface profile
function $h_\G\in C^2(\R)$ having a compact support in $\R$, and by $D_+$ the unbounded domain above $\G$.
Let $\Sp^1_\pm:=\{x=(x_1,x_2):|x|=1,x_1\gtrless0\}$ be the upper and lower parts of the unit circle
$\Sp^1:=\{x=(x_1,x_2):|x|=1\}$. Suppose a time-harmonic ($e^{-i\omega t}$ time dependence) plane wave
\ben
u^i=u^i(x;d,k):=\exp({ikd\cdot x})
\enn
is incident on the locally rough surface from the upper unbounded domain $D_+$,
where $d=(\sin\theta,-\cos\theta)^T\in\Sp^1_-$ is the incident direction with $\theta\in(-\pi/2,\pi/2)$
the angle of incidence, $k=\omega/c>0$ is the wave number, $\omega$ and $c$ are the wave frequency and
speed in $D_+$, respectively.
Then the total field $u(x)$ is given as the sum of the incident wave $u^i(x)$,
the reflected wave $u^r(x)$ and the unknown scattered wave $u^s(x)$. Here,
$u^r$ is the reflected wave with respect to the infinite plane $x_2=0$ given by
\ben
u^r=u^r(x;d,k):=-\exp({ikd'\cdot x})
\enn
where $d'=(\sin\theta,\cos\theta)\in\Sp^1_+$ is the reflected direction.
Furthermore, the scattered field $u^s$ is required to satisfy the Helmholtz equation in $D_+$,
the boundary condition on $\Gamma$ and the so-called Sommerfeld radiation condition at infinity, respectively:
\begin{align}\label{eq1_nr}
\Delta u^s+k^2 u^s=0&\quad \textrm{in}\;\;D_+\\
\label{eq2_nr}u^s=f&\quad \textrm{on}\;\;\G \\
\label{eq3_nr}\lim_{r\to\infty}r^\half\left(\frac{\pa u^s}{\pa r}-ik u^s\right)=0&\quad r=|x|
\end{align}
where $f=-(u^i+u^r)$ has a compact support on $\Gamma$. Here, $u^{near}:=u^r+u^s$ is called the near field.
In addition, (\ref{eq3_nr}) implies that $u^s$ has the following asymptotic behavior
(see \cite{Willers1987,ZhangZhang2013}):
\be\label{eq4}
u^s(x;d,k)=\frac{e^{ik|x|}}{\sqrt{|x|}}\left(u^\infty(\hat{x};d,k)+O\Big(\frac{1}{|x|}\Big)\right),\qquad |x|\to\infty,
\en
uniformly for all observation directions $\hat{x}=x/|x|\in\Sp^1_+$, where $u^\infty$ is
called the far-field pattern of the scattered field $u^s$.
The geometry of the scattering problem is presented in Figure \ref{fig1}.

\begin{figure}\label{fig1}
  \centering
    \includegraphics[width=4in]{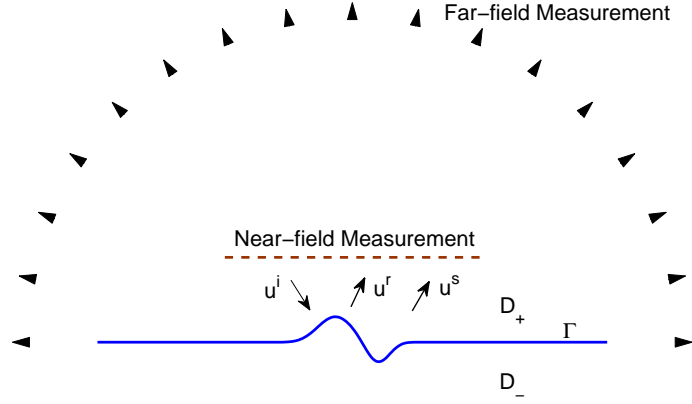}
    \vspace{-0.4in}
 \caption{The scattering problem from a locally rough surface}
\end{figure}

The existence and uniqueness of solutions to the scattering problem (\ref{eq1_nr})-(\ref{eq3_nr})
has been studied by using the integral equation method in \cite{Willers1987} and
a variational method in \cite{BaoLin11}. Recently in \cite{ZhangZhang2013},
a novel integral equation formulation was proposed for this scattering problem,
which leads to a fast numerical solution of the scattering problem including the case with a large wave number.
It should be noted that, in the past years, the mathematical and computational aspects of the
scattering problem (\ref{eq1_nr})-(\ref{eq3_nr}) and other boundary conditions
have been studied extensively by using the integral equation method and variational approaches
for the case when the surface $\G$ is a non-local (or global) perturbation of the infinite plane $x_2=0$
(which is called the rough surface scattering in the engineering community)
(see, e.g. \cite{CZ98,CRZ99,CM05,CHP06,CE10,SS01,V99,WC01,ZC03}).

On the other hand, many reconstruction algorithms have been developed for the inverse problem of reconstructing
locally rough surfaces from the scattered near-field or far-field data, corresponding to incident
point sources or plane waves (see, e.g. \cite{AKY,BaoLin11,BaoLin13,CGHIR,DPT06,KressTran2000,LSZ16}).
For example, a Newton method was proposed in \cite{KressTran2000} to reconstruct a locally rough surface
from the far-field pattern under the condition that the local perturbation is both star-like
and {\em above} the infinite plane.
An optimization method was introduced in \cite{BaoLin13} to recover a mild, locally rough surface
from the scattered field measured on a straight line within one wavelength above the locally
rough surface, under the assumption that the local perturbation is {\em above} the infinite plane.
In \cite{BaoLin11}, a continuation approach over the wave frequency was developed for
reconstructing a general, locally rough surface from the scattered field measured on an upper
half-circle enclosing the local perturbation, based on the choice of the descent vector field.
A regularized Newton method was proposed in \cite{ZhangZhang2013} to reconstruct a general, locally rough surface
from multi-frequency far-field data, where the novel integral equation introduced in \cite{ZhangZhang2013}
is used to solve the forward scattering problem in each iteration.
The reconstruction results obtained in \cite{BaoLin11,ZhangZhang2013} are stable and accurate even for
multi-scale surface profiles in view of using multiple frequency data and considering multiple scattering.
We point out that many reconstruction algorithms have also been developed for reconstructing non-locally rough surfaces
from the scattered near-field data (see, e.g. \cite{BaoLi13,BaoLi14,BP2010,BP2009,CL05,DeSanto1,DeSanto2,LS15}).

In practical applications, it is much harder to obtain data with accurate phase information
compared with just measuring the intensity (or the modulus) of the data, and therefore it is often desirable to
reconstruct the scattering surface profile from the phaseless near-field or far-field data.
However, not many results are available for such problems both mathematically and numerically.
Kress and Rundell first studied such inverse problems in \cite{KR97} and proved that
for the sound-soft bounded obstacle case with one incidence plane wave, the modulus of the far-field pattern
is invariant under translations of the obstacle, and therefore it is impossible to reconstruct the location of the
obstacle from the phaseless far-field data for one incident plane wave.
It was further proved in \cite{KR97} that this ambiguity cannot be remedied by using the phaseless
far-field pattern for finitely many incident plane waves with different wave numbers or different incident directions.
Regularized Newton and Landweber iteration methods have also been discussed in \cite{KR97} for recovering the
shape of the obstacle from the phaseless far-field data.
In \cite{Ivanyshyn07,IvanyshynKress2010}, a nonlinear integral equation method was proposed
to reconstruct the shape of the obstacle from the phaseless far-field data.
Further, in \cite{IvanyshynKress2010} after the shape of the obstacle is reconstructed from the phaseless far-field data,
an algorithm is proposed for the localization of the obstacle by utilizing the translation invariance property
together with several full far-field measurements at the backscattering direction.
In \cite{IvanyshynKress2011}, a nonlinear integral equation method was developed to reconstruct the real-valued surface
impedance function from the phaseless far-field data provided that the bounded obstacle is known in advance.
Recently, a continuation algorithm was proposed in \cite{BaoLiLv13} to reconstruct the shape of a perfectly reflecting
periodic surface from the phaseless near-field data, in \cite{BLT11} to deal with the phaseless measurements for
an inverse source problem, and in \cite{BaoZhang16} to recover the shape of multi-scale sound-soft large rough surfaces
from phaseless measurements of the scattered field generated by tapered waves with multiple frequencies.
Recently, for inverse acoustic scattering with bounded obstacles it was proved in \cite{ZZ16}
that the translation invariance property of the phaseless far-field pattern can be broken by using
superpositions of two plane waves as the incident fields in conjunction with all wave numbers in a finite interval.
Further, a recursive Newton-type iteration algorithm in frequencies was developed in \cite{ZZ16} to numerically
reconstruct both the location and the shape of the obstacle simultaneously from multi-frequency phaseless far-field data.

The purpose of this paper is to develop an efficient imaging algorithm to reconstruct the locally rough surface $\G$
from phaseless data associated with incident plane waves. Two types of phaseless data will be considered:
the phaseless far-field data and the phaseless near-field data (see Figure \ref{fig1}).
Similarly as in the bounded obstacle case, the phaseless far-field pattern, $|u^\infty(\hat{x};d,k)|$,
is also invariant under translations of the local perturbation along the $x_1$ direction for one incident plane wave,
that is, $|u^\infty_\ell(\hat{x};d,k)|=|u^\infty(\hat{x};d,k)|$, $\hat{x}\in\Sp^1_+$ for all $\ell\in\R,$
where $u^\infty_\ell(\hat{x};d,k)$ is the far-field pattern of the scattering solution with
respect to the shifted surface $\G^{\ell}:=\{(x_1+\ell,h_\G(x_1)):x_1\in\R\}$ of $\G$ along the $x_1$ direction
(see Theorem \ref{th2} below).
Thus, it is impossible to recover the location of the locally rough surface $\G$ from phaseless far-field data
corresponding to one incident plane wave. To overcome this difficulty, motivated by \cite{ZZ16},
we will use the following superposition of two plane waves rather than one plane wave as the incident field:
\be\label{IW}
u^i=u^i(x;d_1,d_2,k):=\exp({ikd_1\cdot x})+\exp({ikd_2\cdot x})
\en
where, for $l=1,2$, $\theta_l\in(-\pi/2,\pi/2)$ is the incidence angle and
$d_l=(\sin\theta_l,-\cos\theta_l)^T\in\Sp^1_-$ is the incident direction.
Then the reflected wave with respect to the infinite plane $x_2=0$ will be given by
\ben
u^r=u^r(x;d_1,d_2,k):=-\exp({ikd_1'\cdot x})-\exp({ikd_2'\cdot x})
\enn
where, for $l=1,2$, $d'_l=(\sin\theta_l,\cos\theta_l)\in\Sp^1_+$ is the reflected direction,
and the scattered field $u^s$ will have the asymptotic behavior
\be\label{asymp-2}
u^s(x;d_1,d_2,k)=\frac{e^{ik|x|}}{\sqrt{|x|}}\left(u^\infty(\hat{x};d_1,d_2,k)
+O\Big(\frac{1}{|x|}\Big)\right),\qquad |x|\to\infty,
\en
uniformly for all observation directions $\hat{x}=x/|x|\in\Sp^1_+$.

We will prove that, if the incident field is taken as $u^i=u^i(x;d_1,d_2,k)$ with
$d_1\not=d_2$ and all wave numbers $k$ in a finite interval,
then the translation invariance property of the phaseless far-field pattern does not hold
for non-trivial locally rough surfaces (that is, $h_\G\not\equiv0$) (see Theorem \ref{th3} below).
Thus, both the location and the shape of the local perturbation of the surface $\G$ can be reconstructed from
the phaseless far-field data, corresponding to such incident fields with multiple wave numbers
(see the numerical experiments in Section \ref{se4}).
Furthermore, a recursive Newton iteration algorithm in frequencies is developed to reconstruct both
the location and the shape of the surface $\G$ simultaneously from multi-frequency phaseless far-field data.
A similar Newton iteration algorithm is also developed for reconstructing the location and shape of
the surface $\G$ from multi-frequency phaseless near-field data.
In our reconstruction algorithms the fast integral equation solver developed in \cite{ZhangZhang2013}
is used to solve the forward scattering problem in each iteration.

This paper is organized as follows. Section \ref{se1} gives a brief introduction to the integral equation
formulation of the forward problem proposed in \cite{ZhangZhang2013}, and
the inverse scattering problem is studied in Section \ref{se2} with phaseless far-field and near-field data.
In Section \ref{se3}, a recursive Newton-type iteration algorithm in frequencies is proposed to solve the
inverse problems. Numerical examples are carried out in Section \ref{se4} to illustrate the effectiveness of
our inversion algorithm. Concluding remarks are presented in Section \ref{se5}.

\section{The integral equation formulation for the scattering problem}\label{se1}
\setcounter{equation}{0}

In this section we give a brief introduction to the integral equation formulation
proposed in \cite{ZhangZhang2013} for the scattering problem (\ref{eq1_nr})-(\ref{eq3_nr})
which leads to a fast numerical solution of the problem and will be used in
our inversion algorithm. To this end, we need the following notations.

Let $B_R\coloneqq\{x=(x_1,x_2)\;|\;|x|<R\}$ be a circle with $R>0$ large enough so that
the local perturbation $\{(x_1,h_{\G}(x_1))\;|\;x_1\in\textrm{supp}(h_{\G})\}\subset B_R$.
Then $\G_R\coloneqq\G\cap B_R$ represents the part of $\G$ containing the local perturbation of the infinite plane.
Denote by $x_A:=(-R,0),x_B:=(R,0)$ the endpoints of $\G_R$.
Write $\mathbb{R}^2_\pm\coloneqq\{(x_1,x_2)\in\R^2\;|\;x_2\gtrless0\}$,
$D^\pm_R\coloneqq B_R\cap D_\pm$ and $\pa B^\pm_R\coloneqq\pa B_R\cap D_\pm$,
where $D_-:=\{(x_1,x_2)\;|\;x_2<h_\G(x_1),x_1\in\R\}.$ See Figure \ref{fig2} for the problem geometry.
\begin{figure}\label{fig2}
  \centering
    \includegraphics[width=4in]{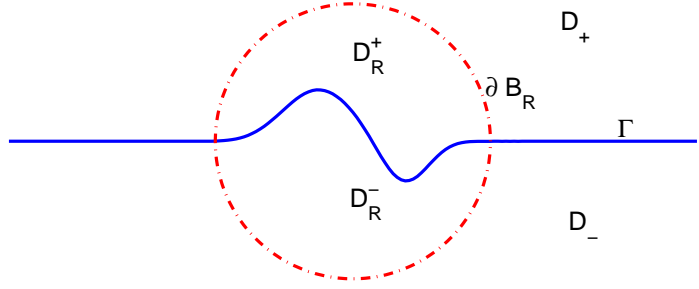}
    \vspace{-0.4in}
 \caption{Problem geometry}
\end{figure}

For $\varphi\in C(\pa D_R^-)$,
define $S_k,S_k^{re},K_k,K_k^{re}$ to be the boundary integral operators of the following form:
\ben
(S_k\varphi)(x)&:=&\int_{\pa D^-_R}\Phi_k(x,y)\varphi(y)ds(y),\quad x\in\pa D^-_R\\
(S_k^{re}\varphi)(x)&:=&
     \left\{\begin{array}{l}
     \ds\int_{\pa D^-_R}\Phi_k(x,y)\varphi(y)ds(y),\quad x\in\Gamma_R\\
     \ds\int_{\pa D^-_R}\Phi_k(x^{re},y)\varphi(y)ds(y),\quad x\in\pa B^-_R\cup\{x_A,x_B\}
     \end{array}\right.\\
(K_k\varphi)(x)&:=&\int_{\pa D^-_R}\frac{\pa\Phi_k(x,y)}{\pa\nu(y)}\varphi(y)ds(y),
     \quad x\in\pa D^-_R\\
(K_k^{re}\varphi)(x)&:=&
\left\{\begin{array}{l}\ds\int_{\pa D^-_R}\frac{\pa\Phi_k(x,y)}{\pa\nu(y)}\varphi(y)ds(y),
    \quad x\in\Gamma_R\\
    \ds\int_{\pa D^-_R}\frac{\pa\Phi_k(x^{re},y)}{\pa\nu(y)}\varphi(y)ds(y),
    \quad x\in\pa B^-_R\cup\{x_A,x_B\}
    \end{array}\right.
\enn
where $x^{re}=(x_1,-x_2)$ is the reflection of $x=(x_1,x_2)$ about the $x_1$-axis,
$\Phi_k(x,y)$ is the fundamental solution of the Helmholtz equation $\Delta w+k^2w=0$
with the wavenumber $k$, and $\nu$ is the unit outward normal on $\pa D^-_R$.

It was proved in \cite{ZhangZhang2013} that the scattering solution $u^s$ of the scattering
problem (\ref{eq1_nr})-(\ref{eq3_nr}) is given as follows
\be\label{eq10_nr}
u^s(x)=\int_{\pa D^-_R}\left[\frac{\pa\Phi_k(x,y)}{\pa\nu(y)}-i\eta\Phi_k(x,y)\right]\varphi(y)ds(y),\quad x\in\pa D^-_R,
\en
with $\eta\neq0$ a real coupling parameter, if and only if $\varphi\in C(\pa D_R^-)$ is the solution of
the integral equation $P\varphi(x)=g(x)$. Here,
\be\label{eq11_nr}
P\varphi\coloneqq\left\{\begin{array}{ll}
\ds\varphi+\left(K_k\varphi-i\eta S_k\varphi\right)
   +\left(K^{re}_k\varphi-i\eta S_k^{re}\varphi\right), &x\in\G_R\\
\ds\half\varphi+\left(K_k\varphi-i\eta S_k\varphi\right)
   +\left(K^{re}_k\varphi-i\eta S^{re}_k\varphi\right), &x\in\pa B_R^-\cup\{x_A,x_B\}
\end{array}\right.
\en
and
\be\label{eq12_nr}
g(x):=\left\{\begin{array}{ll}
-2\big(u^i(x)+u^r(x)\big), &x\in\G_R\\
0, &x\in\pa B_R^-\cup\{x_A,x_B\}
\end{array}\right.
\en

In \cite[theorem 2.4]{ZhangZhang2013}, it was shown that the integral equation $P\varphi=g$ is uniquely solvable
in $C(\pa D_R^-)$, which is solved numerically by using the Nystr\"{o}m method with a graded mesh
introduced in \cite[Section 3.5]{ColtonKress1998} (see \cite{ZhangZhang2013} for details).

\begin{remark}\label{re3}{\rm
By (\ref{eq10_nr}) and the asymptotic behavior of the fundamental solution $\Phi_k$, it follows that the
far-field pattern of the scattered field $u^s$ is given by
\be\label{eq14_nr}
u^\infty(\hat{x};d_1,d_2,k)=\frac{e^{-i\pi/4}}{\sqrt{8\pi k}}\int_{\pa D_R^-}
[k\nu(y)\cdot\hat{x}+\eta]e^{-ik\hat{x}\cdot y}\varphi(y)ds(y),\qquad \hat{x}\in\Sp^1_+,
\en
which is an analytic function on the unit circle $\Sp^1_+$, where $\varphi$ is the solution of the
integral equation $P\varphi=g$.
}
\end{remark}

\section{The inverse problems}\label{se2}
\setcounter{equation}{0}

In this paper we consider two types of scattered field measurement data without phase information:
the phaseless far-field data and the phaseless near-field data (see Figure \ref{fig1}).
The far-field pattern $u^\infty$ is already defined in (\ref{eq4}), while the near-field $u^{near}$
is defined as $u^{near}(x;d,k)=u^r(x;d,k)+u^s(x;d,k)$, $x\in\G_{H,L}$ for positive constants $H$ and $L$,
where $\G_{H,L}:=\{(x_1,H),x_1\in[-L,L]\}$ is the measurement straight line segment above the surface $\G$.

We first study properties of the phaseless far-field pattern and the phaseless near-field under
translations of the locally rough surface. To this end, for $\ell\in\R$ define
$\G^{\ell}:=\{(x_1+\ell,h_\G(x_1)):x_1\in\R\}$ to be the shifted surface of $\G$ along the $x_1$ direction.
Then we have the following theorem.

\begin{theorem}\label{th2}
Given the wavenumber $k>0$, let the incident wave be given by $u^i=u^i(x;d,k)$ with the incident direction
$d=(\sin\theta,-\cos\theta)$ and the incident angle $\theta\in(-\pi/2,\pi/2)$ and let $u^s(x;d,k),u^s_\ell(x;d,k)$
be the scattering solutions of the scattering problem $(\ref{eq1_nr})-(\ref{eq3_nr})$, corresponding to the locally
rough surface $\G$ and the shifted one $\G^\ell$, respectively.
Assume that $u^\infty(\hat{x};d,k),u^\infty_\ell(\hat{x};d,k)$ ($u^{near}(x;d,k),u^{near}_\ell(x;d,k)$) are the
far-field pattern (the near-field) of the scattering solutions $u^s,u^s_\ell$, respectively.
Then we have $u^s_\ell(x^{\ell};d,k)=e^{ik\ell\sin\theta}u^s(x;d,k)$,
$u^{near}_\ell(x^{\ell};d,k)=e^{ik\ell\sin\theta}u^{near}(x;d,k)$, $x\in D_+$
and $u^\infty_\ell(\hat{x};d,k)=e^{ik\ell(\sin\theta-\hat{x}_1)}u^\infty(\hat{x};d,k)$, $\hat{x}\in\Sp^1_+$
where $x^{\ell}:=x+(\ell,0)^T$ for $x\in\R^2$ and $\hat{x}_1$ is the first component of $\hat{x}$.
\end{theorem}

\begin{proof}
Assume that $D^\ell_+$ is the unbounded domain above $\G^\ell$.
Let $v(x):=e^{ik\ell\sin\theta}u^s(x^{-\ell};d,k)$ for $x\in D^\ell_+$.
Then it is easily seen that $v$ is well defined in $D^\ell_+$.
Thus, by the properties of the scattered field $u^s$ it follows that $v$ satisfies the Helmholtz equation (\ref{eq1_nr})
in $D^\ell_+$ and the Sommerfeld radiation condition (\ref{eq3_nr}).
Further, it can be seen that
\ben
u^i(x,d,k)=e^{ik\ell\sin\theta}u^i(x^{-\ell},d,k),\quad
u^r(x,d,k)=e^{ik\ell\sin\theta}u^r(x^{-\ell},d,k)
\enn
Then from the boundary condition (\ref{eq2_nr}) on $u^s$, we have that $u^i(x,d,k)+u^r(x,d,k)+v(x)=0$ for $x\in\G^\ell$.
Now, the uniqueness result of the scattering problem (\ref{eq1_nr})-(\ref{eq3_nr}) implies that
$v(x)=u^s_\ell(x;d,k)$ for $x\in D^\ell_+$.
Therefore, we obtain that $u^s_\ell(x^{\ell};d,k)=e^{ik\ell\sin\theta}u^s(x;d,k)$,
$u^{near}_\ell(x^{\ell};d,k)=e^{ik\ell\sin\theta}u^{near}(x;d,k)$ for $x\in D_+$.

Finally, from the asymptotic behavior (\ref{eq4}) of the scattered field and the fact that for $a\in\R^2$,
\ben
|x-a|-|x|=-a\cdot\hat{x}+O\left(\frac{1}{|x|}\right),\quad |x|\rightarrow\infty
\enn
uniformly for all directions $\hat{x}=x/|x|\in\Sp^1$, it is derived that
$u^\infty_\ell(\hat{x};d,k)=e^{ik\ell(\sin\theta-\hat{x}_1)}u^\infty(\hat{x};d,k)$ for $\hat{x}\in\Sp^1_+$.
The proof is thus completed.
\end{proof}

Theorem \ref{th2} indicates that the modulus of the far field pattern (or the phaseless far-field pattern)
for one incident plane wave is invariant under translations along the $x_1$ direction of the boundary,
that is, $|u^\infty_\ell(\hat{x};d,k)|=|u^\infty(\hat{x};d,k)|$, $\hat{x}\in\Sp^1_+$ for all $\ell\in\R.$
This means that the location of the local perturbation on the boundary $\G$ can not be determined from
the phaseless far-field pattern for one incident plane wave, as shown in Example \ref{fig4-1}.
In the next theorem we prove that, if the locally rough surface is non-trivial (that is, $x_2\not=0$)
then the translation invariance property of the phaseless far-field pattern only holds for a countably infinite
number of real numbers $\ell$ in the case when the incident field is taken as a superposition of two plane waves
with different directions, that is, $u^i=u^i(x;d_1,d_2,k)$ with $d_1\not=d_2$.

\begin{theorem}\label{th3}
Let the incident wave be given by $u^i=u^i(x;d_1,d_2,k)$ with $d_j=(\sin\theta_j,-\cos\theta_j),$
$\theta_j\in(-\pi/2,\pi/2),\;j=1,2$ and $\theta_1\neq\theta_2$.
Assume that $u^s(x;d_1,d_2,k),u^s_l(x;d_1,d_2,k)$ are the scattering solutions of the problem
$(\ref{eq1_nr})-(\ref{eq3_nr})$, corresponding to the locally rough surface $\G$ and the shifted one $\G^\ell$,
respectively, with $u^\infty(\hat{x};d_1,d_2,k),u^\infty_\ell(\hat{x};d_1,d_2,k)$ the corresponding far-field
patterns.
Assume further that the locally rough surface $\G$ is non-trivial (that is, $x_2\not=0$ or $h_\G\not\equiv0$).
Then we have
\be\label{TI}
|u^\infty(\hat{x};d_1,d_2,k)|=|u^\infty_\ell(\hat{x};d_1,d_2,k)|,\;\;\hat{x}\in\Sp^1_+
\en
for all $\ell=\ell_n:=2\pi n/[k(\sin\theta_1-\sin\theta_2)]$ with any $n\in\Z$.
Further, except for $\ell_n$, there may exist at most one real constant $\tau$ with $0<\tau<2\pi$ such that
(\ref{TI}) holds for $\ell=\ell_{n\tau}:=(2\pi n+\tau)/[k(\sin\theta_1-\sin\theta_2)]$ with any $n\in\Z$.
\end{theorem}

\begin{proof}
Let $v_j(\hat{x}):=u^\infty(\hat{x};d_j,k)$, where $u^\infty(\hat{x};d_j,k)$ is the far-field pattern of the
solution $u^s(x;d_j,k)$ to the scattering problem (\ref{eq1_nr})-(\ref{eq3_nr}), corresponding to the incident field
$u^i=u^i(x;d_j,k)$, $j=1,2.$ From the linearity of the scattering problem and the definition of $u^i({x};d_1,d_2,k)$,
we have that $u^\infty(\hat{x};d_1,d_2,k)=v_1(\hat{x})+v_2(\hat{x})$.
Further, by Theorem \ref{th2} it follows that
$$
u^\infty_\ell(\hat{x};d_1,d_2,k)=e^{ik\ell(\sin\theta_1-\hat{x}_1)}v_1(\hat{x})
+e^{ik\ell(\sin\theta_2-\hat{x}_1)}v_2(\hat{x}).
$$
Then (\ref{TI}) becomes
\ben
|v_1(\hat{x})+v_2(\hat{x})|=|e^{ik\ell(\sin\theta_1-\hat{x}_1)}v_1(\hat{x})
+e^{ik\ell(\sin\theta_2-\hat{x}_1)}v_2(\hat{x})|,\quad\hat{x}\in\Sp^1_+
\enn
for $\ell\in\R$. By a direct calculation, the above equation is reduced to
\be\label{c5eq6}
\Rt\left(v_1(\hat{x})\ol{v_2(\hat{x})}\right)=
\Rt\left(e^{ik\ell(\sin\theta_1-\sin\theta_2)}v_1(\hat{x})\ol{v_2(\hat{x})}\right),\quad\hat{x}\in\Sp^1_+
\en
for $\ell\in\R$.

We can claim that $v_1(\hat{x})\ol{v_2(\hat{x})}\not\equiv 0,\;\;\hat{x}\in\Sp^1_+.$
In fact, if this is not true, then we have
\be\label{c5eq7}
v_1(\hat{x})\ol{v_2(\hat{x})}=0,\quad\hat{x}\in\Sp^1_+
\en
We now distinguish between the following two cases.

{\bf Case 1.} $v_1\equiv0$ on $\Sp^1_+$.
Since $v_1(\hat{x})$ is the far-field pattern of $u^s(x;d_1,k)$, then, by \cite[Lemma 3.2]{Ma05}
we have that $u^s(x;d_1,k)\equiv0$ in $D_+$. Further, from the boundary condition (\ref{eq2_nr})
for $u^s(x;d_1,k)$ it follows that $u^i(x;d_1,k)+u^r(x;d_1,k)\equiv0$ on $\G$. This implies that
$\exp(2ikh_\G(x_1)\cos\theta_1)\equiv 1$ for all $x_1\in\R$. Thus, $kh_\G(x_1)\cos\theta_1\equiv n\pi$
for all $x_1\in\R$, where $n=0,\pm1,\pm2,\cdots$, so $h_\G\equiv 0$. This is a contradiction.

{\bf Case 2.} There exists an $\hat{x}_0\in\Sp^1_+$ such that $v_1(\hat{x}_0)\neq0$.
Then we have $v_1\neq0$ in a neighborhood of $x_0$, which, together with (\ref{c5eq7}), implies that
$v_2=0$ in a neighborhood of $x_0$. Since $v_2$ is an analytic function on $\Sp^1_+$ (see Remark \ref{re3}),
we have $v_2\equiv0$ on $\Sp^1_+$. Arguing similarly as in Case 1 gives that $h_\G\equiv 0$,
contradicting to the assumption of the theorem.

Now, by (\ref{c5eq6}) it is easy to see that (\ref{c5eq6}) or equivalently (\ref{TI}) holds if $\ell$
satisfies the condition
\be\label{TI-C1}
k\ell(\sin\theta_1-\sin\theta_2)=2\pi n,\qquad n\in\Z,
\en
or if $\ell=2\pi n/[k(\sin\theta_1-\sin\theta_2)]$ with any $n\in\Z$.
Further, assume that $v_1(\hat{x}_0)\ov{v_2(\hat{x}_0)}\not=0$ for some $\hat{x}_0\in\Sp^1_+$ and write $v_1(\hat{x}_0)\ov{v_2(\hat{x}_0)}=|v_1(\hat{x}_0)\ov{v_2(\hat{x}_0)}|\exp(i\theta(\hat{x}_0))$
with $0<\theta(\hat{x}_0)<\pi$. Then (\ref{c5eq6}) is reduced to the equation
\be\label{TI-C1a}
\cos[k\ell(\sin\theta_1-\sin\theta_2)+\theta(\hat{x}_0)]-\cos[\theta(\hat{x}_0)]=0.
\en
By (\ref{TI-C1a}) we know that, except for the above $\ell$ satisfying the condition (\ref{TI-C1}),
(\ref{c5eq6}) or equivalently (\ref{TI}) also holds for $\ell$ satisfying the condition
\be\label{TI-C1+}
k\ell(\sin\theta_1-\sin\theta_2)=2\pi n-2\theta(\hat{x}_0),\qquad n\in\Z,
\en
or equivalently for
\be\label{TI-C1b}
\ell=\frac{2\pi n-2\theta(\hat{x}_0)}{k(\sin\theta_1-\sin\theta_2)},\;\; n\in\Z.
\en
Thus, except for $\ell=\ell_n$, there may exist at most one real number $\tau$ with $0<\tau<2\pi$ such that
(\ref{c5eq6}) or equivalently (\ref{TI}) holds for $\ell=(2\pi n+\tau)/[k(\sin\theta_1-\sin\theta_2)]$
with any $n\in\Z$. In fact, by (\ref{TI-C1b}) we have $\tau=2\pi-2\theta(\hat{x}_0)$.
The proof is thus complete.
\end{proof}

Theorem \ref{th3} indicates that the translation invariance property of the phaseless far-field pattern can be broken
for non-trivial locally rough surfaces if the incident field is taken as $u^i=u^i(x;d_1,d_2,k)$ with
$d_1\not=d_2$ and all wave numbers $k\in[k_1,k_N]$ for some wave numbers $k_N>k_1>0.$
Then it is expected that both the location and the shape of the local perturbation part of the boundary $\G$ can be
reconstructed simultaneously from the phaseless far-field data, corresponding to such incident fields with multiple
wave numbers, as demonstrated in the numerical experiments.

On the other hand, from Theorem \ref{th2} and Figure \ref{fig4-7} it is expected that the translation invariance
property of the phaseless near-field data measured on the line segment $\G_{H,L}$ does not hold even for one incident
plane wave though we can not prove this rigorously.
Thus, both the location and the shape of the local perturbation part of the boundary $\G$ can also be
reconstructed from the phaseless near-field data, corresponding to one incident plane wave (see the numerical examples).

Based on the above discussions, we consider the following two inverse problems.

{\bf Inverse problem (IP1):} Given the incident fields $u^i=u^i(x;d_1,d_2,k)$ with multiple wave numbers
$k=k_1,\cdots,k_N$, where $d_1,d_2\in\Sp^1_-$ and $d_1\neq d_2$, to reconstruct the locally rough surface $\G$
from the corresponding intensity-only far-field data $|u^\infty(\hat{x};d_1,d_2,k)|^2$, $\hat{x}\in\Sp^1_+$,
$k=k_1,\cdots,k_N$.

{\bf Inverse problem (IP2):} Given the wave number $k>0$ and the incident field $u^i=u^i(x;d,k)$
with $d\in\Sp^1_-$, to reconstruct the locally rough surface $\G$ from the corresponding intensity-only
near-field data $|u^{near}(x;d,k)|^2$, $x\in\G_{H,L}$.

In the next section we develop a recursive Newton-type iteration algorithm in frequencies for solving
the inverse problems (IP1) and (IP2).
To this end, given the incident wave $u^i=u^i(x;d_1,d_2,k)$ with $d_1,d_2\in\Sp^1_-$ and $d_1\neq d_2$,
define the far-field operator $\mathcal{F}_{(d_1,d_2,k)}$ mapping the function $h_\G$ which describes the locally
rough surface $\G$ to the intensity of the corresponding far-field pattern,
$|u^\infty(\hat{x};d_1,d_2,k)|^2$ in $L^2(\Sp^1_+)$ of the scattered wave $u^s(x;d_1,d_2,k)$ of the
scattering problem (\ref{eq1_nr})-(\ref{eq3_nr}), that is,
\be\label{FFO}
\big(\mathcal{F}_{(d_1,d_2,k)}[h_\G]\big)(\hat{x})=|u^\infty(\hat{x};d_1,d_2,k)|^2,\quad
\hat{x}\in\Sp^1_+
\en
Similarly, given the incident wave $u^i=u^i(x;d,k)$ with $d\in\Sp^1_-$, define the near-field operator
$\mathcal{N}_{(d,k)}$ mapping the function $h_\G$ to the intensity of the corresponding near-field,
$|u^{near}(x;d,k)|^2$ in $L^2(\G_{H,L})$, of the scattering problem (\ref{eq1_nr})-(\ref{eq3_nr}), that is,
\be\label{eq5}
\big(\mathcal{N}_{(d,k)}[h_\G]\big)(x)=|u^{near}(x;d,k)|^2,\quad x\in\G_{H,L}
\en

Our Newton-type iterative algorithm consists in solving the nonlinear and ill-posed equation (\ref{FFO})
or (\ref{eq5}) for the unknown $h_\G$. To this end, we need to investigate the Frechet differentiability of
$\mathcal{F}_{(d_1,d_2,k)}$ and $\mathcal{N}_{(d,k)}$ at $h_\G.$
Now, let $\triangle h\in C^2_{0,R}(\R)\coloneqq\{h\in C^2(\R):\textrm{supp}(h)\subset(-R,R)\}$ be a small
perturbation and let $\G_{\triangle h}\coloneqq \{(x_1,h_\G(x)+\triangle h(x)):x_1\in\R\}$ denote the
corresponding boundary for $h_\G(x)+\triangle h(x)$.
Then $\mathcal{F}_{(d_1,d_2,k)}$ is said to be Frechet differentiable at $h_\G$ if there exists a linear
bounded operator $\mathcal{F}'_{(d_1,d_2,k)}:C^2_{0,R}(\R)\rightarrow L^2(\Sp^1_+)$ such that
\ben
\left\|\mathcal{F}_{(d_1,d_2,k)}[h_\G+\triangle h]-\mathcal{F}_{(d_1,d_2,k)}[h_\G]
-\mathcal{F}'_{(d_1,d_2,k)}[h_\G;\triangle h]\right\|_{L^2(\Sp^1_+)}=o\left(||\triangle h||_{C^2(\R)}\right)
\enn
The Frechet differentiable at $h_\G$ of $\mathcal{N}_{(d,k)}$ is defined similarly.
We have the following theorem.

\begin{theorem}\label{th1_nr}
(i) Given the incident wave $u^i=u^i(x;d_1,d_2,k)$ with $d_1,d_2\in\Sp^1_-$ and $d_1\neq d_2$,
let $u(x)=u^i(x)+u^r(x)+u^s(x)$, where $u^s$ solves the scattering problem $(\ref{eq1_nr})-(\ref{eq3_nr})$
with the boundary data $f=-(u^i+u^r)$. If $h_\G\in C^2(\R)$, then $\mathcal{F}_{(d_1,d_2,k)}$ is Frechet
differentiable at $h_\G$ with the derivative given by
$\mathcal{F}'_{(d_1,d_2,k)}[h_\G;\triangle h]=2\Rt\left[\ov{u^\infty}(u')^\infty\right]$ for
$\triangle h\in C^2_{0,R}(\R)$.
Here, $(u')^\infty$ is the far-field pattern of $u'$ which solves the scattering problem
$(\ref{eq1_nr})-(\ref{eq3_nr})$ with the boundary data $f=-(\triangle h\cdot\nu_2){\pa u}/{\pa\nu}$,
where $\nu_2$ is the second component of the unit normal $\nu$ on $\G$ directed into the infinite domain $D_+$.

(ii) Given the incident wave $u^i=u^i(x;d,k)$ with $d\in\Sp^1_-$, let $u$ and $u^{near}$ be
the total and near-field, respectively, corresponding to the scattering problem $(\ref{eq1_nr})-(\ref{eq3_nr})$.
If $h_\G\in C^2(\R)$, then $\mathcal{N}_{(d,k)}$ is Frechet differentiable at $h_\G$ with the derivative
given by $\mathcal{N}'_{(d,k)}[h_\G;\triangle h]=2\Rt\left[\ov{u^{near}}u'\right]$
for $\triangle h\in C^2_{0,R}(\R)$, where $u'$ is the solution of the problem ($(\ref{eq1_nr})-(\ref{eq3_nr})$
with the boundary data $f=-(\triangle h\cdot\nu_2){\pa u}/{\pa\nu}$.
\end{theorem}

\begin{proof}
The proof is similar to that of Theorems 2.1 and 2.2 in \cite{BaoLiLv13} for
inverse diffraction grating problems with appropriate modifications.
\end{proof}

\section{Reconstruction algorithm}\label{se3}
\setcounter{equation}{0}

In this section, we describe the Newton-type iteration algorithm for the inverse problem (IP1).
For the inverse problem (IP2), the approach is similar, so we omit it.

Let $d_{1l},d_{2l}\in\Sp^1_-,l=1,2,\ldots,n_d,$ be the incident directions and let $k>0$
be the fixed wave number. Assume that $h^{app}$ is an approximation to the function $h_\G$.
We replace (\ref{FFO}) by the linearized equations:
\be\label{LFFO}
\big(\mathcal{F}_{(d_{1l},d_{2l},k)}[h^{app}]\big)(\hat{x})
+\big(\mathcal{F}'_{(d_{1l},d_{2l},k)}[h^{app};\triangle h]\big)(\hat{x})
\approx|u^\infty(\hat{x};d_{1l},d_{2l},k)|^2,\quad l=1,2,\ldots,n_d,
\en
where $\triangle h$ is the update function to be determined. Our Newton iterative algorithm consists in
iterating the equations (\ref{LFFO}) by using the Levenberg-Marquardt algorithm (see, e.g. \cite{Hohage1999}).

In the numerical examples, the noisy phaseless far-field pattern
$|u^\infty_{\delta}(\hat{x_j};d_{1l},d_{2l},k)|,j=1,2,\ldots,n_f,l=1,2,\ldots,n_d,$
are considered as the measurement data which satisfies
\ben
\big\||u^\infty_{\delta}(\,\cdot\,;d_{1l},d_{2l},k)|^2-|u^\infty(\,\cdot\,;d_{1l},d_{2l},k)|^2\big\|_{L^2(\Sp^1_+)}
\leq\delta\big\||u^\infty(\,\cdot\,;d_{1l},d_{2l},k)|^2\big\|_{L^2(\Sp^1_+)},\quad l=1,2,\ldots,n_d,
\enn
where $\delta>0$ is called the noise ratio and the observation directions $\hat{x_j},j=1,2,\ldots,n_f,$
are the equidistant points on $\Sp^1_+$.
In practical computation, $h^{app}$ has to be taken from a finite-dimensional subspace $R_M$.
Here, $R_M=\textrm{span}\{\phi_{1,M},\phi_{2,M},\cdots,\phi_{M,M}\}$ is a subspace of $C^2_{0,R}(\R)$,
where $\phi_{j,M},j=1,2,\ldots,M,$ are spline basis functions with support in $(-R,R)$ (see Remark \ref{re1_nr}).
Then, by the strategy in \cite{Hohage1999}, we seek an updated function
$\triangle h=\sum^M_{i=1}\triangle a_i\phi_{i,M}$ in $R_M$ so as to solve the minimization problem:
\be\no
&&\min_{\triangle a_i}\left\{\sum_{l=1}^{n_d}\sum_{j=1}^{n_f}
\Big|\big(\mathcal{F}_{(d_{1l},d_{2l},k)}[h^{app}]\big)(\hat{x}_j)
+\big(\mathcal{F}_{(d_{1l},d_{2l},k)}'[h^{app};\triangle h]\big)(\hat{x}_j)\right.\\ \label{eq15_nr}
&&\qquad\qquad\left. -|u^\infty_{\delta}(\hat{x_j};d_{1l},d_{2l},k)|^2\Big|^2+\beta\sum_{i=1}^M|\triangle a_i|^2\right\}
\en
where $\beta>0$ is chosen such that
\be\label{eq16_nr}
\left[\sum_{l=1}^{n_d}\sum_{j=1}^{n_f}\Big|\big(\mathcal{F}_{(d_{1l},d_{2l},k)}[h^{app}]\big)(\hat{x}_j)
+\big(\mathcal{F}_{(d_{1l},d_{2l},k)}'[h^{app};\triangle h]\big)(\hat{x}_j)
-|u^\infty_{\delta}(\hat{x_j};d_{1l},d_{2l},k)|^2\Big|^2\right]^{\half}\no\\
=\rho\left[\sum_{l=1}^{n_d}\sum_{j=1}^{n_f}\Big|\big(\mathcal{F}_{(d_{1l},d_{2l},k)}[h^{app}]\big)(\hat{x}_j)
-|u^\infty_{\delta}(\hat{x_j};d_{1l},d_{2l},k)|^2\Big|^2\right]^{\half}
\en
for a given constant $\rho\in(0,1)$. Then the approximation function $h^{app}$ is updated by $h^{app}+\triangle h$.
Further, define the error function
\ben
Err_k:=\frac{1}{n_d}\sum_{l=1}^{n_d}
\begin{array}{c}
{\left[\sum\limits_{j=1}^{n_f}\Big|\big(\mathcal{F}_{(d_{1l},d_{2l},k)}[h^{app}]\big)(\hat{x}_j)
-|u^\infty_{\delta}(\hat{x_j};d_{1l},d_{2l},k)|^2\Big|^2\right]^{1/2}} \\
\hline
{\left[\sum\limits_{j=1}^{n_f}|u^\infty_{\delta}(\hat{x_j};d_{1l},d_{2l},k)|^4\right]^{1/2}}
\end{array}
\enn
Then the iteration is stopped if $Err_k< \tau\delta$, where $\tau>1$ is a fixed constant.

\begin{remark}\label{re1_nr} {\rm
For a positive integer $M\in\N^+$ let $h=2R/(M+5)$ and $t_i=(i+2)h-R$.
Then the spline basis functions of $R_M$ are defined by
$\phi_{i,M}(t)=\phi((t-t_i)/h),i=1,2,\ldots,M,$ where
\ben
\phi(t):=\sum^{k+1}_{j=0}\frac{(-1)^j}{k!}\left(\begin{array}{c}k+1\\
                                      j\end{array}\right)
       \left(t+\frac{k+1}{2}-j\right)^k_+
\enn
with $z^k_+=z^k$ for $z\geq0$ and $=0$ for $z<0$.
In this paper, we choose $k=4$, that is, $\phi$ is the quartic spline function.
Note that $\phi_{i,M}\in C^3(\R)$ with support in $(-R,R)$. See \cite{DeBoor} for details.
}
\end{remark}

\begin{remark}\label{re2_nr} {\rm
The integral equation method proposed in \cite{ZhangZhang2013} (cf. Section \ref{se1}) is used to solve
the direct scattering problem in each iteration.
}
\end{remark}

Our recursive Newton iteration algorithm in frequencies is given in Algorithm \ref{alg1} for the inverse problem (IP1).

\begin{algorithm}\label{alg1}
Given the phaseless far-field data $|u^\infty_{\delta}(\hat{x_j};d_{1l},d_{2l},k_m)|$,
$j=1,2,\ldots,n_f$, $l=1,2,\ldots,n_d$, $m=1,2,\ldots,N$ with $d_{1l},d_{2l}\in\Sp^1_-$, $d_{1l}\neq d_{2l}$
and $k_1<k_2<\cdots<k_N$.

1) Let $h^{app}$ be the initial guess of $h_\G$. Set $i=0$ and go to Step 2).

2) Set $i=i+1$. If $i>N$, then stop the iteration; otherwise, set $k=k_i$ and go to Step 3).

3) If $Err_k<\tau\delta$, go to Step 2); otherwise, go to Step 4).

4) Solve (\ref{eq15_nr}) with the strategy (\ref{eq16_nr}) to get an updated function $\triangle h$.
Let $h^{app}$ be updated by $h^{app}+\triangle h$ and go to Step 3).
\end{algorithm}

\begin{remark}\label{re3_nr} {\rm
Since $\mathcal{F}'_{(d_1,d_2,k)}[0;\triangle h]=0$ for any $d_1,d_2\in\Sp^1_-$, $k>0$ and
$\triangle h\in C^2_{0,R}(\R^2)$, then the initial guess $h^{app}$ should not be zero for
the inverse problem (IP1). However, in all the numerical examples for the inverse problem (IP2),
the initial guess of $h_\G$ is chosen to be $h^{app}=0$.
}
\end{remark}

\section{Numerical experiments}\label{se4}

In this section, several numerical experiments are carried out to illustrate the effectiveness of
the inversion algorithm. The following assumptions are made in all numerical experiments.

1) For each example we use multi-frequency data with the wave numbers $k=1,3,\ldots,2N-1,$ where $N$
is the total number of frequencies.

2) To generate the synthetic data, the integral equation method is used to the direct scattering problem.
For the inverse problem (IP1), the intensity of the far-field pattern is measured along the upper half-aperture
(that is, the measurement angle is between $0$ and $\pi$) with $200$ equidistant measurement points.
For the inverse problem (IP2), the intensity of the near-field is measured on the straight line segment
$\G_{1,1}:=\{(x_1,1)\;:\;x_1\in[-1,1]\}$ also with $200$ equidistant measurement points.
The corresponding noisy data $|u^\infty_{\delta}|$ and $|u^{near}_{\delta}|$ are simulated as
$|u^\infty_{\delta}|^2=|u^\infty|^2(1+\delta\zeta)$ and $|u^{near}_{\delta}|^2=|u^{near}|^2(1+\delta\zeta)$,
respectively, where $\zeta$ is a normally distributed random number in $[-1,1].$
In all the numerical examples, the noise level is taken as $\delta=5\%$.

3) The parameters are taken as $\rho=0.8$ and $\tau=1.5$.

4) In all the numerical examples, the local perturbation of the infinite plane is assumed to be restricted to
the range $-1<x_1<1$, that is, $\textrm{supp}(h_\G)\in(-1,1)$.

\textbf{Example 1 (Shape reconstruction).} We first demonstrate numerically that the location of the local
perturbation on the infinite plane can not be determined from the intensity far-field data if only one plane
wave is used as the incident field, that is, $u^i(x)=u^i(x;d,k)$. The locally rough surface is given by
\ben
h_\G(x_1)=\phi(({x_1+0.2})/{0.3}),
\enn
where $\phi$ is defined as in Remark \ref{re1_nr}, and the incident direction is taken as
$d=(\sin(-\pi/6),-\cos(-\pi/6))$. The number of the spline basis functions is chosen to be $M=10$,
and the total number of frequencies used is $N=10$. Further, we choose two different initial guesses
for the iteration inversion algorithm: 1) $h^{app}=0.1\sum_{i=2}^4\phi_{i,M}$, and
2) $h^{app}=0.1\sum_{i=5}^7\phi_{i,M}$, where the corresponding reconstructed curves
are denoted as "Reconstructed curve 1" and "Reconstructed curve 2".
Figure \ref{fig4-1} presents the initial curves and the reconstructed curves at the wavenumbers $k=1,7,19$.
From the reconstructions presented in Figure \ref{fig4-1} it is seen that
the shape of the local perturbation can be well reconstructed numerically, but its reconstructed location
is translated along the $x_1$ direction, depending on the choice of initial guess. This means that
the location of the local perturbation can not be recovered from the intensity far-field data
if one plane wave is taken as the incident field.

\begin{figure}[htbp]
  \centering
  \subfigure{\includegraphics[width=3in]{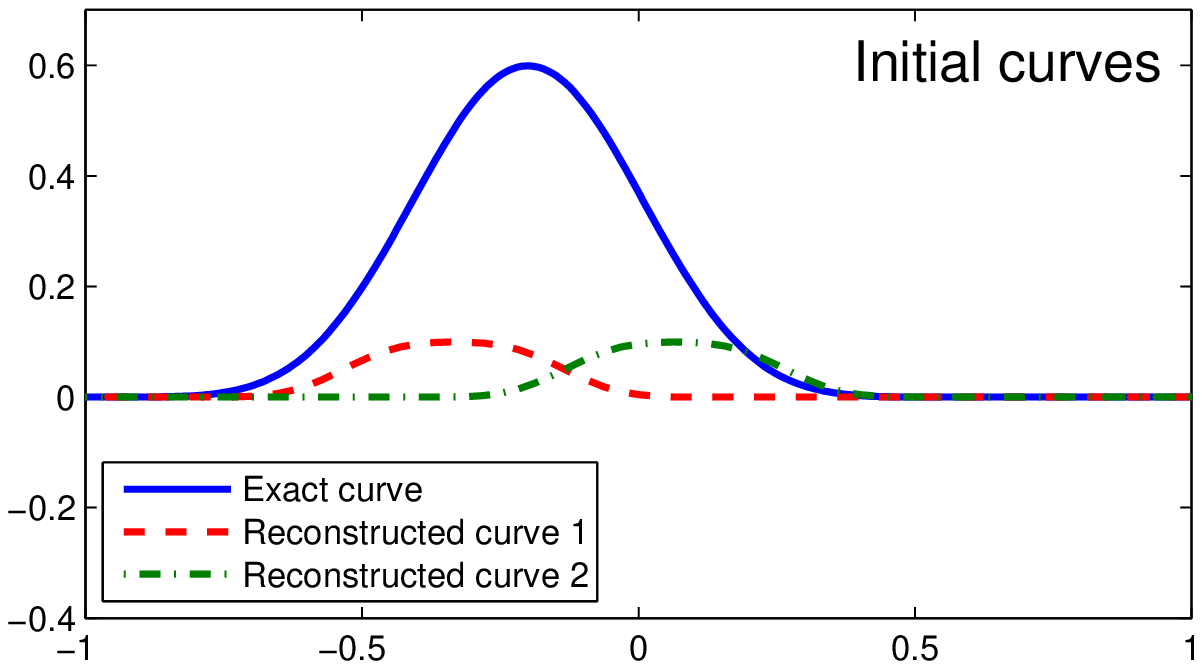}}
  \hspace{-0.35in}
  \vspace{-0.5in}
  \subfigure{\includegraphics[width=3in]{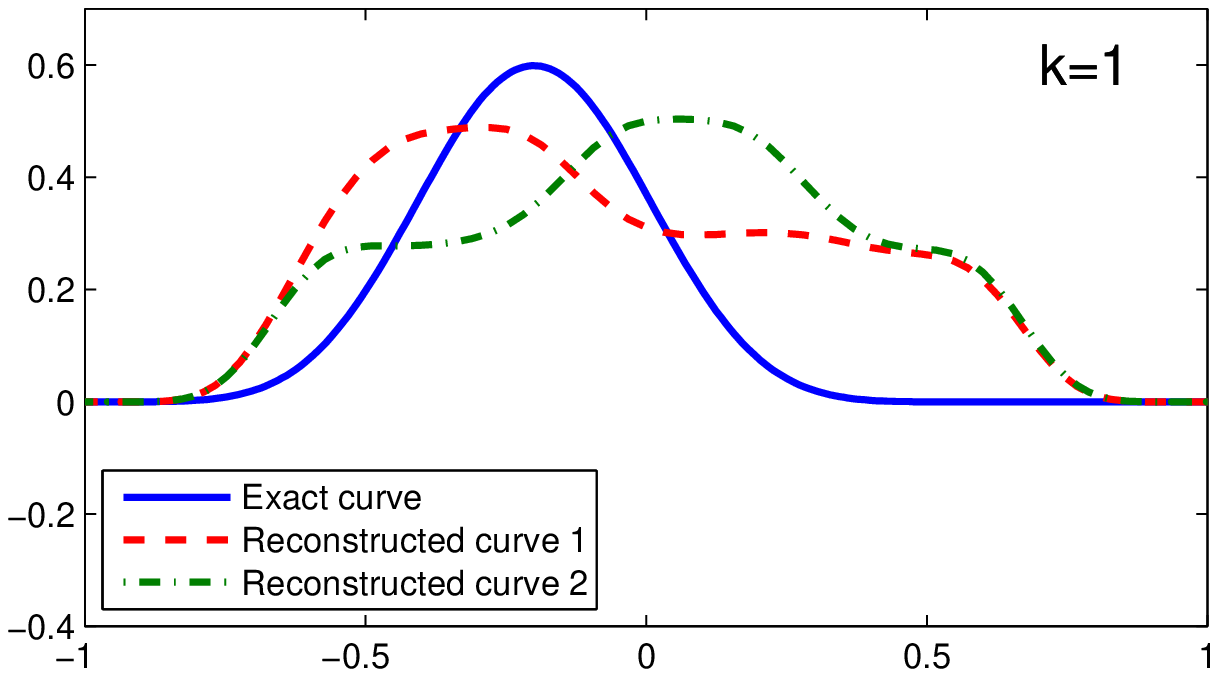}}
  \vspace{0in}
  \hspace{-0.35in}
  \subfigure{\includegraphics[width=3in]{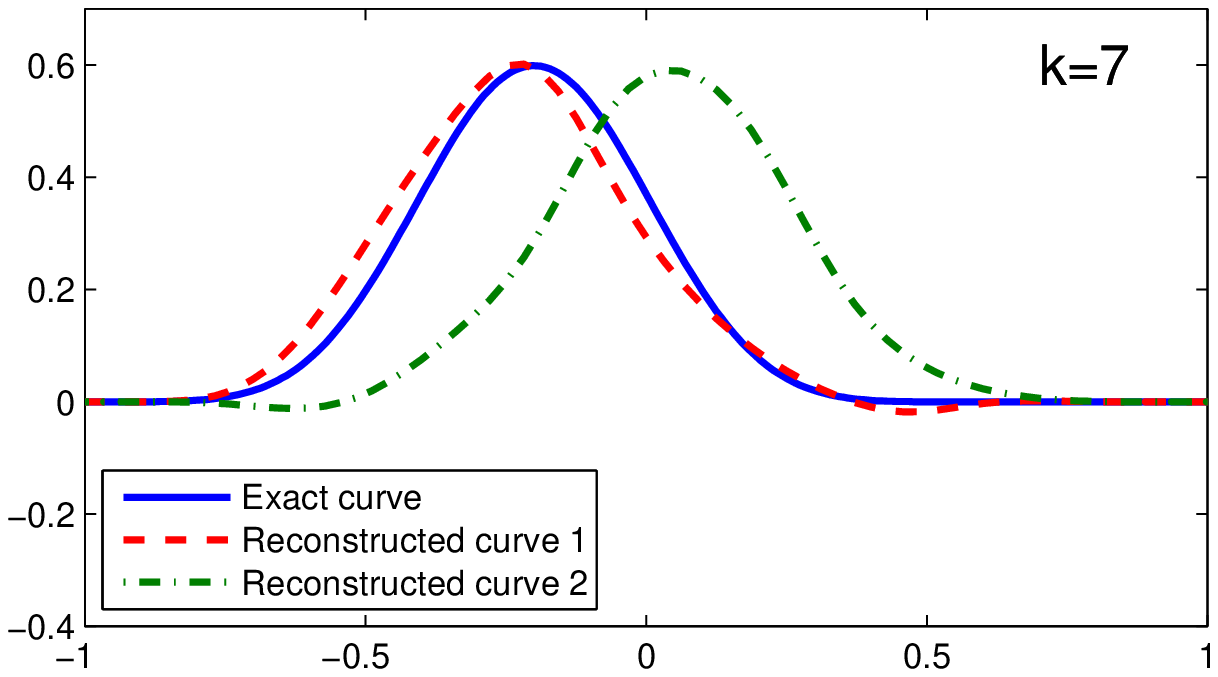}}
  \vspace{0in}
  \hspace{-0.35in}
  \subfigure{\includegraphics[width=3in]{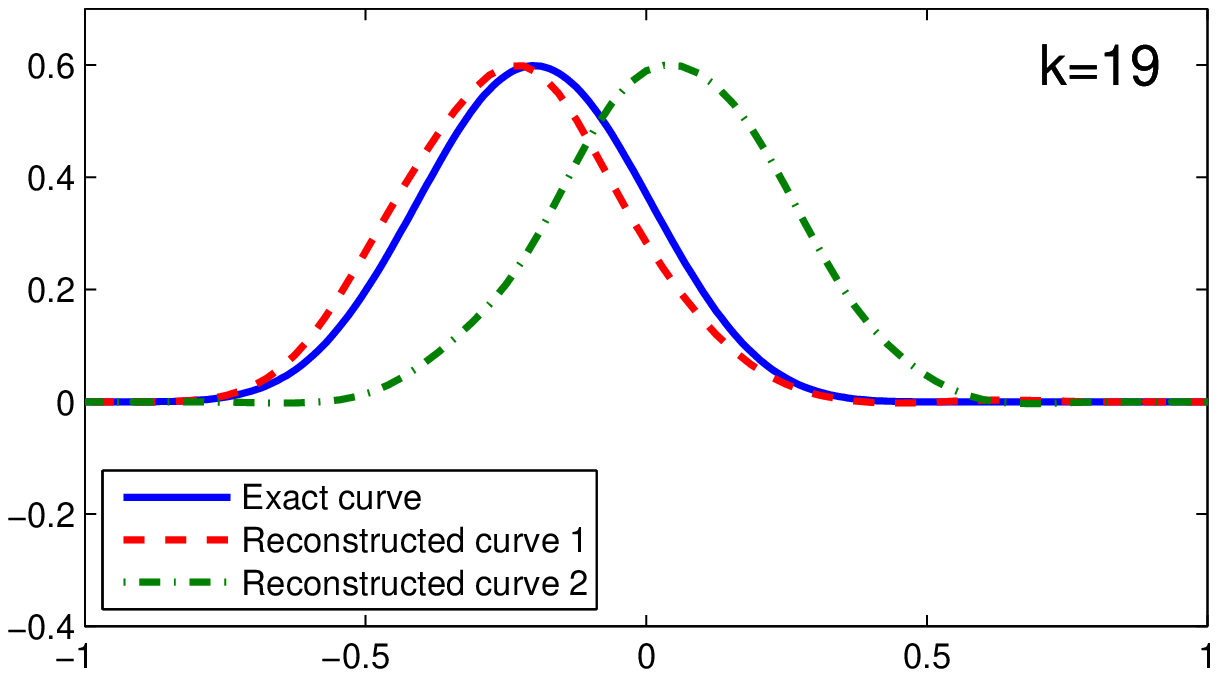}}
  \vspace{-0.5in}
\caption{Shape reconstruction of a locally rough surface from $5\%$ noisy intensity far-field data
with one incident plane wave $u^i=u^i(x;d,k)$, where $d=(\sin(-\pi/6),-\cos(-\pi/6))$.
Here, the initial and reconstructed curves are presented at the wave numbers $k=1,7,19.$
}\label{fig4-1}
\end{figure}

\textbf{Example 2 (Shape and location reconstruction).} We consider the same locally rough surface as
in Example 1. Here, we demonstrate that by using a superposition of two plane waves with different
directions as the incident field, that is, $u^i=u^i(x;d_1,d_2,k)$ with $d_1\not=d_2$,
both the location and the shape of the local perturbation on the rough surface can be reconstructed from
the corresponding intensity far-field data.
In the numerical experiment, the two incident directions are taken as $d_1=(\sin(-\pi/6),-\cos(-\pi/6))$ and
$d_2=(\sin(\pi/6),-\cos(\pi/6))$, the number of the spline basis functions is chosen to be $M=10$, and
the total number of frequencies used is $N=13$. Further, we use the same initial guesses and
the same terms for the corresponding reconstructed curves of $h_\G$ as in Example 1.
The initial and reconstructed curves are presented in Figure \ref{fig4-2}, at the wavenumbers $k=1,7,25,$
respectively. From Figures \ref{fig4-1} and \ref{fig4-2} it can be seen that both the location and the shape
of the local perturbation are accurately reconstructed with two different initial guesses.

\begin{figure}[htbp]
  \centering
  \subfigure{\includegraphics[width=3in]{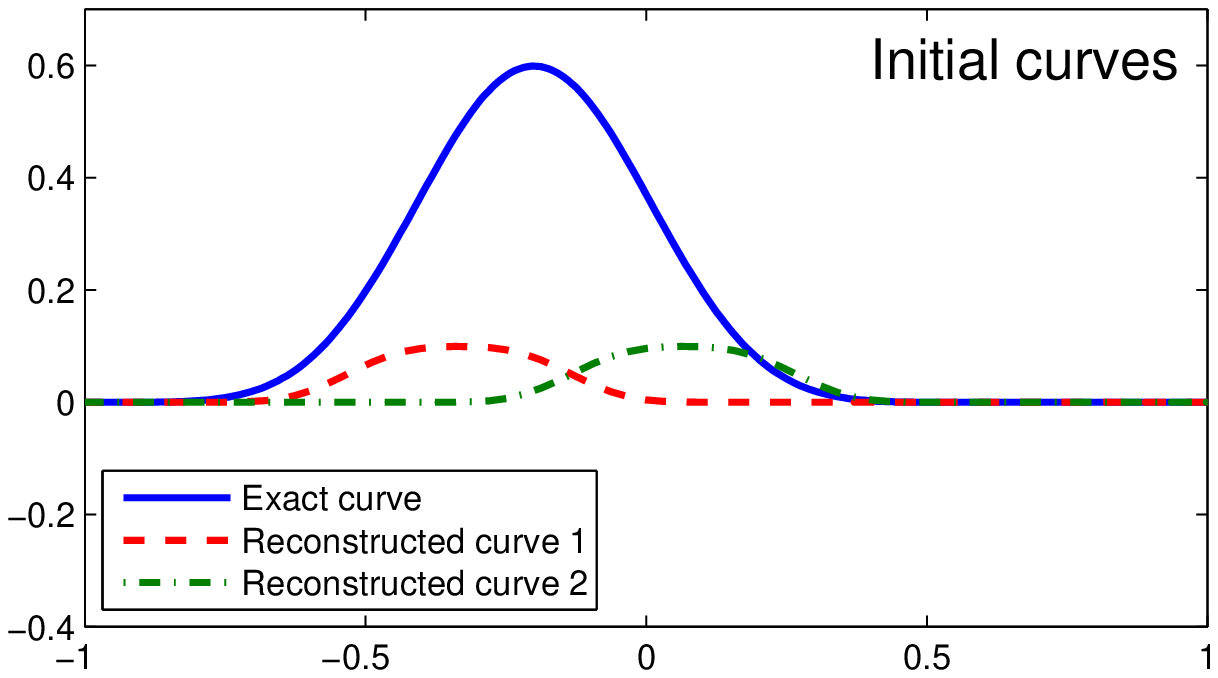}}
  \hspace{-0.35in}
  \vspace{-0.5in}
  \subfigure{\includegraphics[width=3in]{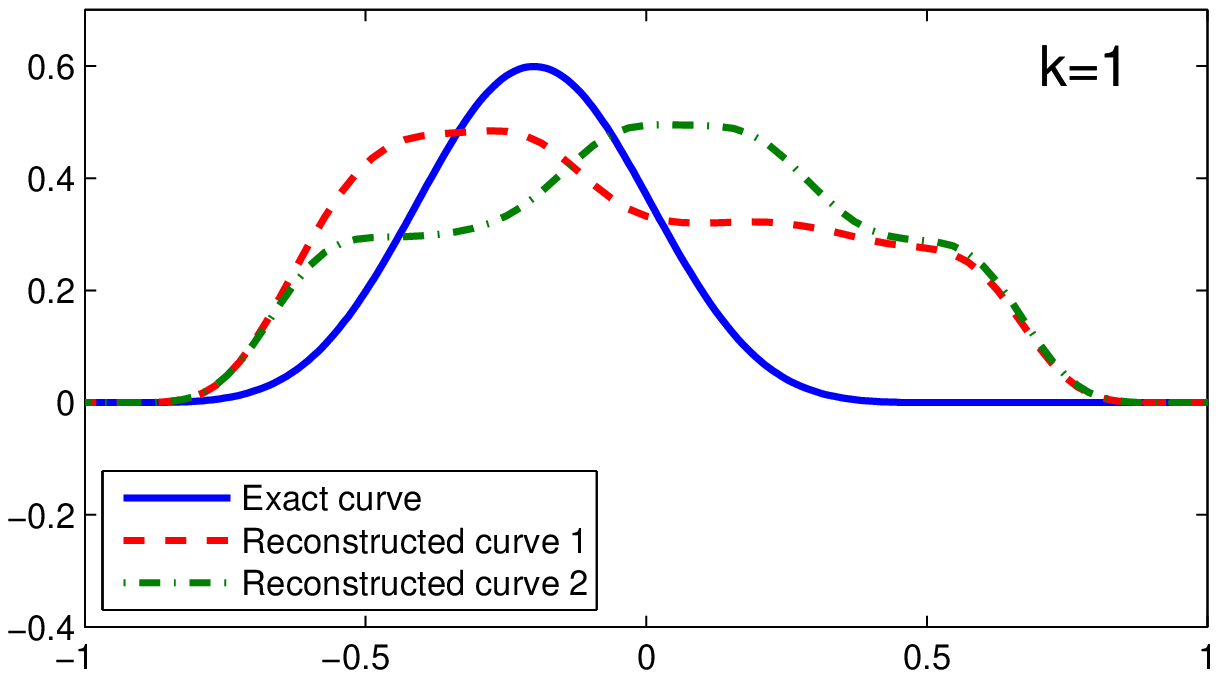}}
  \vspace{0in}
  \hspace{-0.35in}
  \subfigure{\includegraphics[width=3in]{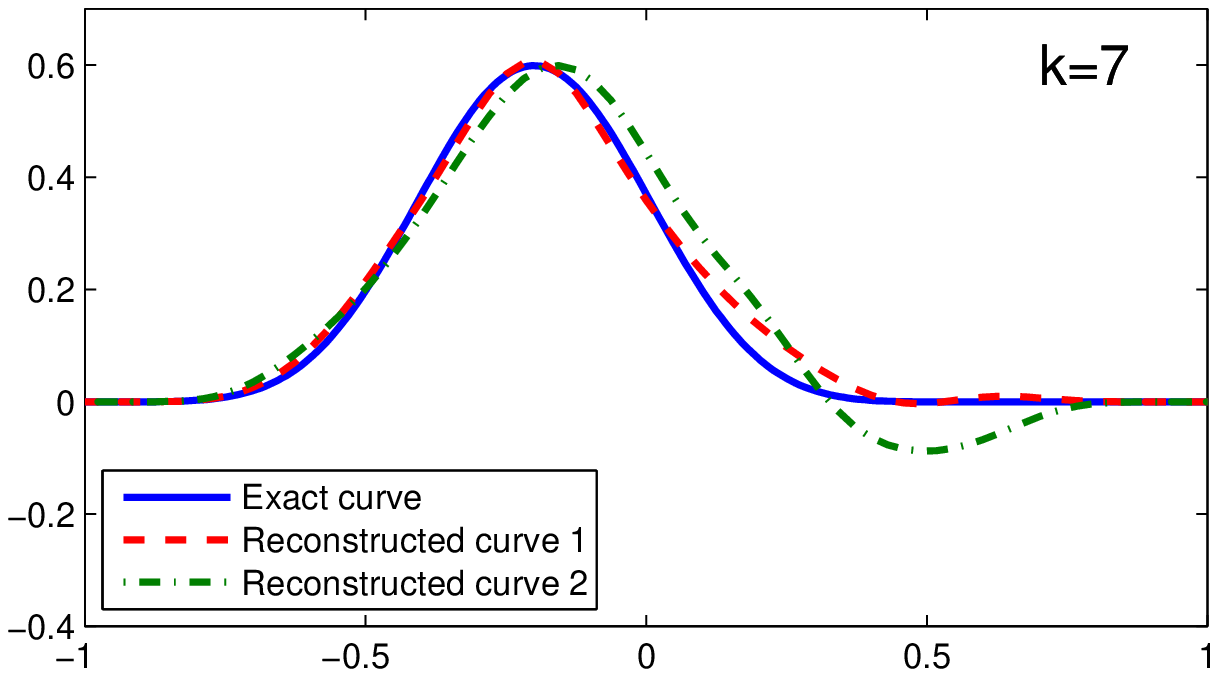}}
  \vspace{0in}
  \hspace{-0.35in}
  \subfigure{\includegraphics[width=3in]{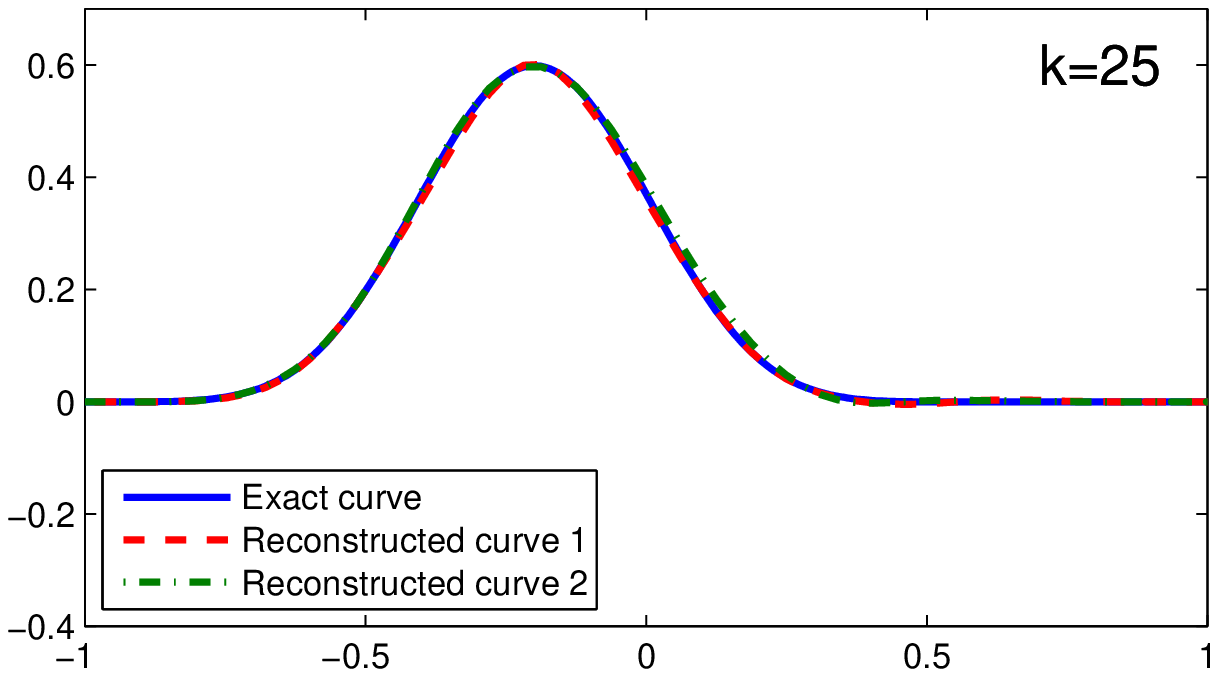}}
  \vspace{-0.5in}
\caption{Location and shape reconstruction of a locally rough surface from $5\%$ noisy intensity far-field data
with a superposition of two plane waves as the incident field $u^i=u^i(x;d_1,d_2,k)$,
where $d_1=(\sin(-\pi/6),-\cos(-\pi/6))$ and $d_2=(\sin(\pi/6),-\cos(\pi/6))$.
Here, the initial and reconstructed curves are presented at the wave numbers $k=1,7,25$.
}\label{fig4-2}
\end{figure}

\textbf{Example 3 (Piecewise linear curve).} We now consider the inverse problem (IP1) with a piecewise linear local perturbation
(the solid line in Figure \ref{c5fig4-9}). The number of the spline basis functions is chosen as $M=40$, the total number of
frequencies used is taken to be $N=18$, and the initial guess for the rough surface profile $h_\G$ is chosen as
$h^{app}=0.05\sum_{i=5}^{15}\phi_{i,M}$. Figure \ref{c5fig4-9} presents the initial and reconstructed curves
at the wavenumbers $k=7,21,35$, respectively, obtained by using the intensity far-field data with $5\%$ noise,
corresponding to the incident wave $u^i=u^i(x;d_1,d_2,k)$ with
$d_1=(\sin(-\pi/6),-\cos(-\pi/6))$ and $d_2=(\sin(\pi/6),-\cos(\pi/6))$.
The reconstruction results show that the piecewise linear surface profile is very accurately reconstructed even
at the corners of the boundary.

\begin{figure}[htbp]
  \centering
  \subfigure{\includegraphics[width=2.8in]{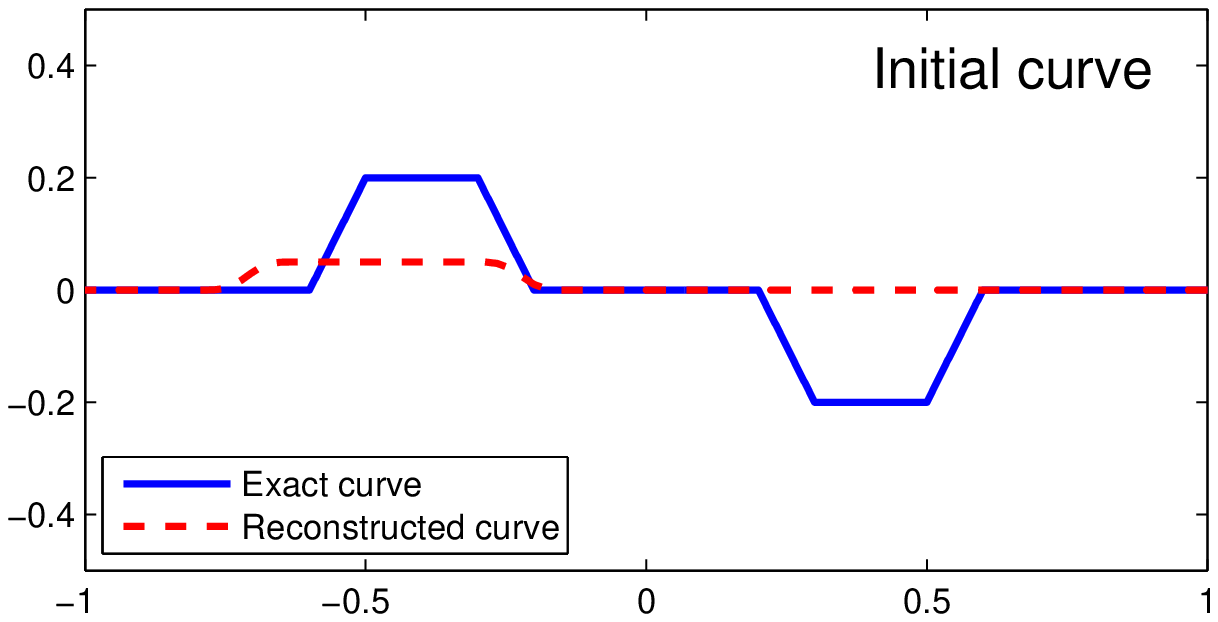}}
  \hspace{-0.35in}
  \vspace{-0.5in}
  \subfigure{\includegraphics[width=2.8in]{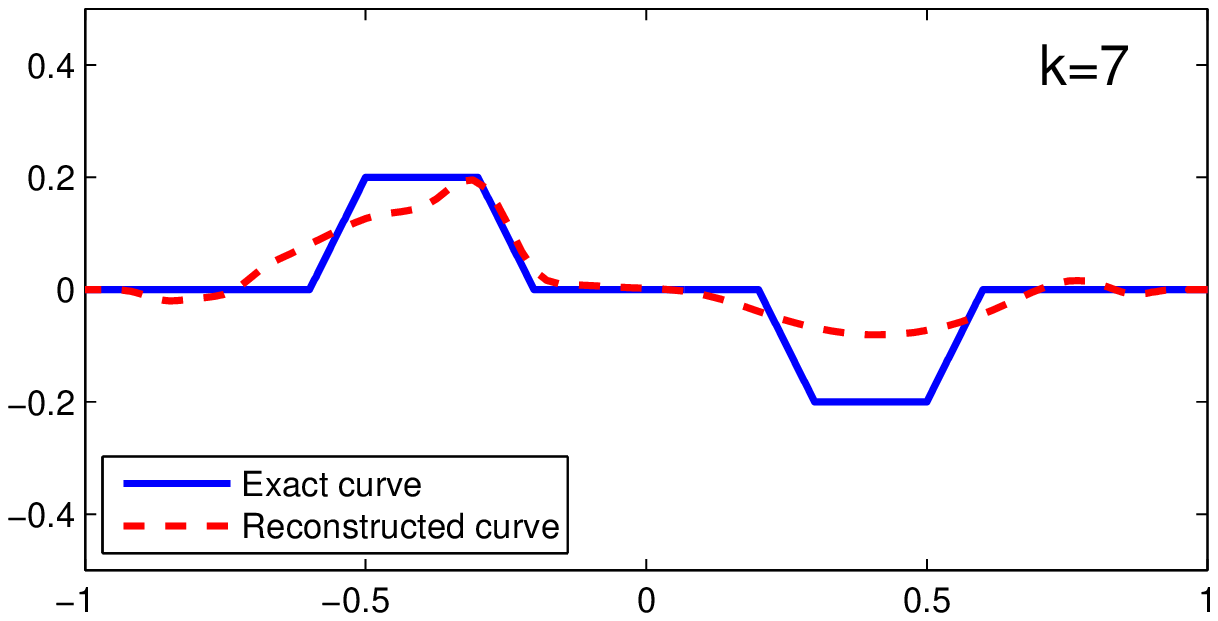}}
  \vspace{0in}
  \hspace{-0.35in}
  \subfigure{\includegraphics[width=2.8in]{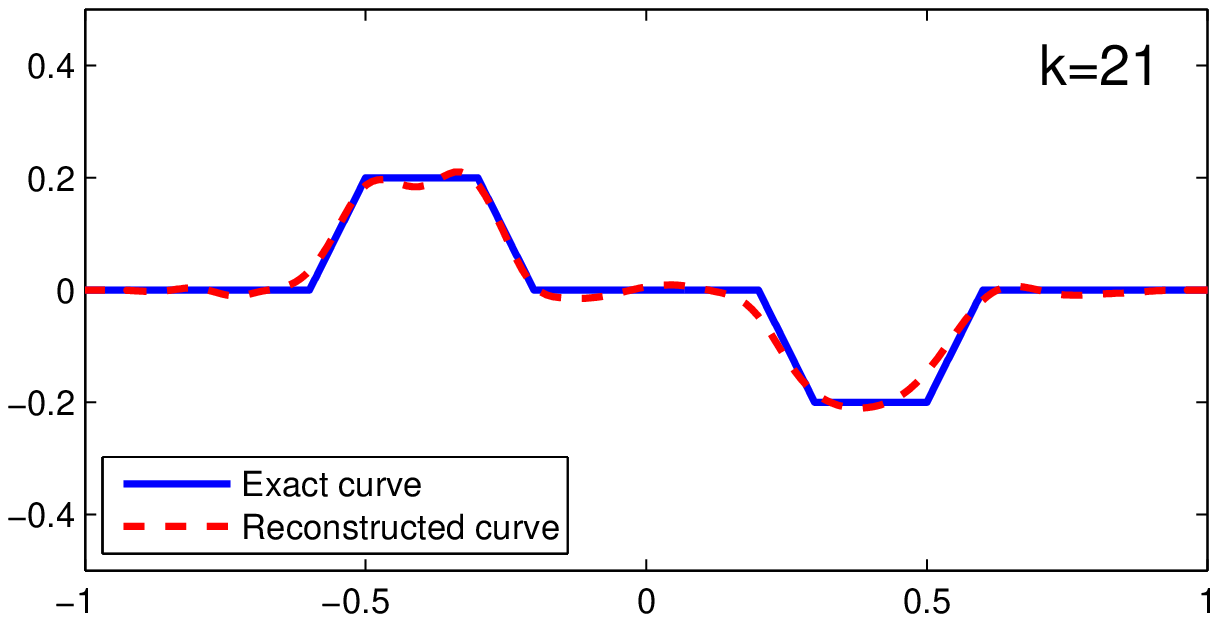}}
  \vspace{0in}
  \hspace{-0.35in}
  \subfigure{\includegraphics[width=2.8in]{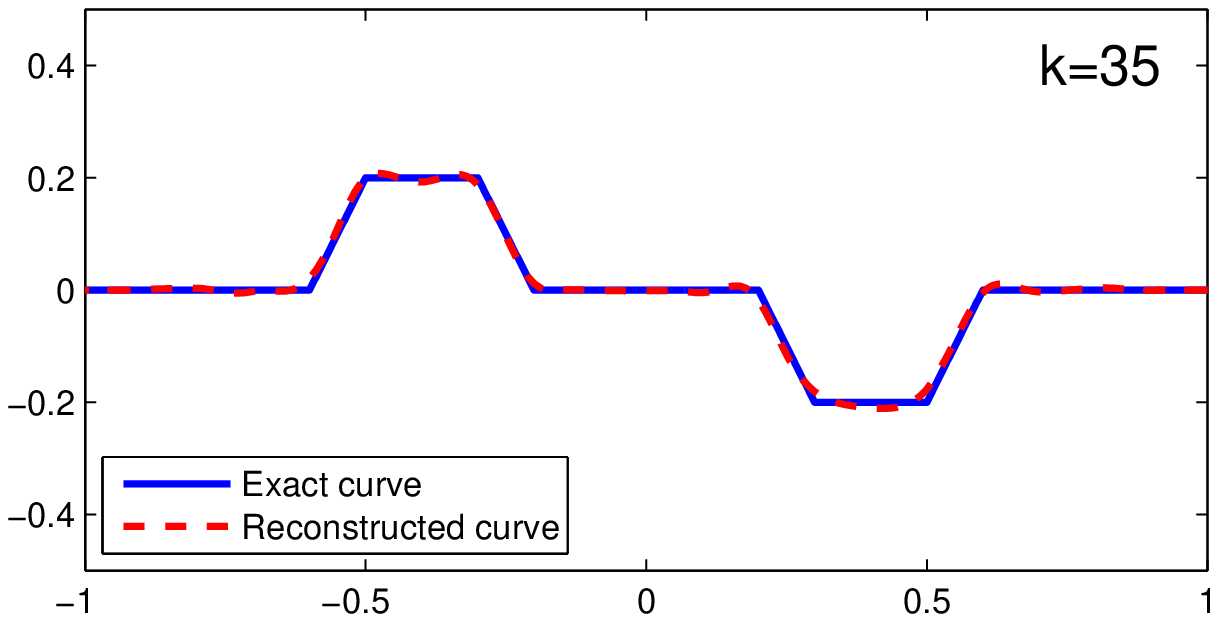}}
  \vspace{-0.5in}
\caption{Location and shape reconstruction of a piecewise linear local perturbation from $5\%$ noisy intensity
far-field data with a superposition of two plane waves as the incident field $u^i=u^i(x;d_1,d_2,k)$,
where $d_1=(\sin(-\pi/6),-\cos(-\pi/6))$ and $d_2=(\sin(\pi/6),-\cos(\pi/6))$.
Here, the initial and reconstructed curves are presented at the wave numbers $k=7,21,35$.
}\label{c5fig4-9}
\end{figure}

\textbf{Example 4 (multi-scale curve).} Consider the inverse problem (IP1) again with a multi-scale surface profile given by
\ben
h_\G(x_1)=\left\{\begin{array}{ll}
\ds \exp\left[16/(25x_1^2-16)\right]\left[0.5+0.1\sin(16\pi x_1)\right], &|x_1|<4/5\\
\ds 0, &|x_1|\geq4/5.
\end{array}\right.
\enn
This function has two scales: the macro-scale represented by the function $0.5\exp[16/(25x_1^2-16)]$ and the micro-scale
represented by the function $0.1\exp[16/(25x_1^2-16)]\sin(16\pi x_1)$.
For the inverse problem, the number of the spline basis functions is chosen to be $M=40$, the total number of
frequencies used is $N=30$, and the initial guess for $h_\G$ is $0.05\sum_{i=10}^{30}\phi_{i,M}$.
Figure \ref{fig4-6} presents the initial curve and the reconstructed curves at $k=19,33,39,49,59$,
obtained by using the intensity far-field data with $5\%$ noise, with respect to the incident wave
$u^i=u^i(x;d_1,d_2,k)$, where $d_1=(\sin(-\pi/6),-\cos(-\pi/6))$ and $d_2=(\sin(\pi/6),-\cos(\pi/6))$.
It is observed from Figure \ref{fig4-6} that the macro-scale of $h_\G$ is recovered at low frequencies (e.g. $k=19$),
but the micro-scale of $h_\G$ is not recovered completely even in the case when data with higher frequencies are used.

In order to get a better reconstruction, more measurement data are used.
Figure \ref{fig4-3} presents the initial and reconstructed curves at $k=19,33,39,49,59,$
obtained by using more intensity far-field data generated with the incident waves $u^i(x)=u^i(x;d_{1l},d_{2l},k),$ $l=1,2$,
where $d_{11}=(\sin(-\pi/6),-\cos(-\pi/6)),d_{21}=(0,-1),d_{12}=(\sin(\pi/6),-\cos(\pi/6))$ and $d_{22}=(0,-1)$.
Compared with Figure \ref{fig4-6}, the micro-scale of the surface profile $h_\G$ is accurately recovered.

\begin{figure}[htbp]
  \centering
  \subfigure{\includegraphics[width=3in]{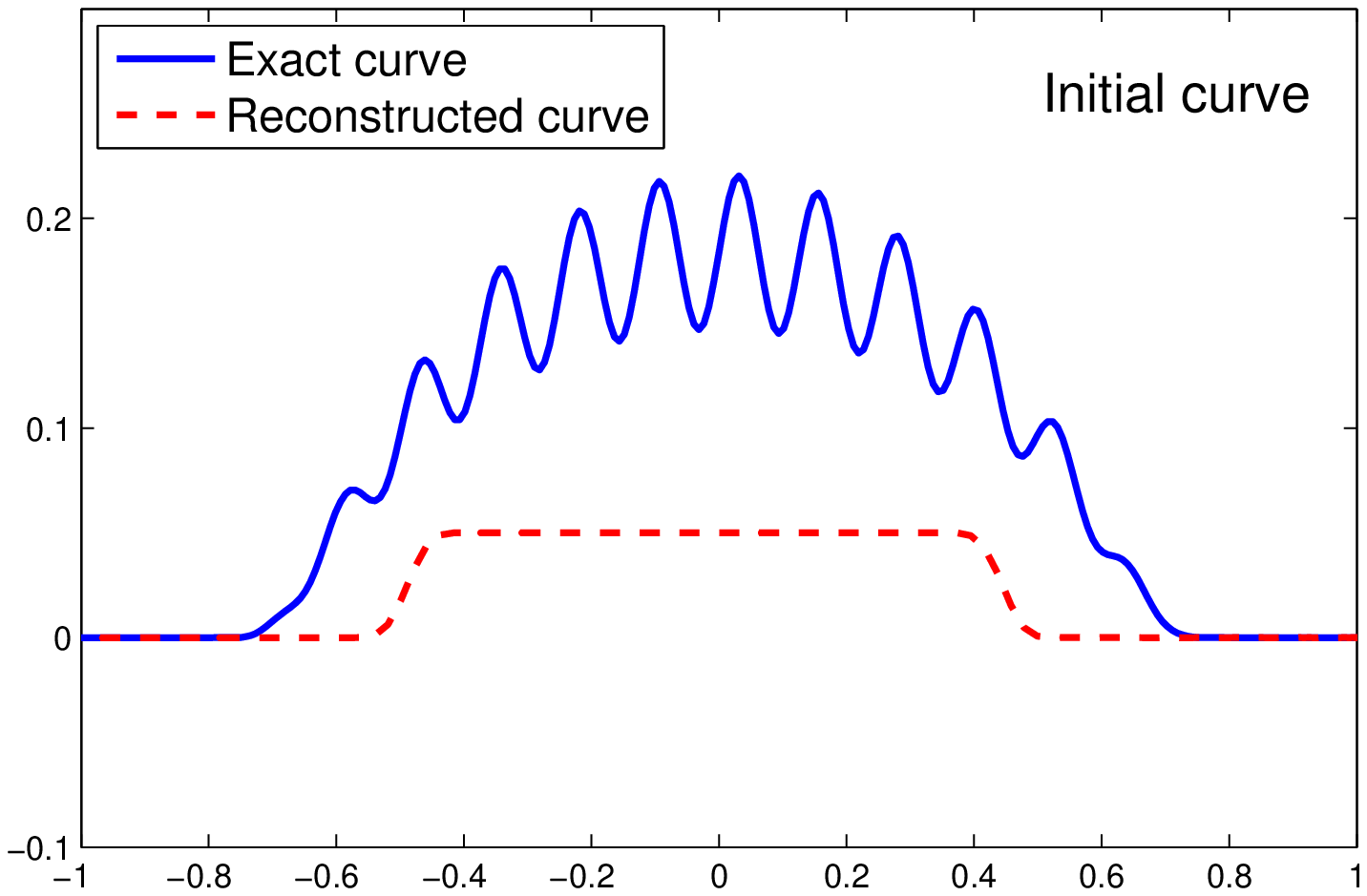}}
  \hspace{-0.35in}
  \vspace{0in}
  \subfigure{\includegraphics[width=3in]{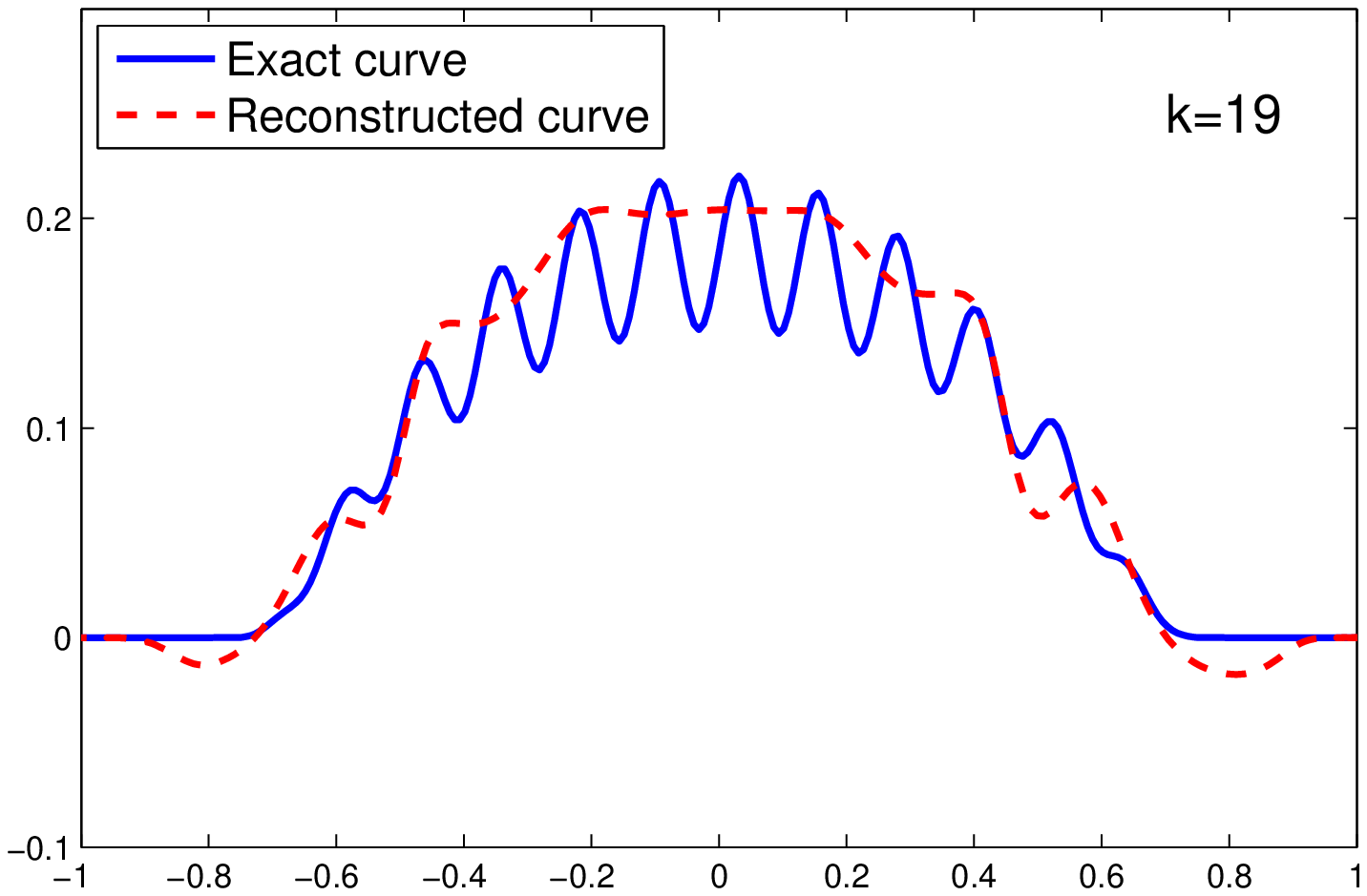}}
  \vspace{0in}
  \hspace{-0.35in}
  \subfigure{\includegraphics[width=3in]{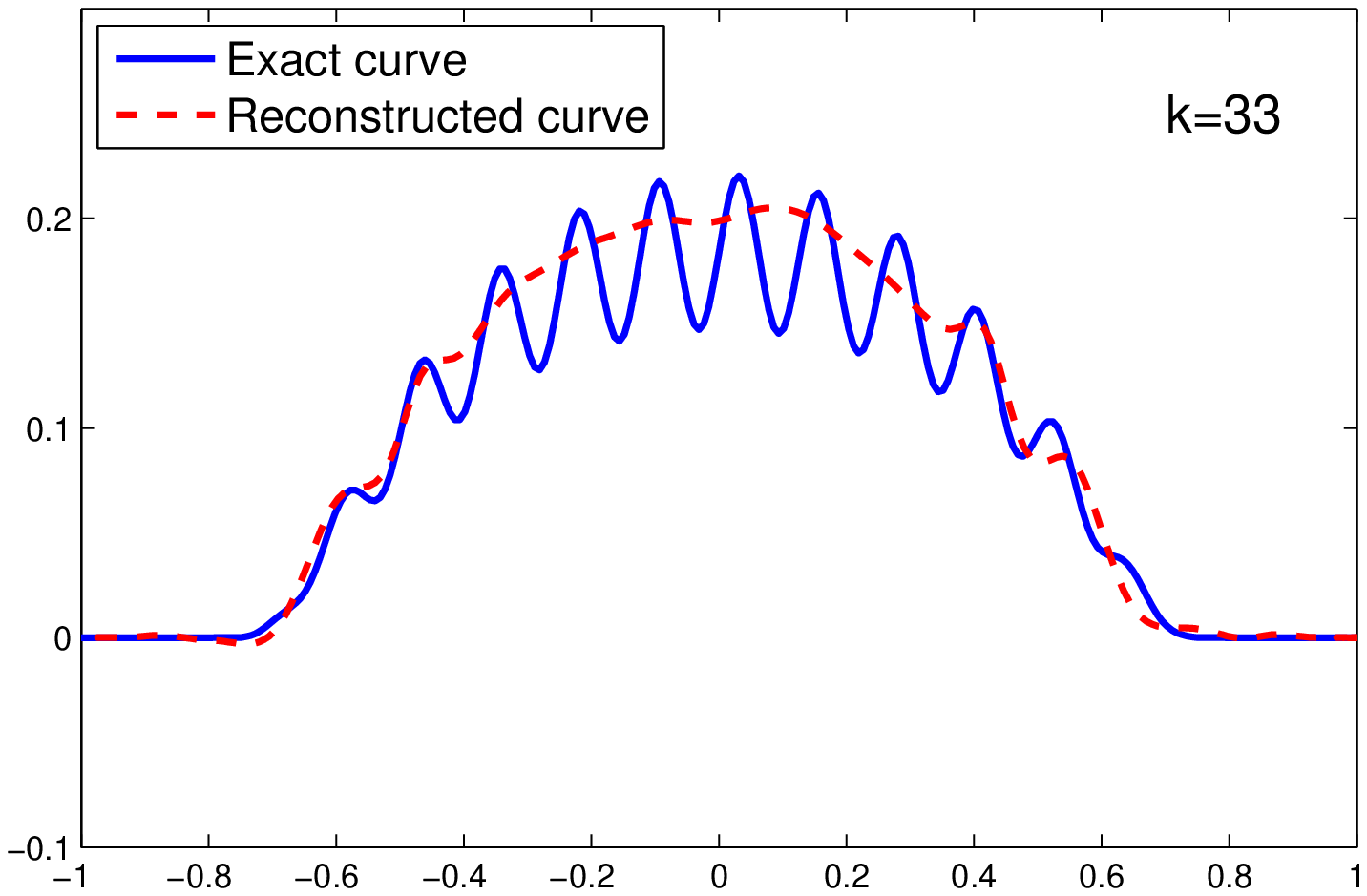}}
  \vspace{0in}
  \hspace{-0.35in}
  \subfigure{\includegraphics[width=3in]{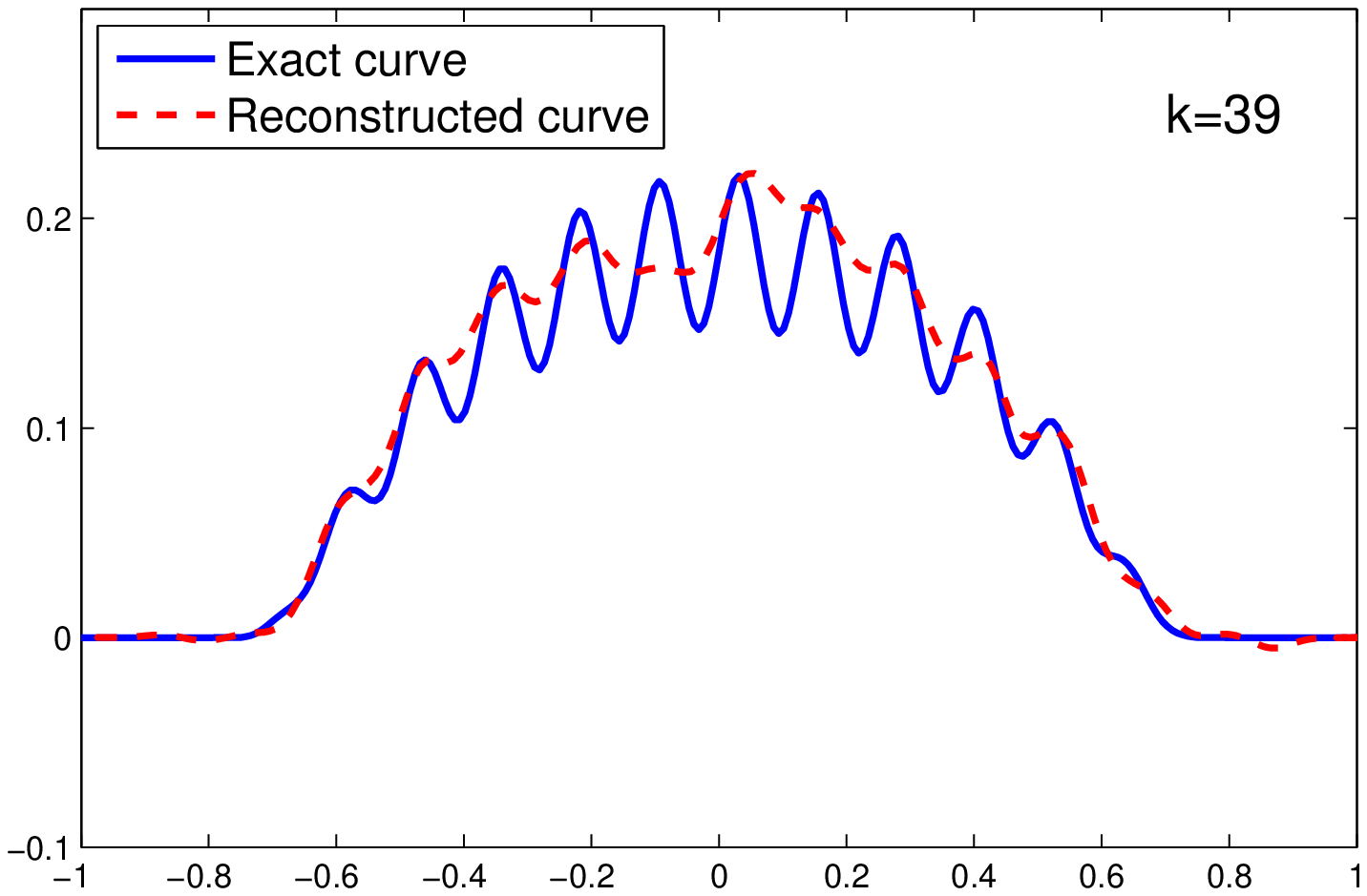}}
  \vspace{0in}
  \hspace{-0.35in}
  \subfigure{\includegraphics[width=3in]{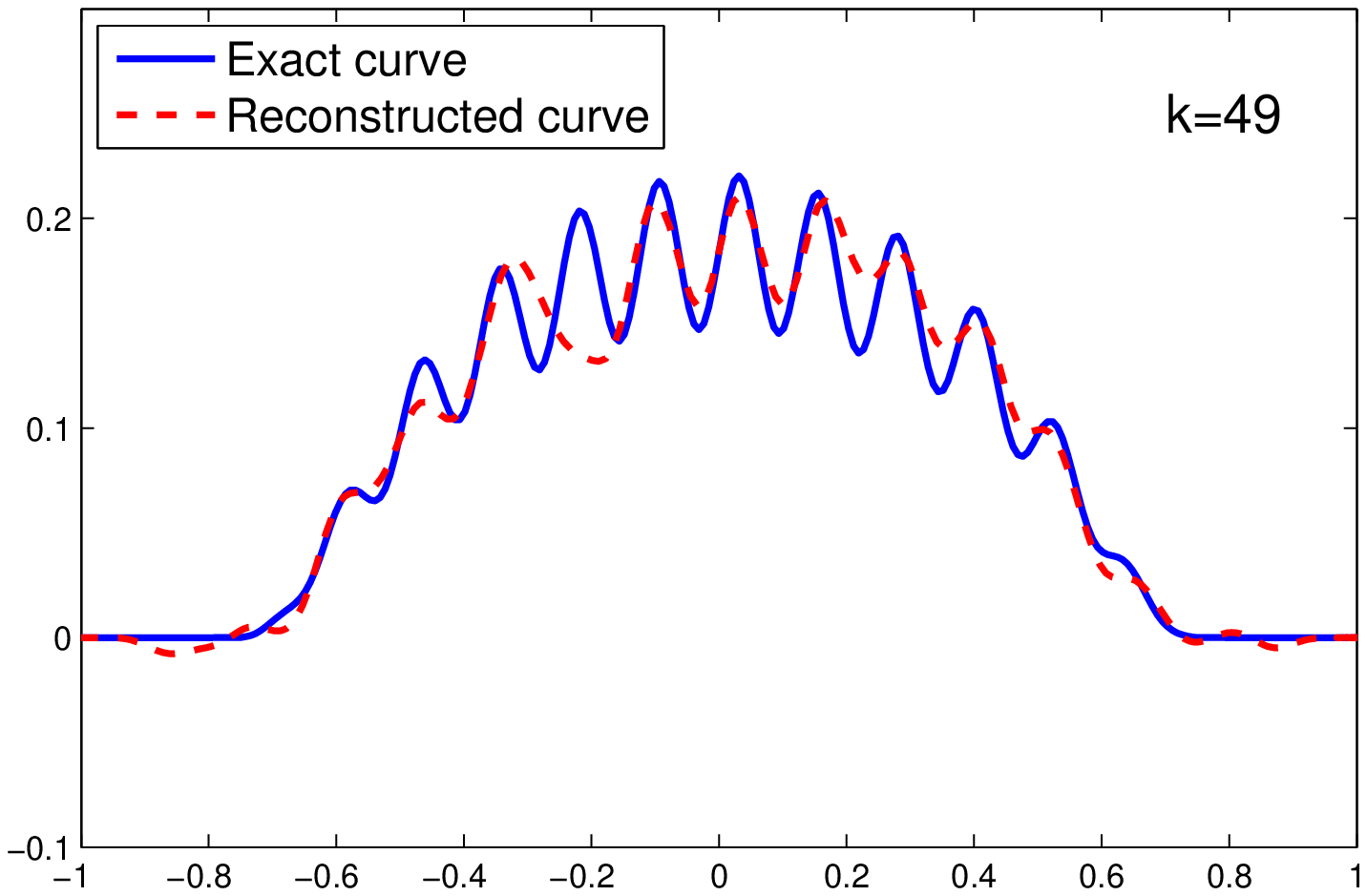}}
  \vspace{0in}
  \hspace{-0.35in}
  \subfigure{\includegraphics[width=3in]{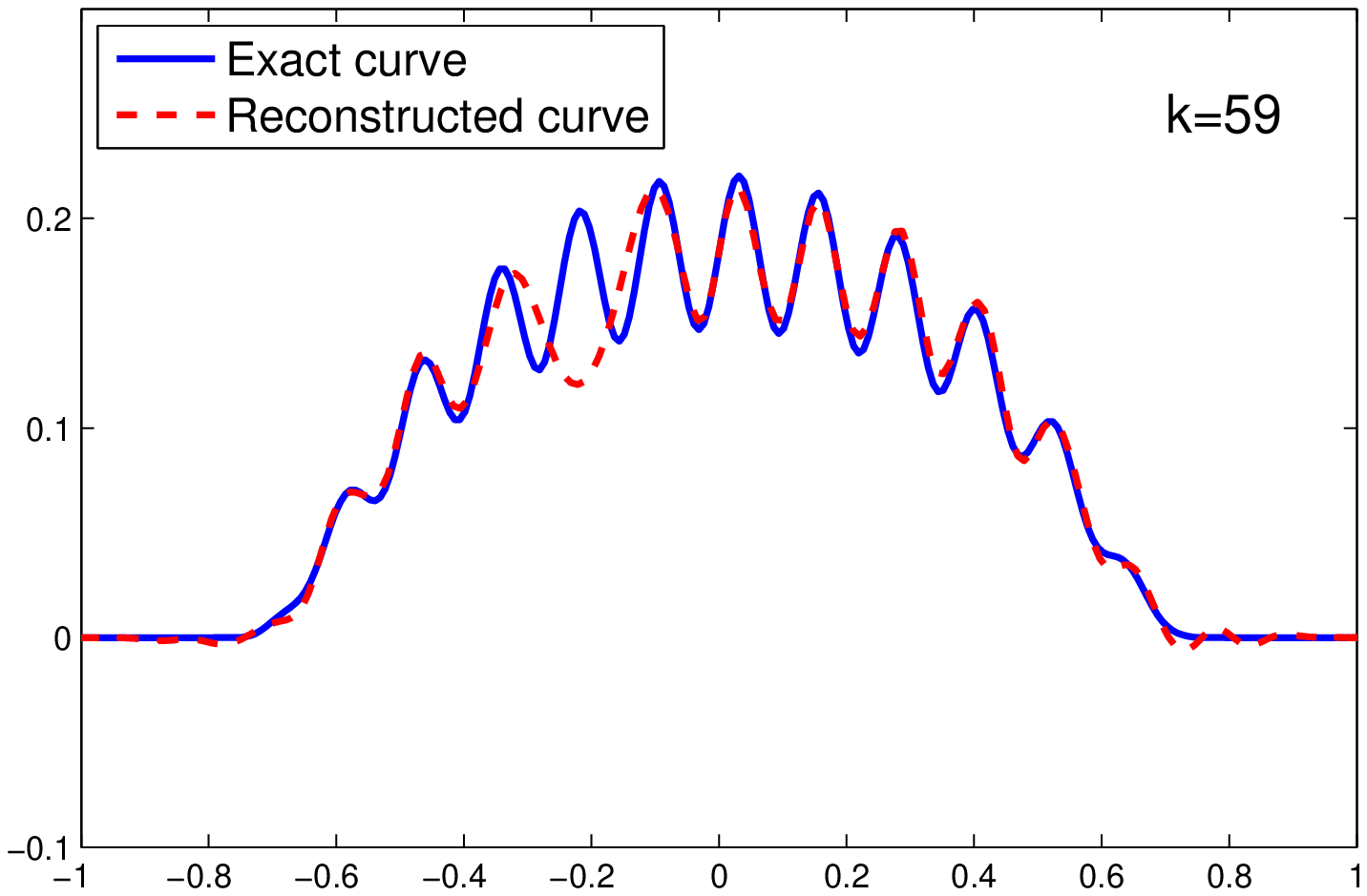}}
\caption{Reconstruction of a two-scale surface profile from $5\%$ noisy intensity far-field data
with a superposition of two plane waves as the incident field $u^i=u^i(x;d_1,d_2,k)$,
where $d_1=(\sin(-\pi/6),-\cos(-\pi/6))$ and $d_2=(\sin(\pi/6),-\cos(\pi/6))$.
Here, the initial and reconstructed curves are presented at the wave numbers $k=19,33,39,49,59$.
}\label{fig4-6}
\end{figure}

\begin{figure}[htbp]
  \centering
  \subfigure{\includegraphics[width=3in]{example/IP1/case6/initial_curve.eps}}
  \hspace{-0.35in}
  \vspace{0in}
  \subfigure{\includegraphics[width=3in]{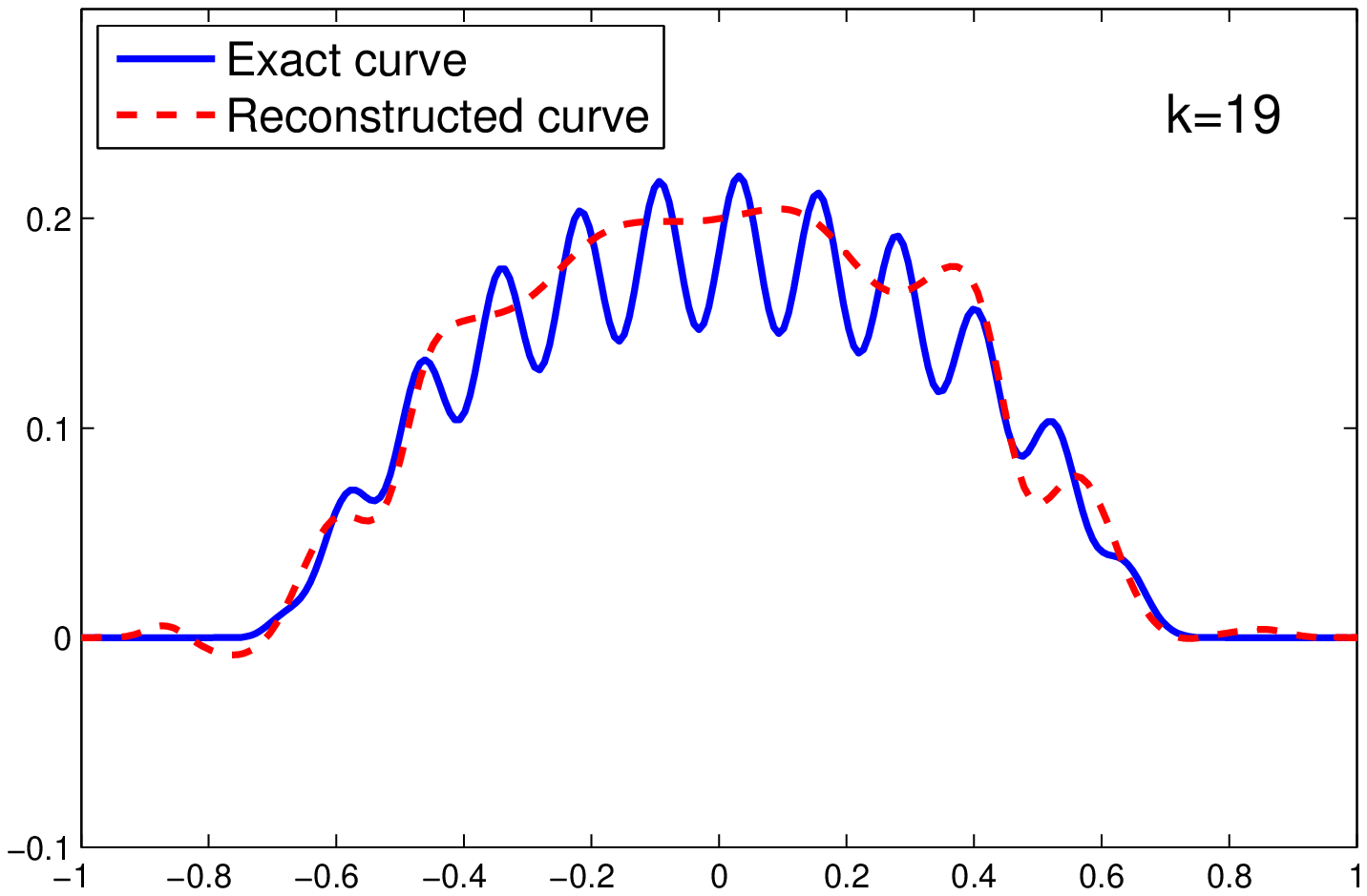}}
  \vspace{0in}
  \hspace{-0.35in}
  \subfigure{\includegraphics[width=3in]{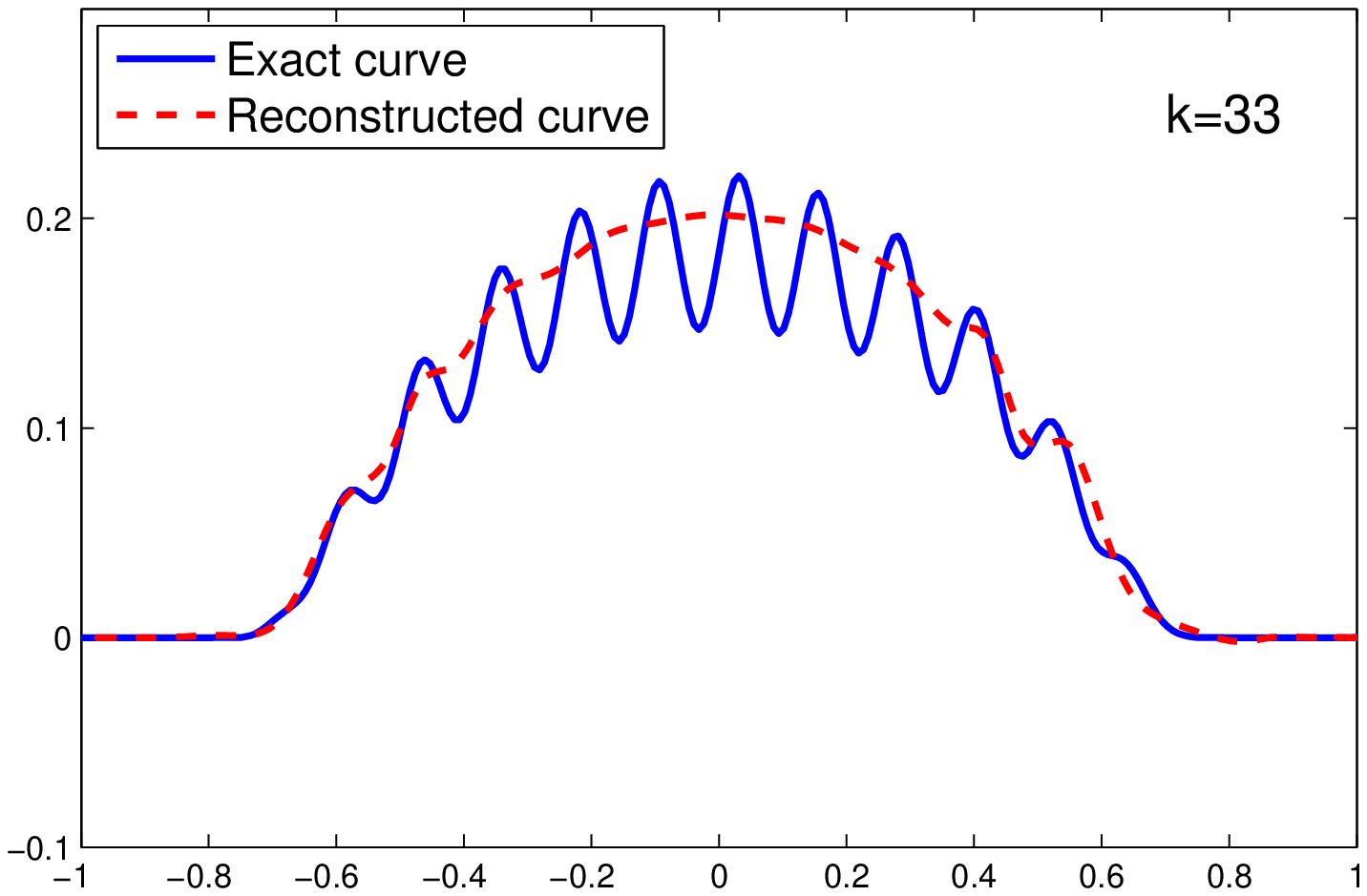}}
  \vspace{0in}
  \hspace{-0.35in}
  \subfigure{\includegraphics[width=3in]{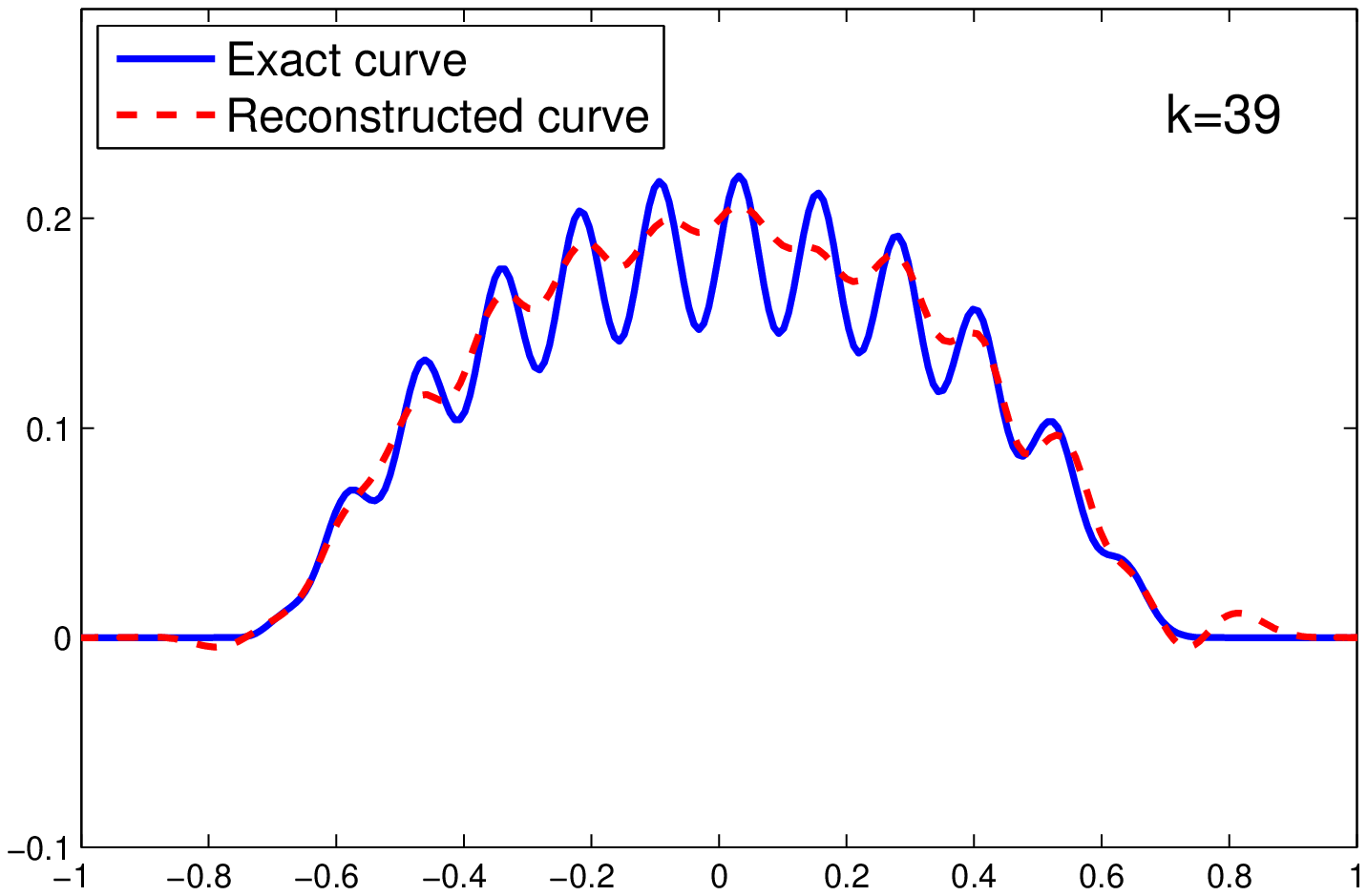}}
  \vspace{0in}
  \hspace{-0.35in}
  \subfigure{\includegraphics[width=3in]{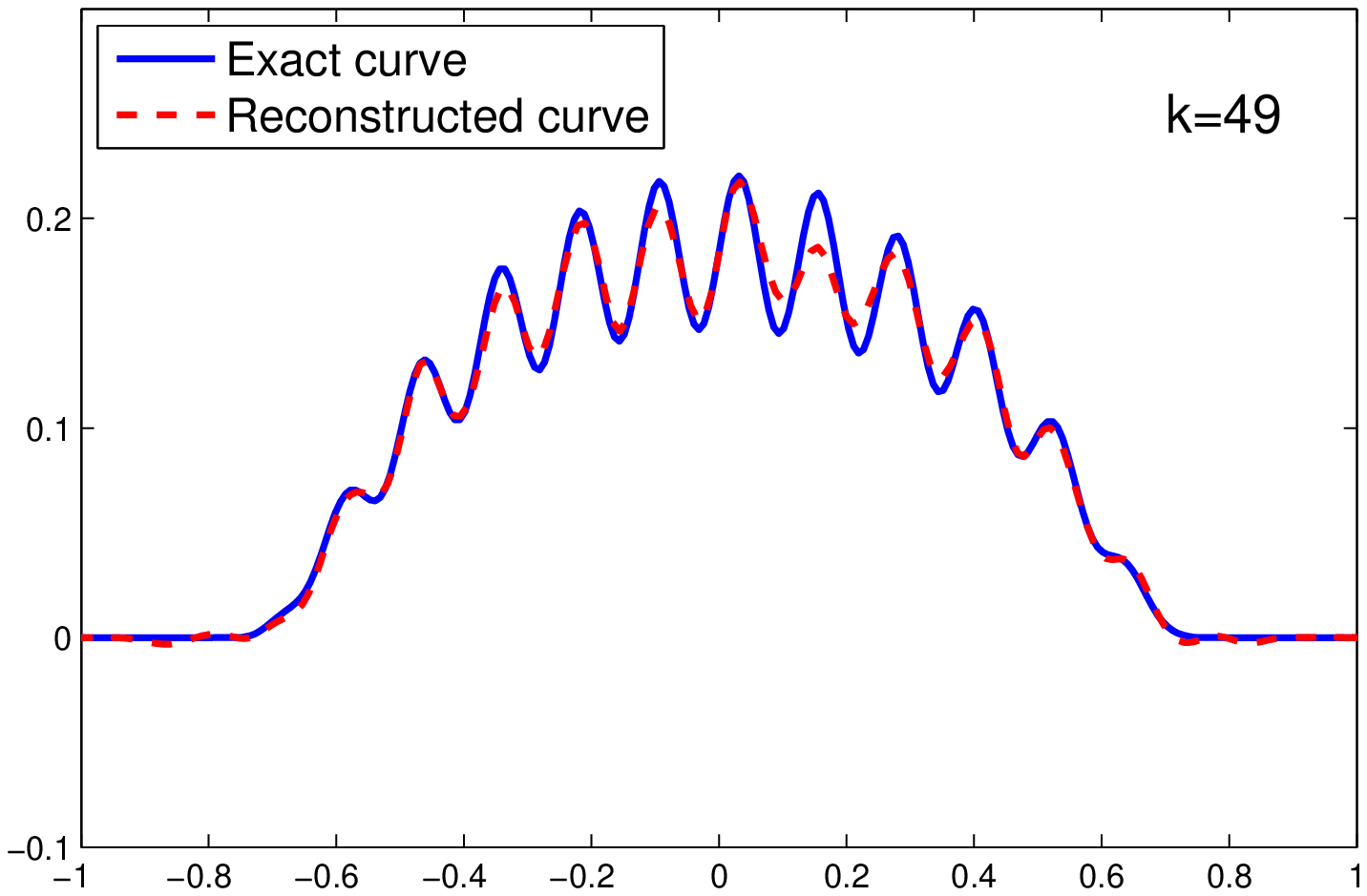}}
  \vspace{0in}
  \hspace{-0.35in}
  \subfigure{\includegraphics[width=3in]{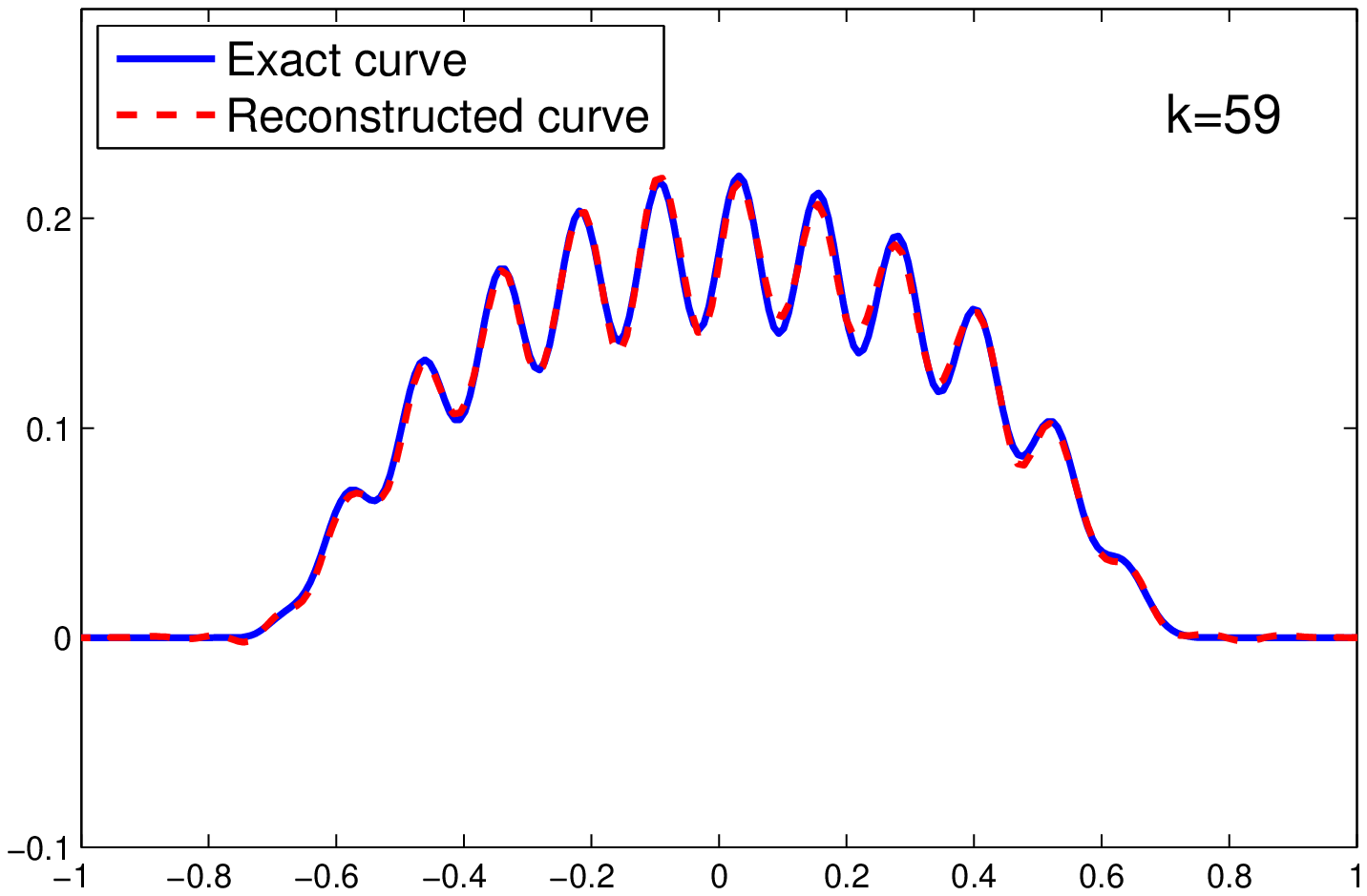}}
\caption{Reconstruction of a two-scale surface profile from $5\%$ noisy intensity far-field data
generated with two superpositions of two plane waves as the incident fields $u^i=u^i(x;d_{1l},d_{2l},k),l=1,2$,
where $d_{11}=(\sin(-\pi/6),-\cos(-\pi/6)),d_{21}=(0,-1),d_{12}=(\sin(\pi/6),-\cos(\pi/6))$ and $d_{22}=(0,-1)$.
Here, the initial and reconstructed curves are presented at the wave numbers $k=19,33,39,49,59$.
}\label{fig4-3}
\end{figure}

\textbf{Example 5 (multi-scale curve).} We consider the inverse problem (IP1) again with the multi-scale surface profile given by
\ben
h_\G(x_1)=\left\{\begin{array}{ll}
\ds \exp\left[16/(25x_1^2-16)\right]\left[0.5+0.1\sin(16\pi x_1)\right]\sin(\pi x_1), &|x_1|<4/5\\
\ds 0, &|x_1|\geq4/5
\end{array}\right.
\enn
This profile consists of a macro-scale represented by $0.5\exp\left[16/(25x_1^2-16)\right]\sin(\pi x_1)$
and a micro-scale represented by $0.1\exp\left[16/(25x_1^2-16)\right]\sin(16\pi x_1)\sin(\pi x_1)$.
For the inverse problem, the number of the spline basis functions is chosen to be $M=40$, the total number of frequencies
used is $N=35$, and the initial guess for $h_\G$ is $0.05\sum_{i=25}^{35}\phi_{i,M}$.
Figure \ref{c5fig4-10} presents the initial and reconstructed curves at $k=19,39,49,59,69$, respectively,
obtained by using $5\%$ noisy intensity far-field data generated with the incident wave $u^i=u^i(x;d_1,d_2,k)$,
where $d_1=(\sin(-\pi/6),-\cos(-\pi/6))$ and $d_2=(\sin(\pi/6),-\cos(\pi/6))$.
From Figure \ref{c5fig4-10} it is observed that the macro-scale features are captured at $k=19$
and the reconstruction improves as the measurement data with higher frequencies are used.
Finally, the whole locally rough surface (including the micro-scale features of the boundary)
is accurately recovered at $k=69$.

\begin{figure}[htbp]
  \centering
  \subfigure{\includegraphics[width=2.8in]{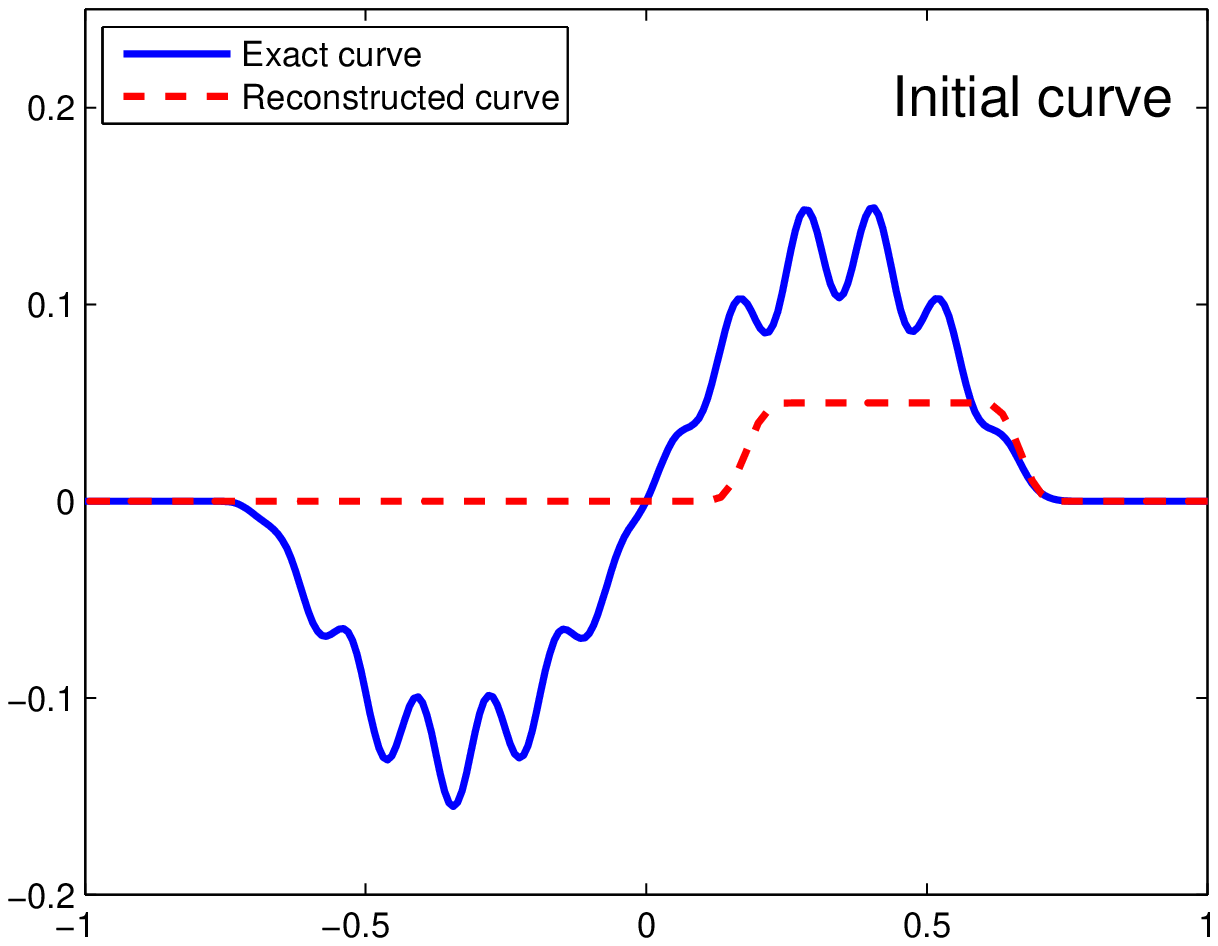}}
  \hspace{-0.35in}
  \vspace{0in}
  \subfigure{\includegraphics[width=2.8in]{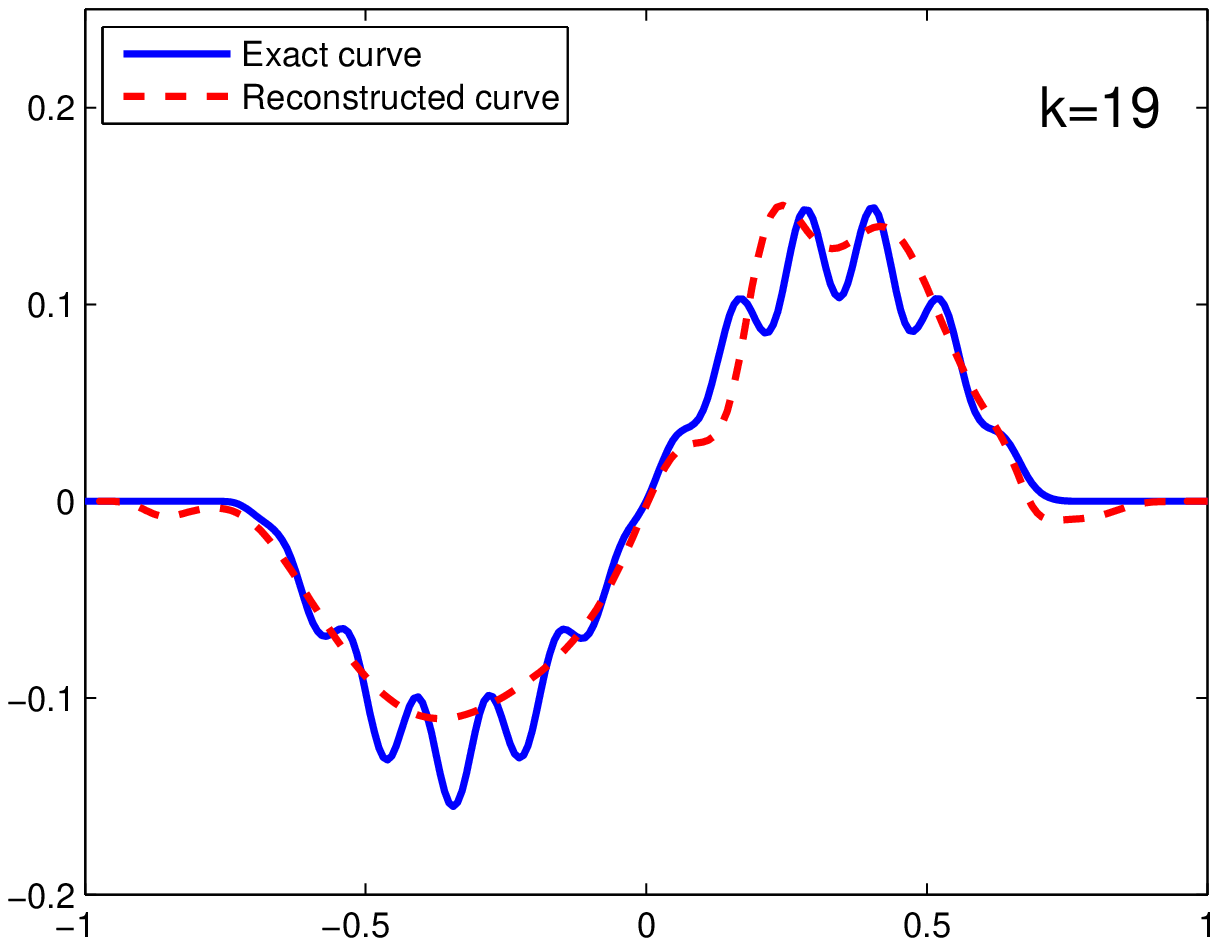}}
  \vspace{0in}
  \hspace{-0.35in}
    \subfigure{\includegraphics[width=2.8in]{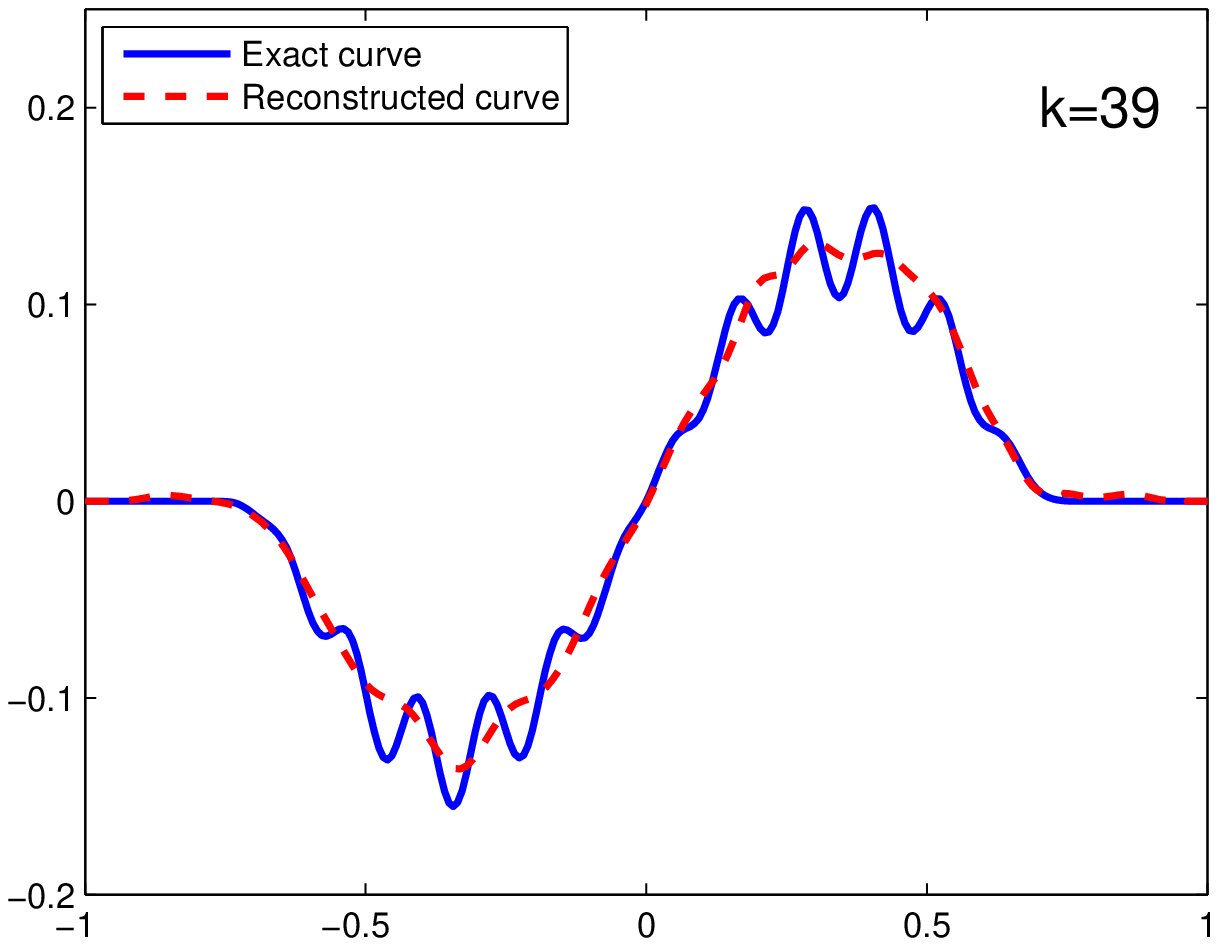}}
  \hspace{-0.35in}
  \vspace{0in}
  \subfigure{\includegraphics[width=2.8in]{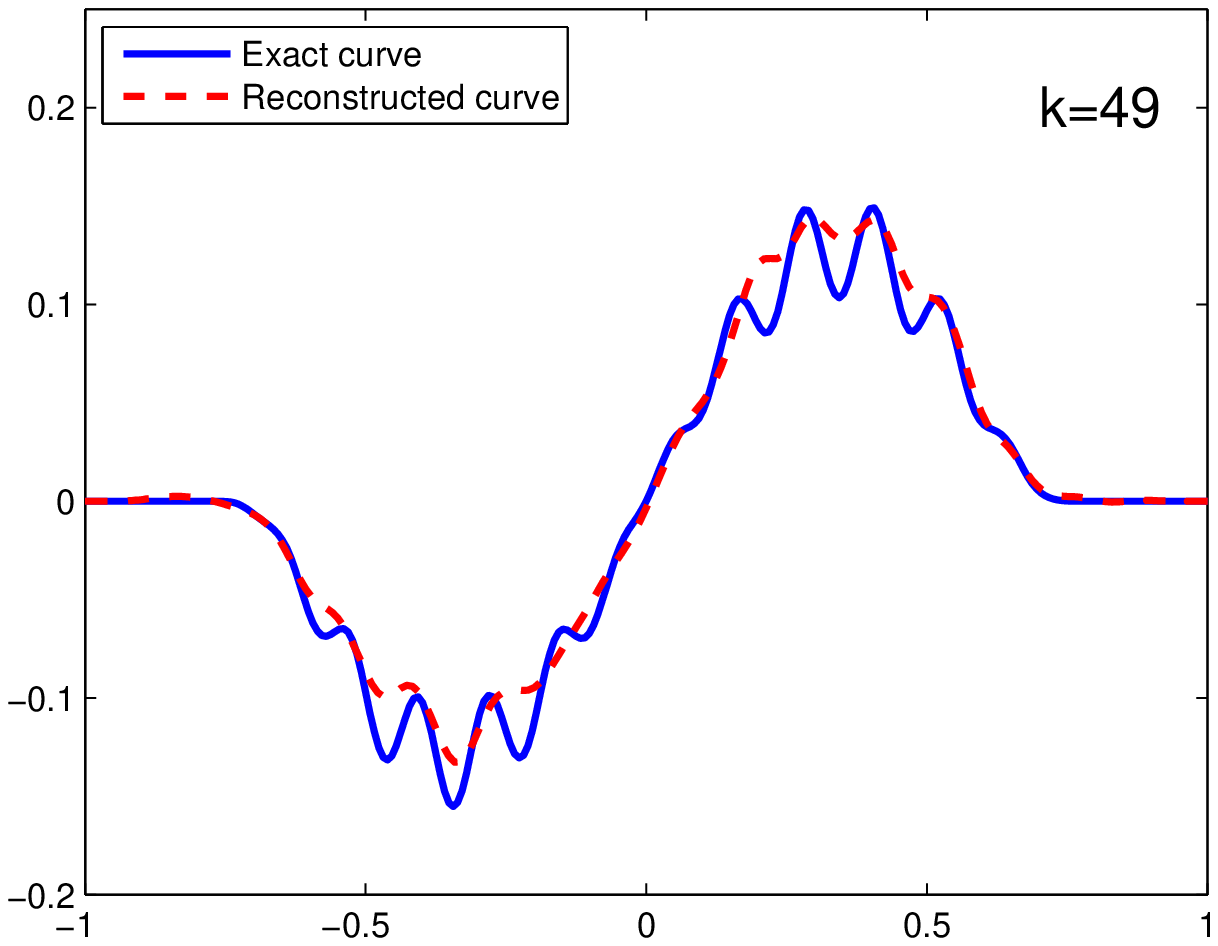}}
  \vspace{0in}
  \hspace{-0.35in}
  \subfigure{\includegraphics[width=2.8in]{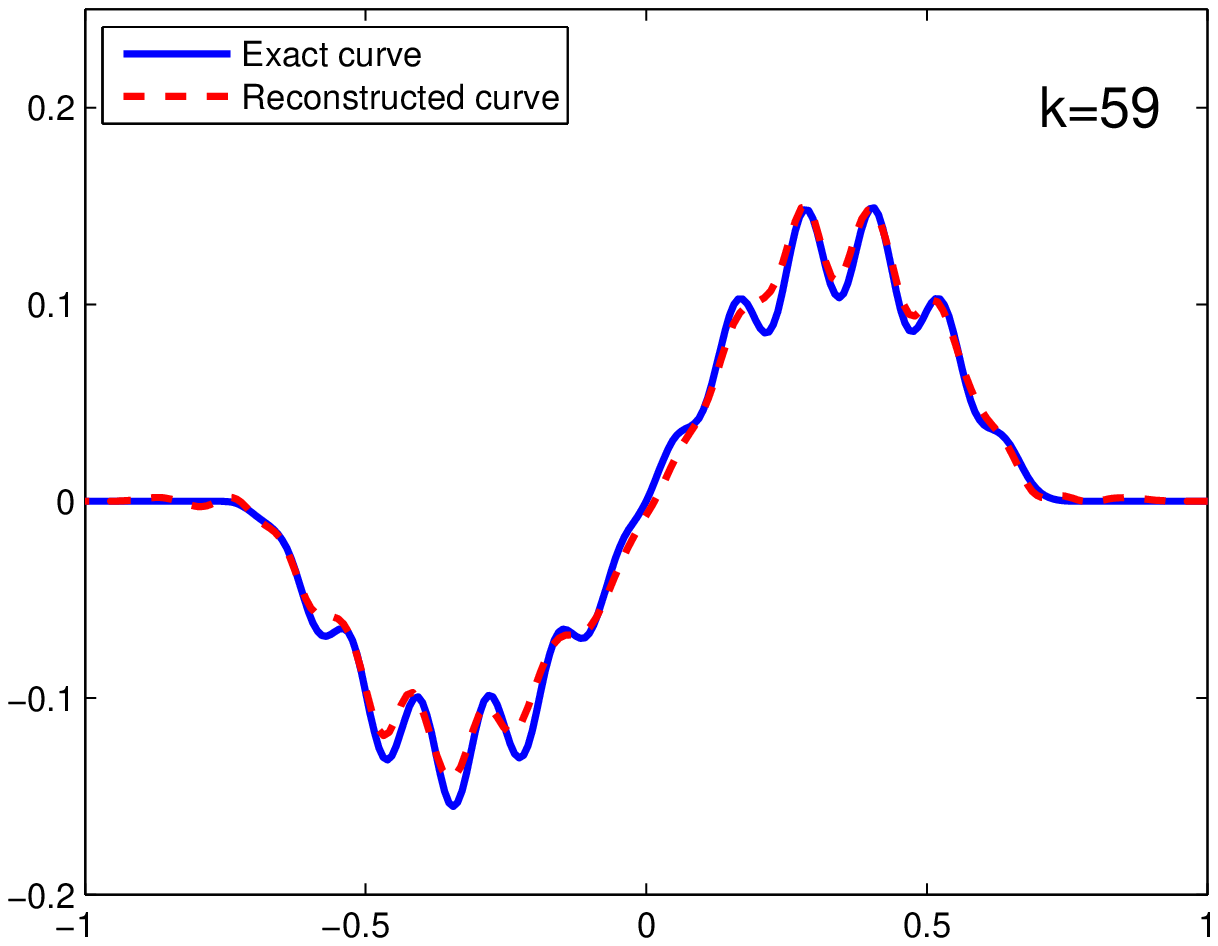}}
  \vspace{0in}
  \hspace{-0.35in}
  \subfigure{\includegraphics[width=2.8in]{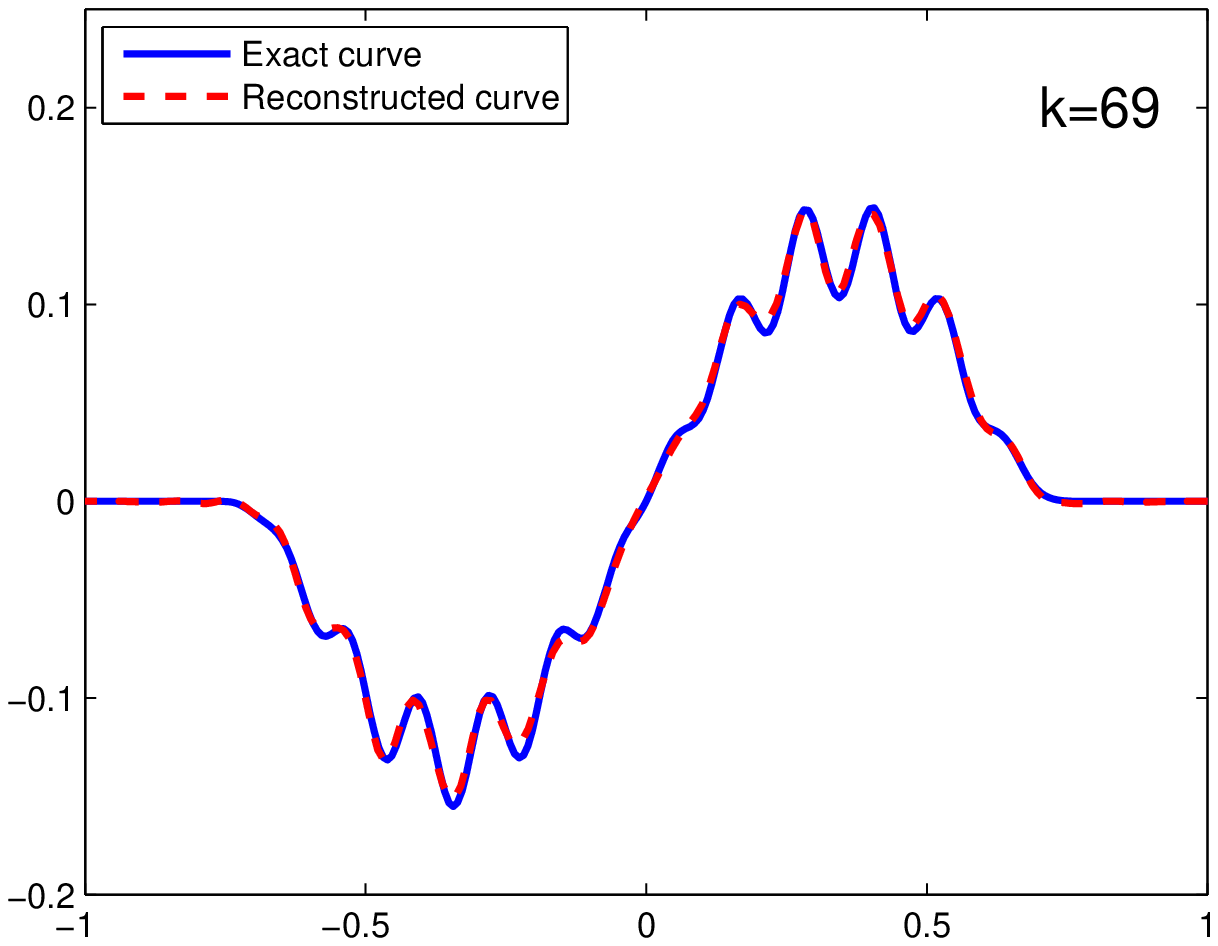}}
  \vspace{0in}
\caption{Reconstruction of a two-scale surface profile from $5\%$ noisy intensity far-field data
generated with a superposition of two plane waves as the incident field $u^i=u^i(x;d_1,d_2,k)$,
where $d_1=(\sin(-\pi/6),-\cos(-\pi/6))$ and $d_2=(\sin(\pi/6),-\cos(\pi/6))$.
Here, the initial and reconstructed curves are presented at the wave numbers $k=19,39,49,59,69$.
}\label{c5fig4-10}
\end{figure}

\textbf{Example 6 (piecewise linear curve).} We now consider the inverse problem (IP2) with the locally rough surface
given as in Example 3. For the inverse problem, the number of the spline basis functions is chosen to be $M=40$,
the total number of frequencies is assumed to be $N=18$, and the initial guess for the reconstructed curve is taken as
the infinite plane $x_2=0$. In Figure \ref{fig4-7}, we present the initial curve and the reconstructed curves at $k=7,15,31,$
respectively, which are obtained from the $5\%$ intensity near-field data generated with the incident wave $u^i=u^i(x;d,k)$,
where $d=(\sin(-\pi/6),-\cos(-\pi/6))$. It can be seen that the reconstruction at the illuminated part of the boundary is
better than that at the shadowed part. Thus, in order to improve the reconstruction, more measurement data are needed.
Figure \ref{fig4-4} presents the initial curve and the reconstructed curves at $k=7,15,35,$ respectively, obtained by using
the $5\%$ intensity near-field data corresponding to two incident waves $u^i=u^i(x;d_l,k),l=1,2$,
where $d_1=(\sin(-\pi/6),-\cos(-\pi/6))$ and $d_2=(\sin(\pi/6),-\cos(\pi/6))$.
It is found that the piecewise linear surface profile is accurately reconstructed even at the corners of the surface by
using two incident plane waves with different directions.

\begin{figure}[htbp]
  \centering
  \subfigure{\includegraphics[width=3in]{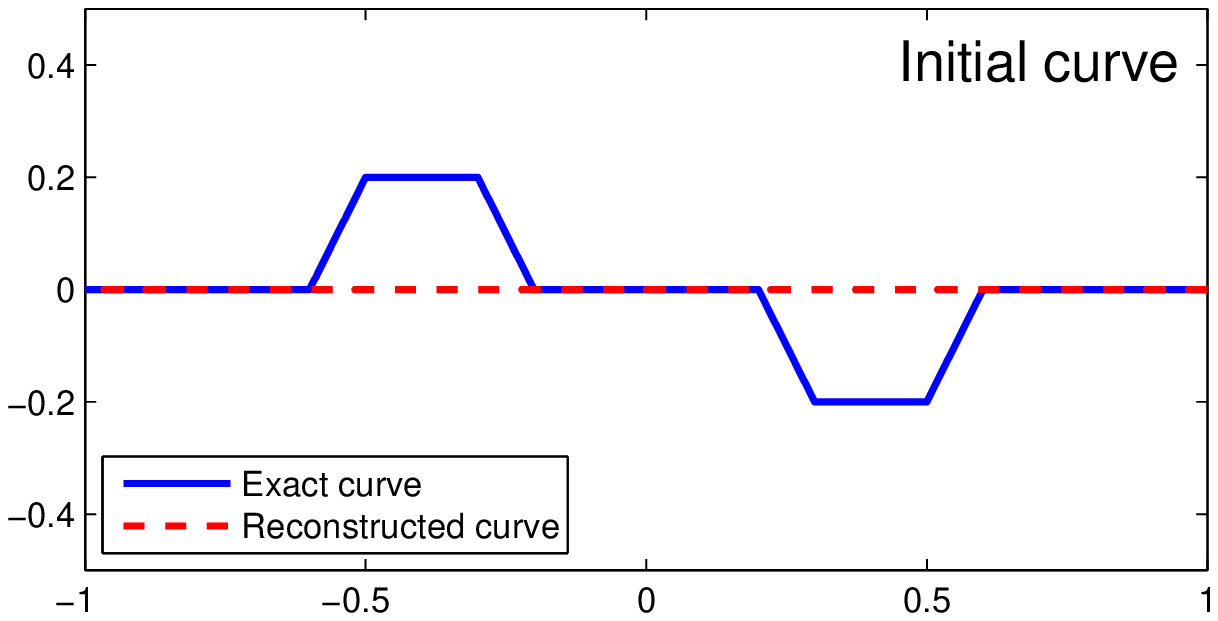}}
  \hspace{-0.35in}
  \vspace{-0.5in}
  \subfigure{\includegraphics[width=3in]{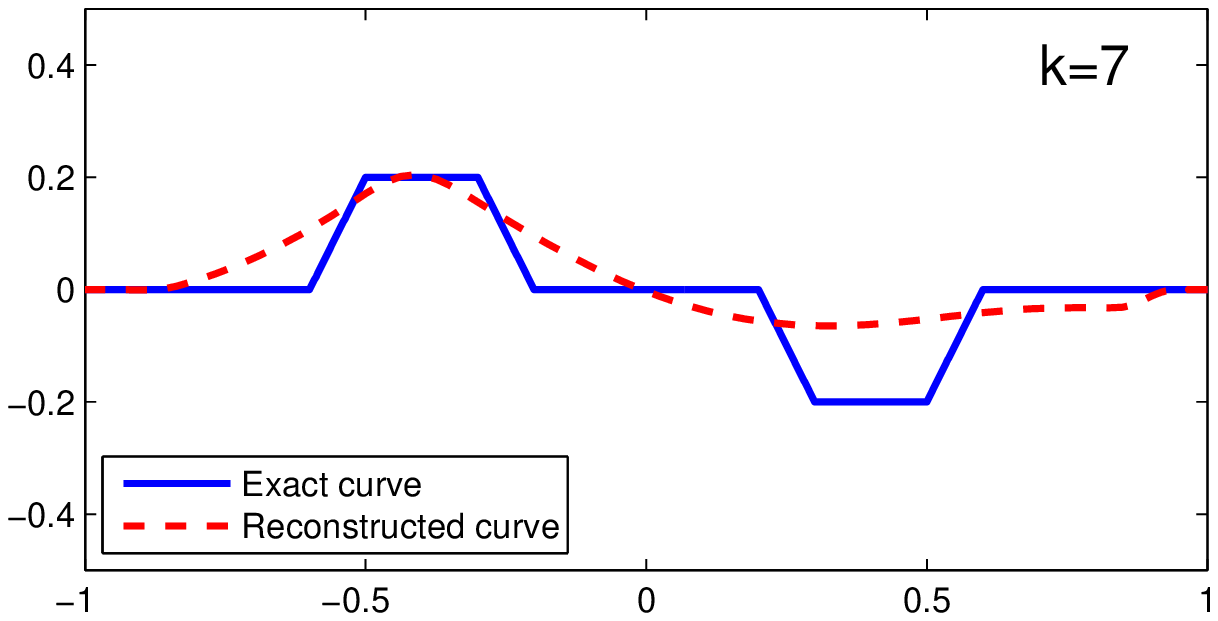}}
  \vspace{0in}
  \hspace{-0.35in}
  \subfigure{\includegraphics[width=3in]{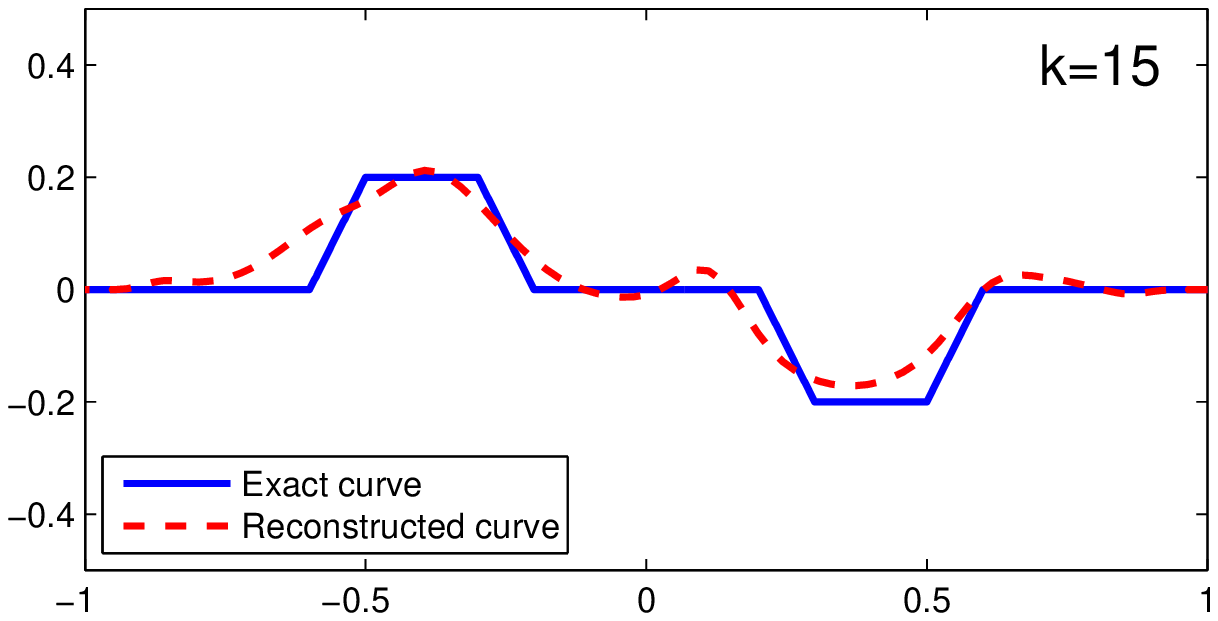}}
  \vspace{0in}
  \hspace{-0.35in}
  \subfigure{\includegraphics[width=3in]{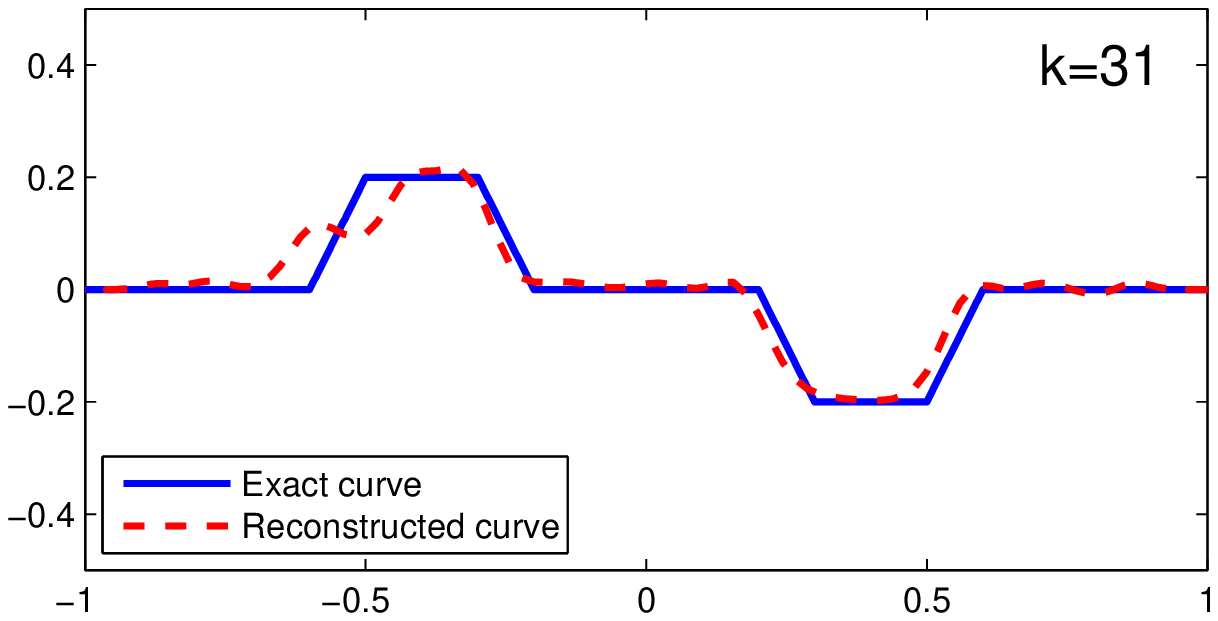}}
  \vspace{-0.5in}
\caption{Reconstruction of a piecewise linear surface profile from $5\%$ noisy intensity near-field data
generated with one incident plane wave $u^i=u^i(x;d,k)$, where $d=(\sin(-\pi/6),-\cos(-\pi/6))$.
Here, the initial and reconstructed curves are presented at the wavenumbers $k=7,15,31,$ respectively.
}\label{fig4-7}
\end{figure}

\begin{figure}[htbp]
  \centering
  \subfigure{\includegraphics[width=3in]{example/IP2/case9/initial_curve.eps}}
  \hspace{-0.35in}
  \vspace{-0.5in}
  \subfigure{\includegraphics[width=3in]{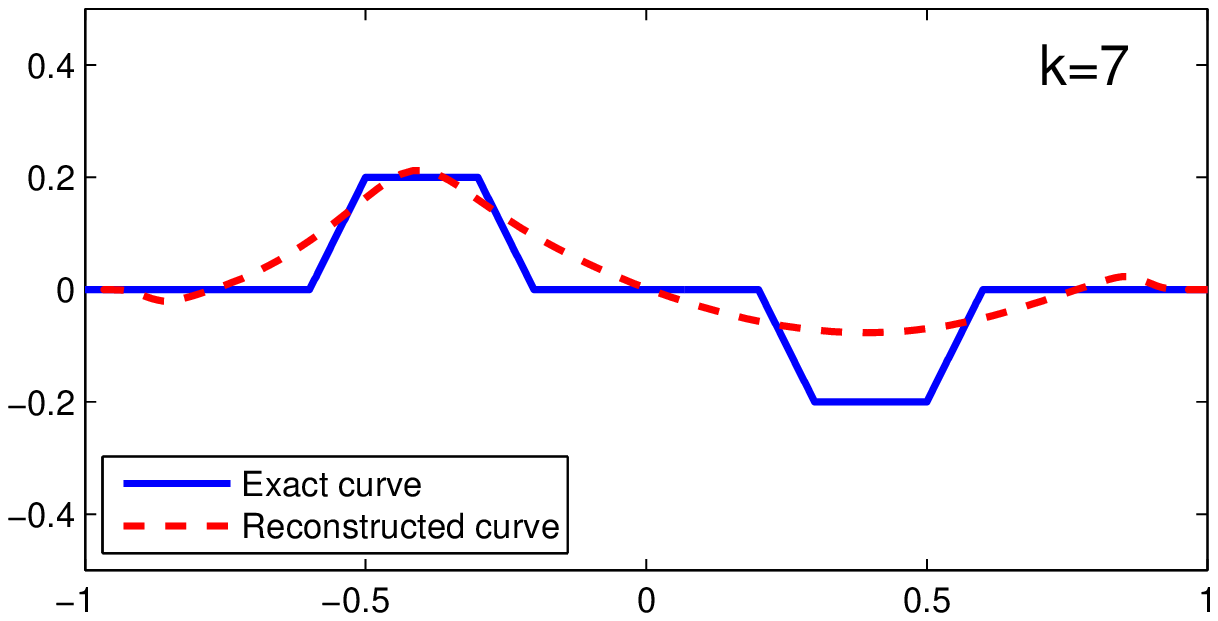}}
  \vspace{0in}
  \hspace{-0.35in}
  \subfigure{\includegraphics[width=3in]{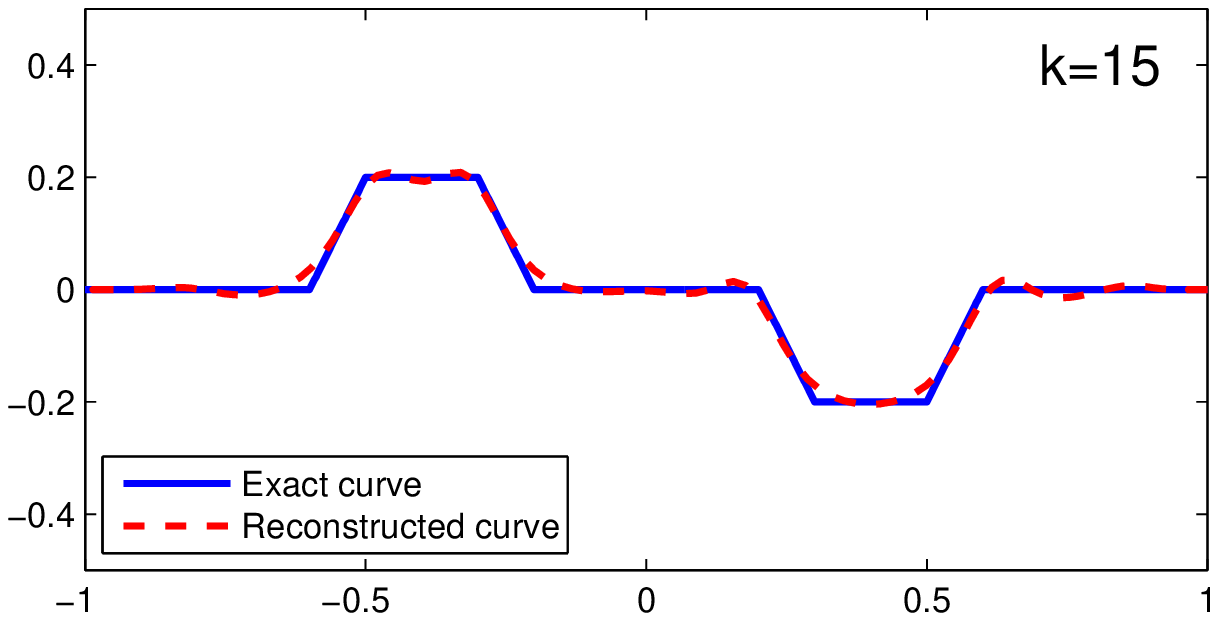}}
  \vspace{0in}
  \hspace{-0.35in}
  \subfigure{\includegraphics[width=3in]{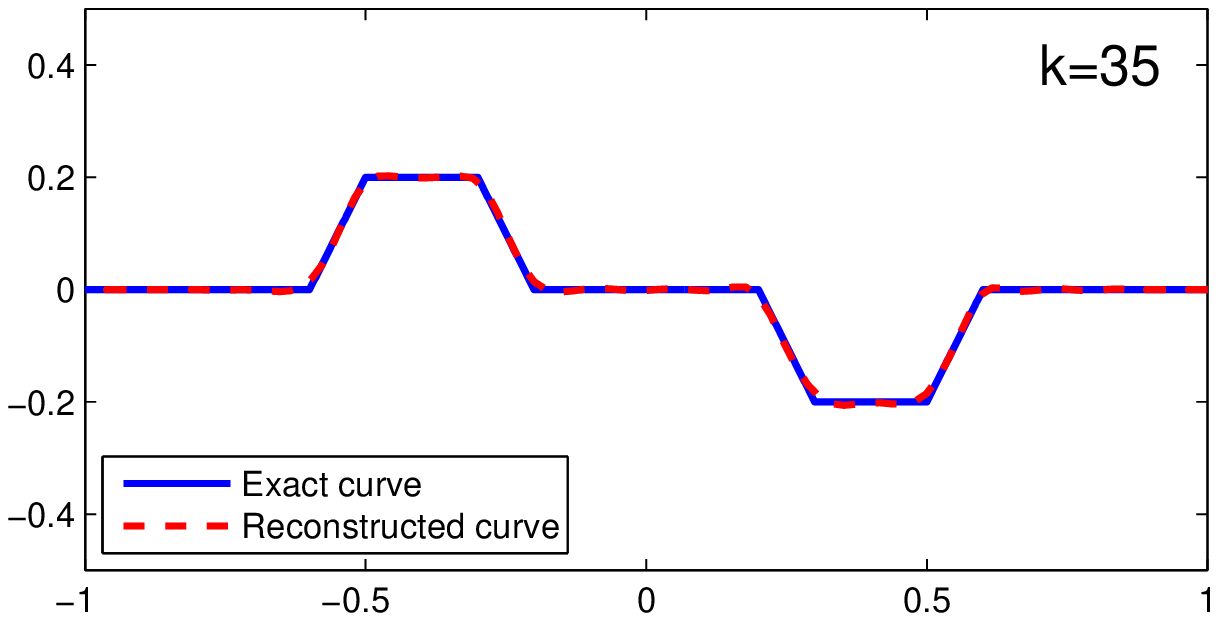}}
  \vspace{-0.5in}
\caption{Reconstruction of a piecewise linear surface profile from $5\%$ noisy intensity near-field data
generated with two incident plane waves $u^i=u^i(x;d_l,k),l=1,2$, where $d_1=(\sin(-\pi/6),-\cos(-\pi/6))$
and $d_2=(\sin(\pi/6),-\cos(\pi/6))$.
Here, the initial and reconstructed curves are presented at the wave numbers $k=7,15,35$.
}\label{fig4-4}
\end{figure}

\textbf{Example 7 (multi-scale curve).} We consider the inverse problem (IP2) again with the multi-scale surface profile
given as in Example 5. For the inverse problem, the number of the spline basis functions is chosen to be $M=40$,
the total number of frequencies is set to be $N=30$, and the initial guess for the reconstructed curve is taken as
the infinite plane $x_2=0$. In Figure \ref{fig4-8},
we present the initial curve and the reconstructed curves at $k=13,29,39,49,59$,
obtained from the $5\%$ noisy intensity near-field data generated with one incident plane wave $u^i=u^i(x;d,k)$,
where $d=(0,-1)$. It is observed that the micro-scale of the boundary surface is not recovered accurately.
Figure \ref{fig4-5} presents the initial curve and the reconstructed curves at $k=13,29,39,49,59,$ respectively,
obtained from the $5\%$ noisy intensity near-field data corresponding to two incident plane waves $u^i=u^i(x;d_l,k),l=1,2$,
with two different directions $d_1=(\sin(-\pi/6),-\cos(-\pi/6))$ and $d_2=(\sin(\pi/6),-\cos(\pi/6))$.
Compared with Figure \ref{fig4-8} it can be seen that the multi-scale surface profile can be accurately recovered
by using two incident plane waves with different directions.

\begin{figure}[htbp]
  \centering
  \subfigure{\includegraphics[width=3in]{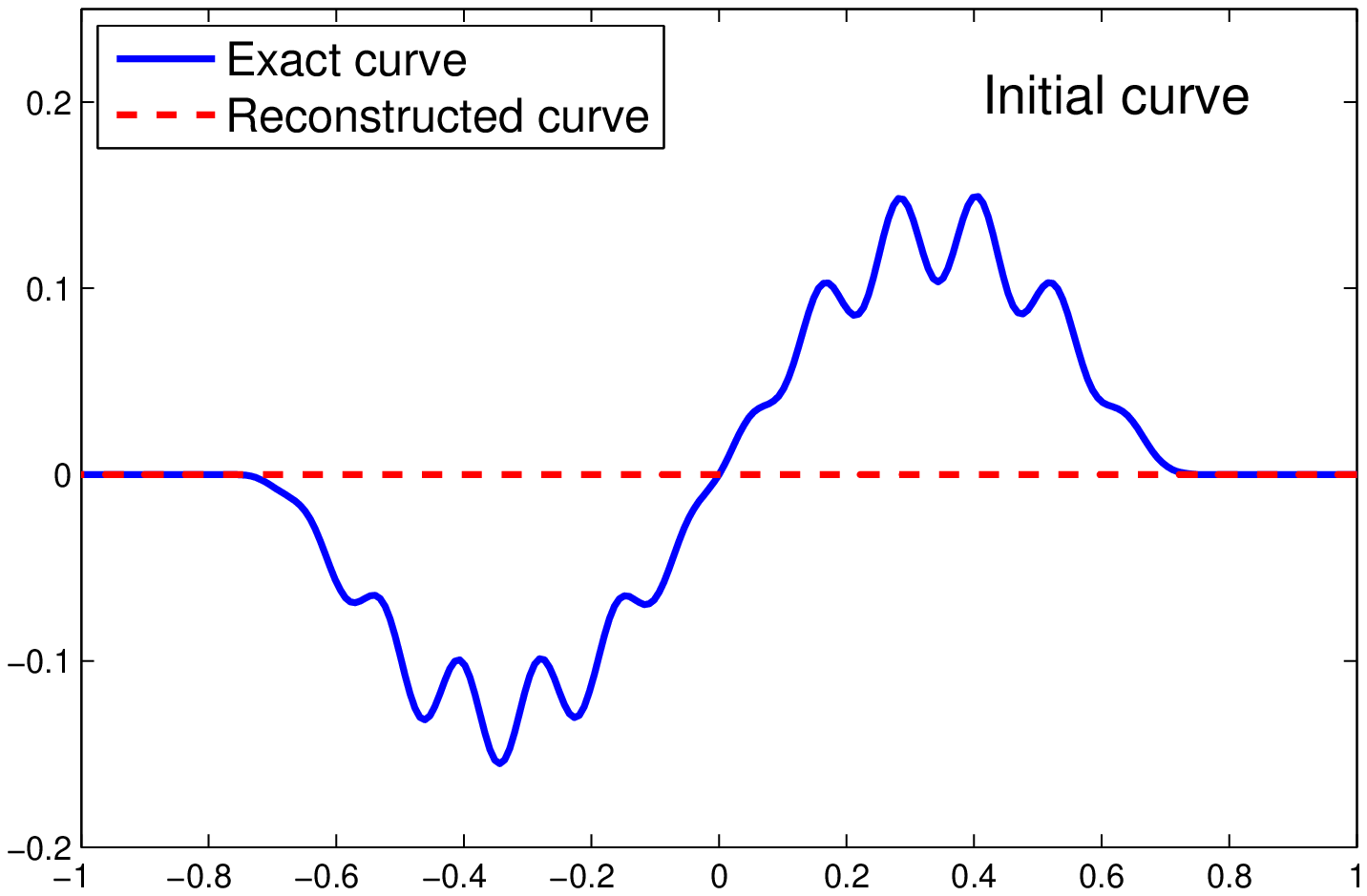}}
  \hspace{-0.35in}
  \vspace{0in}
  \subfigure{\includegraphics[width=3in]{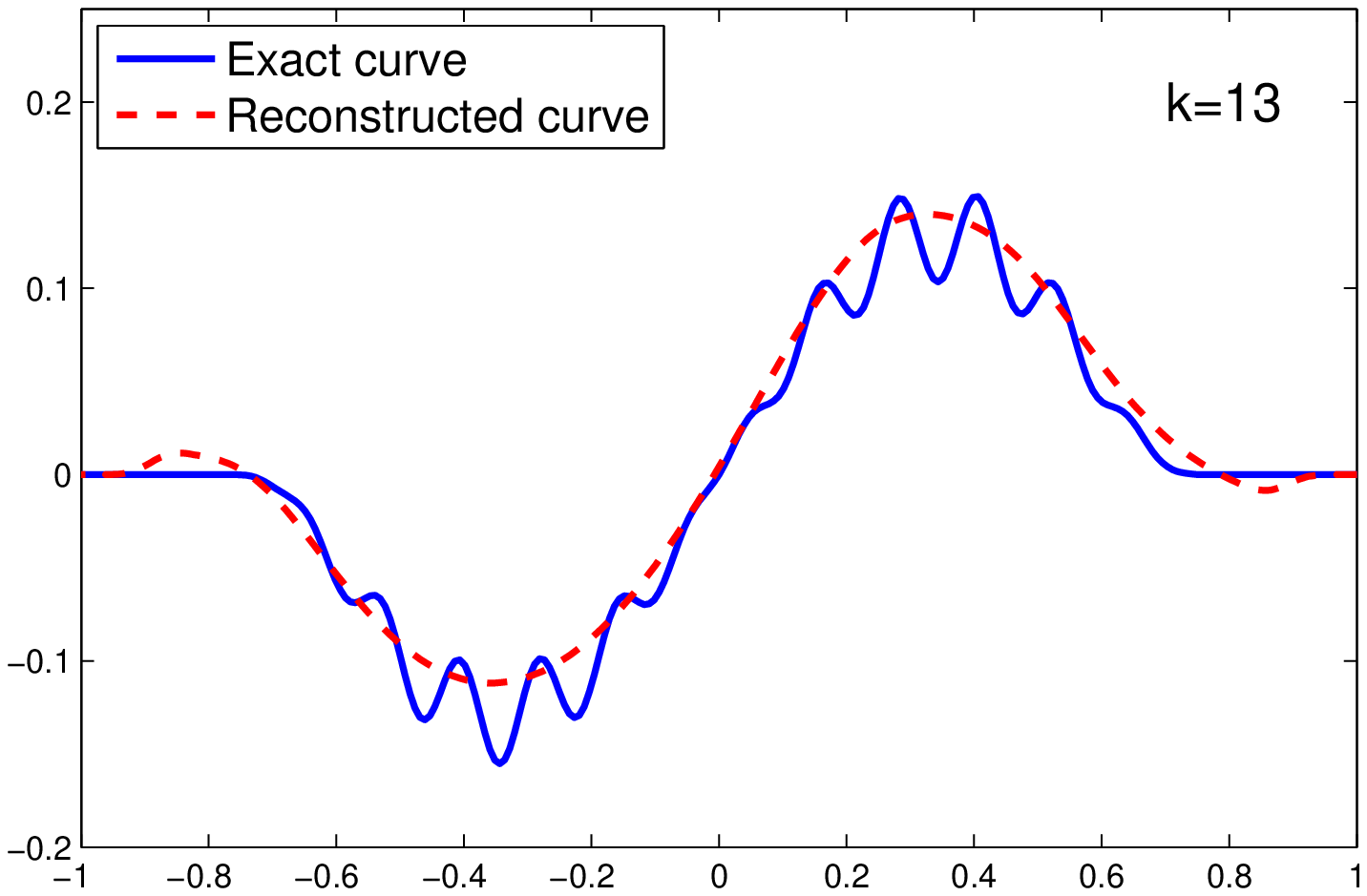}}
  \vspace{0in}
  \hspace{-0.35in}
  \subfigure{\includegraphics[width=3in]{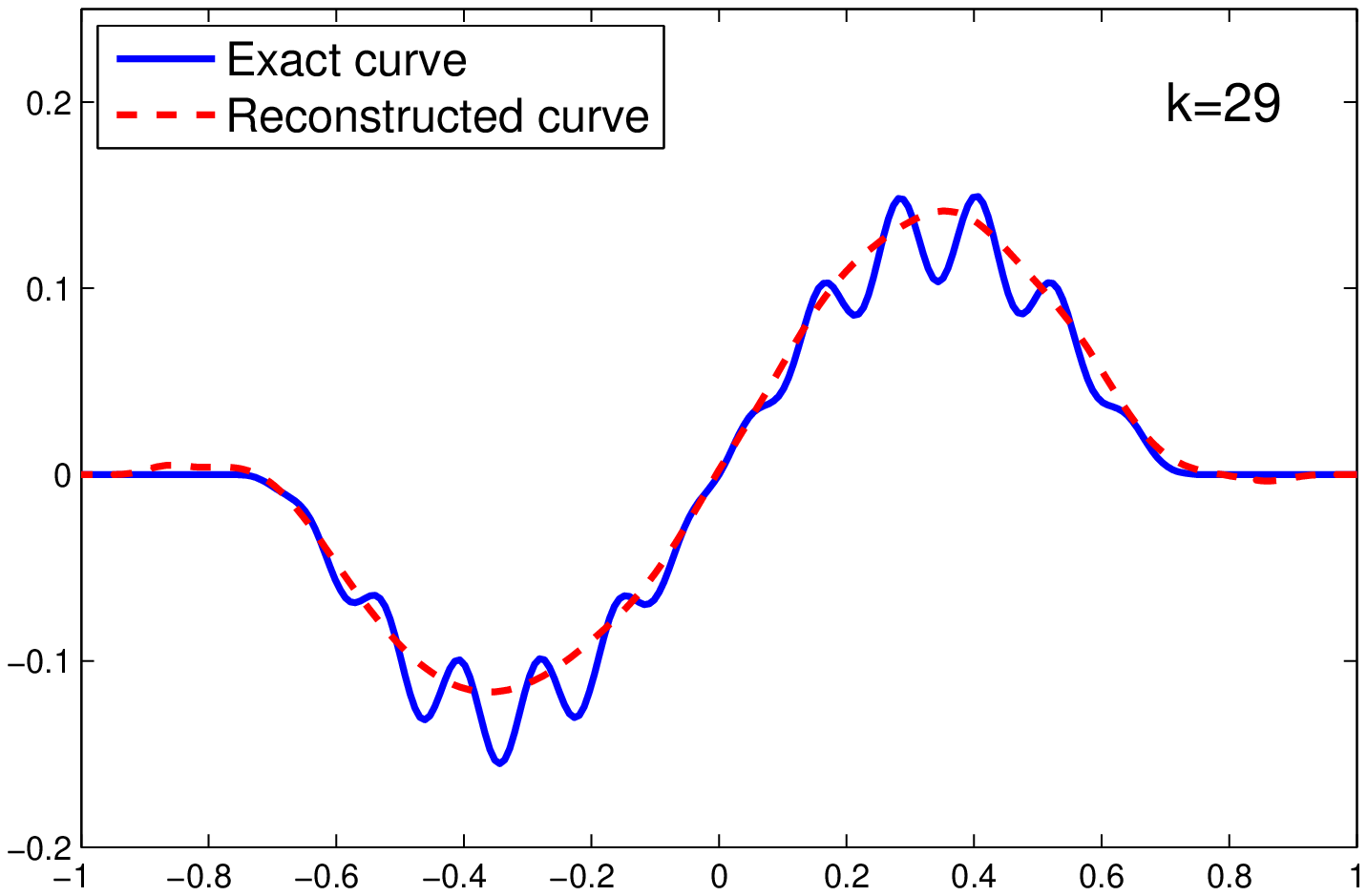}}
  \hspace{-0.35in}
  \vspace{0in}
  \subfigure{\includegraphics[width=3in]{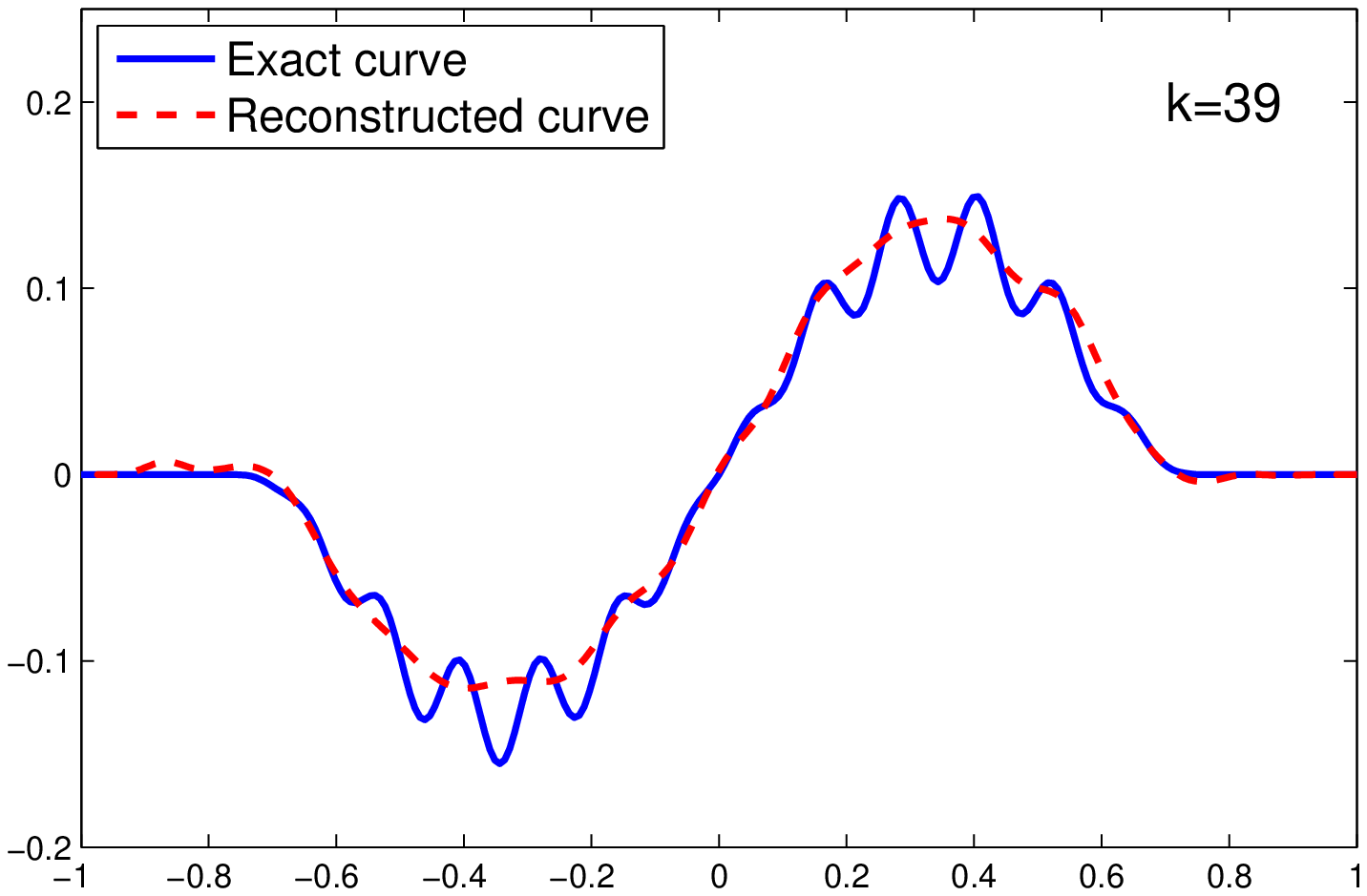}}
  \vspace{0in}
  \hspace{-0.35in}
  \subfigure{\includegraphics[width=3in]{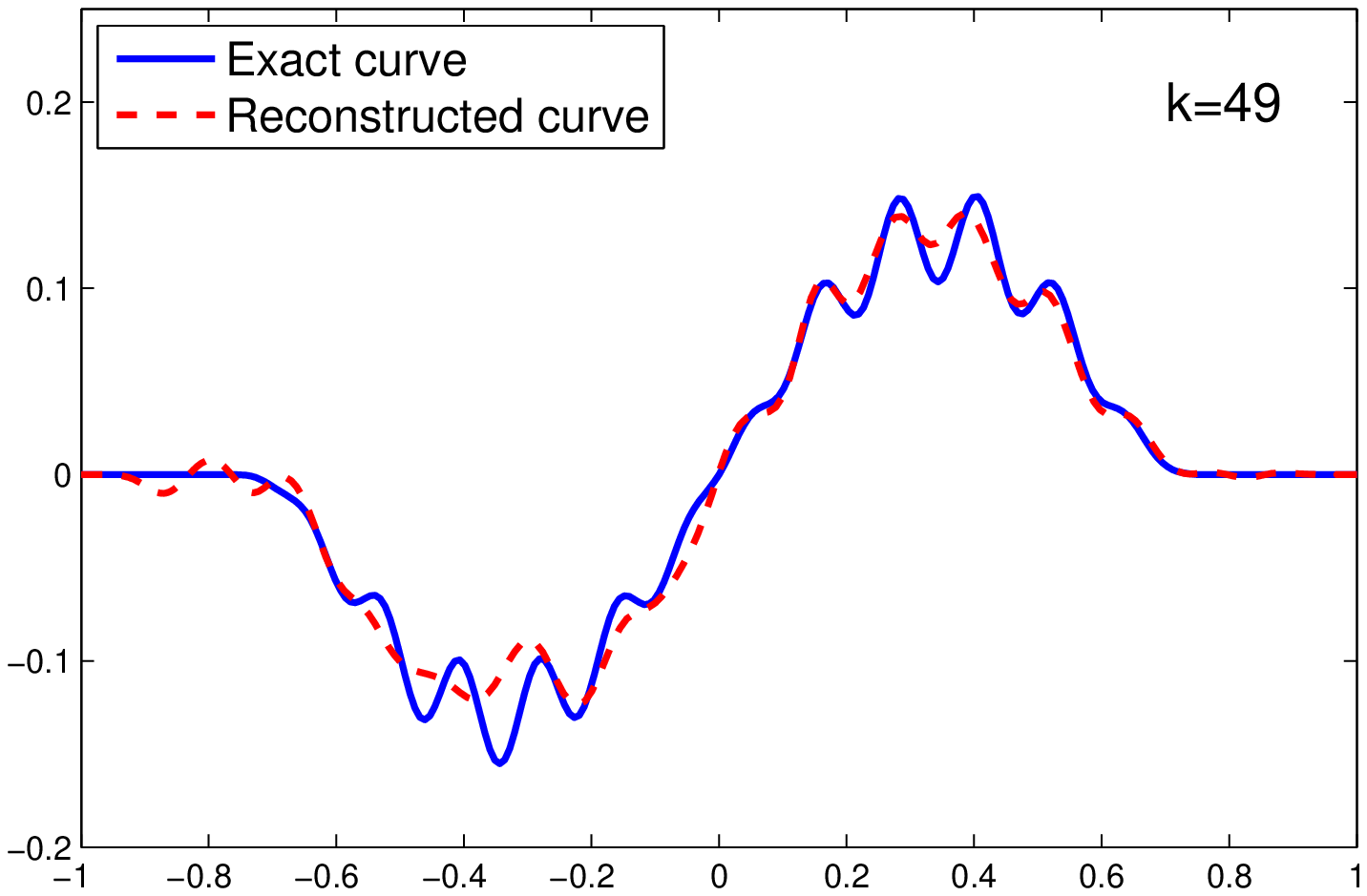}}
  \vspace{0in}
  \hspace{-0.35in}
  \subfigure{\includegraphics[width=3in]{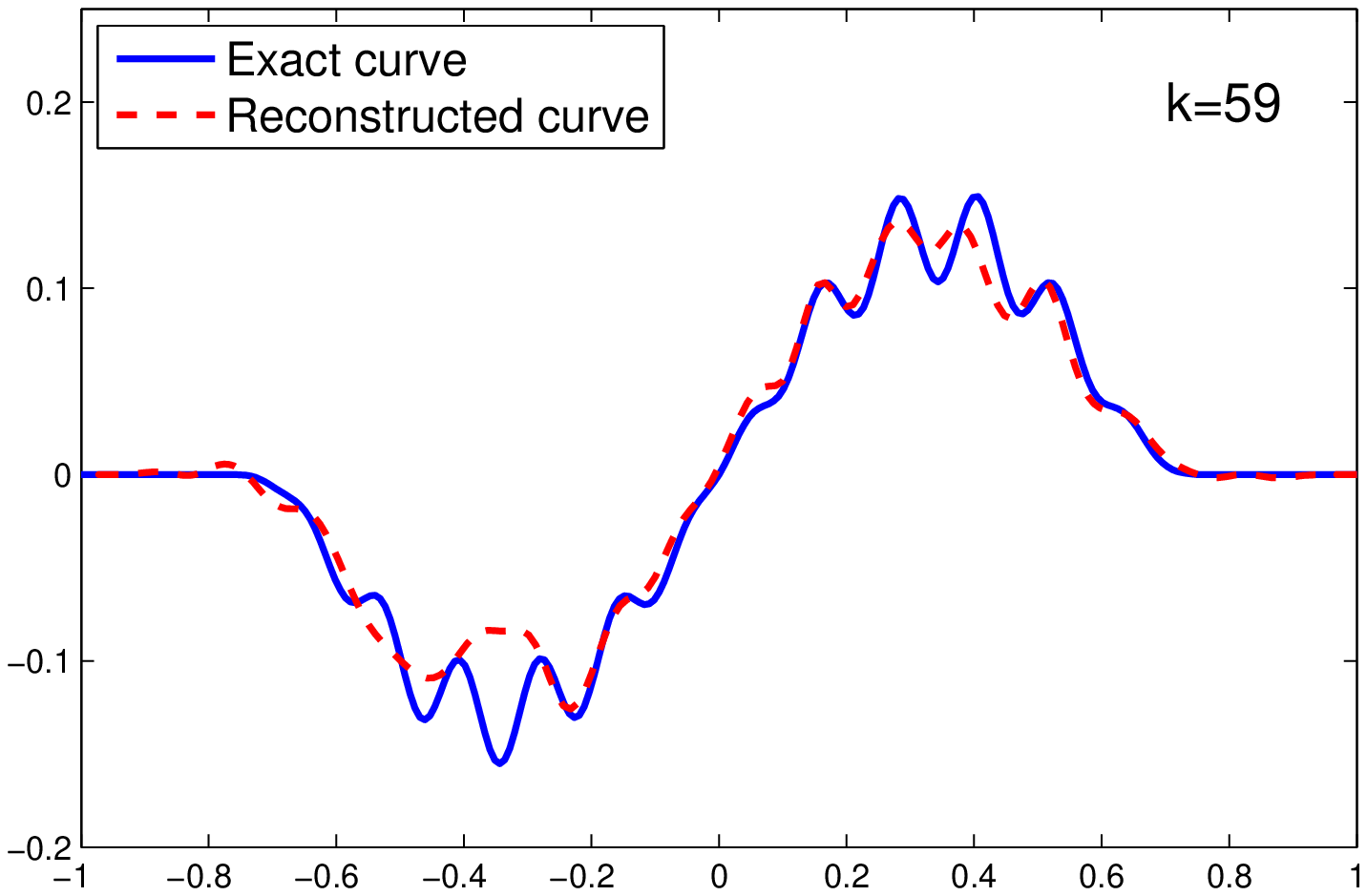}}
  \vspace{0in}
\caption{Reconstruction of a multi-scale surface profile from $5\%$ noisy intensity near-field data
generated by one incident plane wave $u^i=u^i(x;d,k)$ with a normal incidence ($d=(0,-1)$).
Here, the initial and reconstructed curves are presented at $k=13,29,39,49,59$.
}\label{fig4-8}
\end{figure}

\begin{figure}[htbp]
  \centering
  \subfigure{\includegraphics[width=3in]{example/IP2/case7/initial_curve.eps}}
  \hspace{-0.35in}
  \vspace{0in}
  \subfigure{\includegraphics[width=3in]{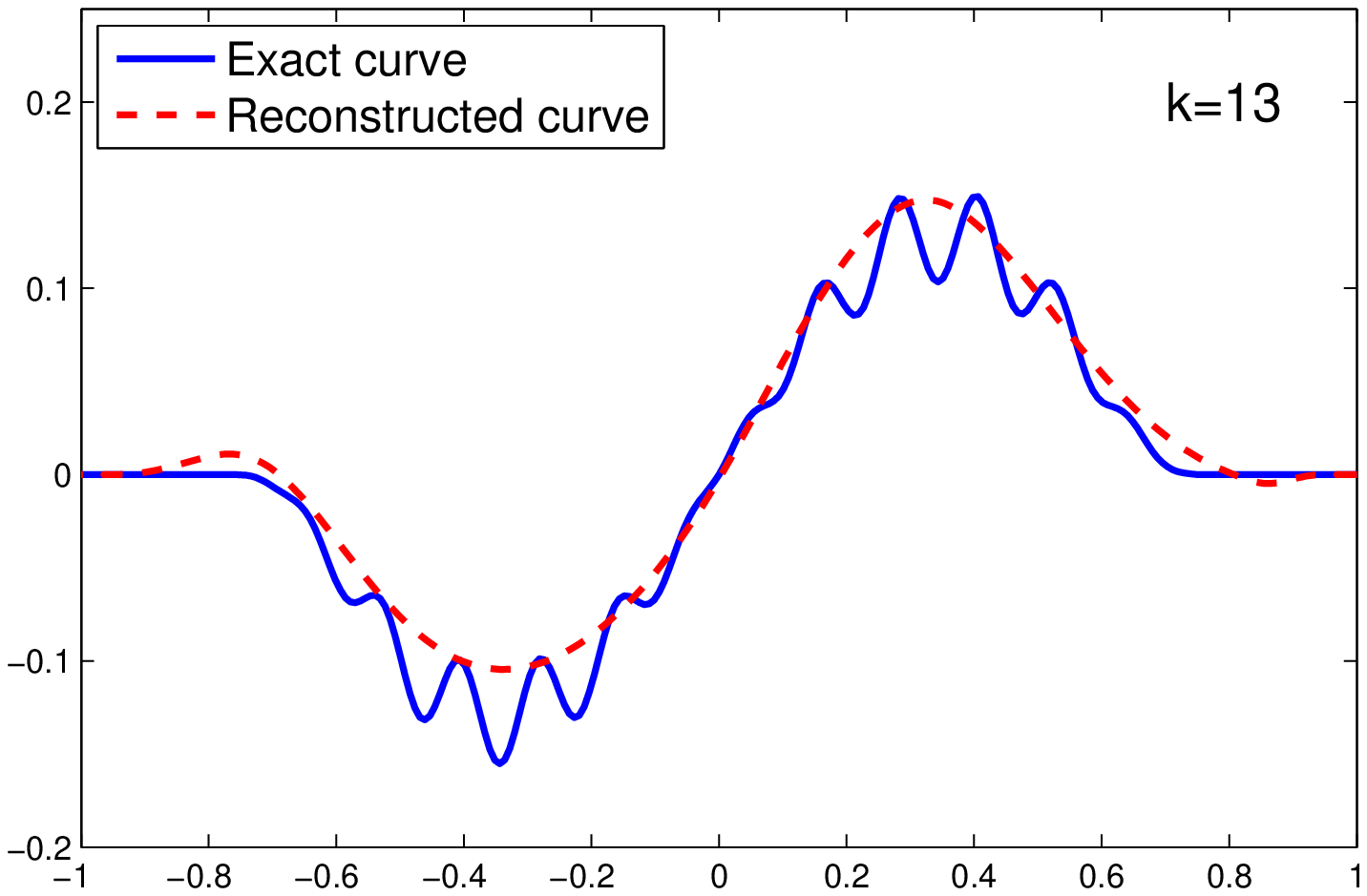}}
  \vspace{0in}
  \hspace{-0.35in}
  \subfigure{\includegraphics[width=3in]{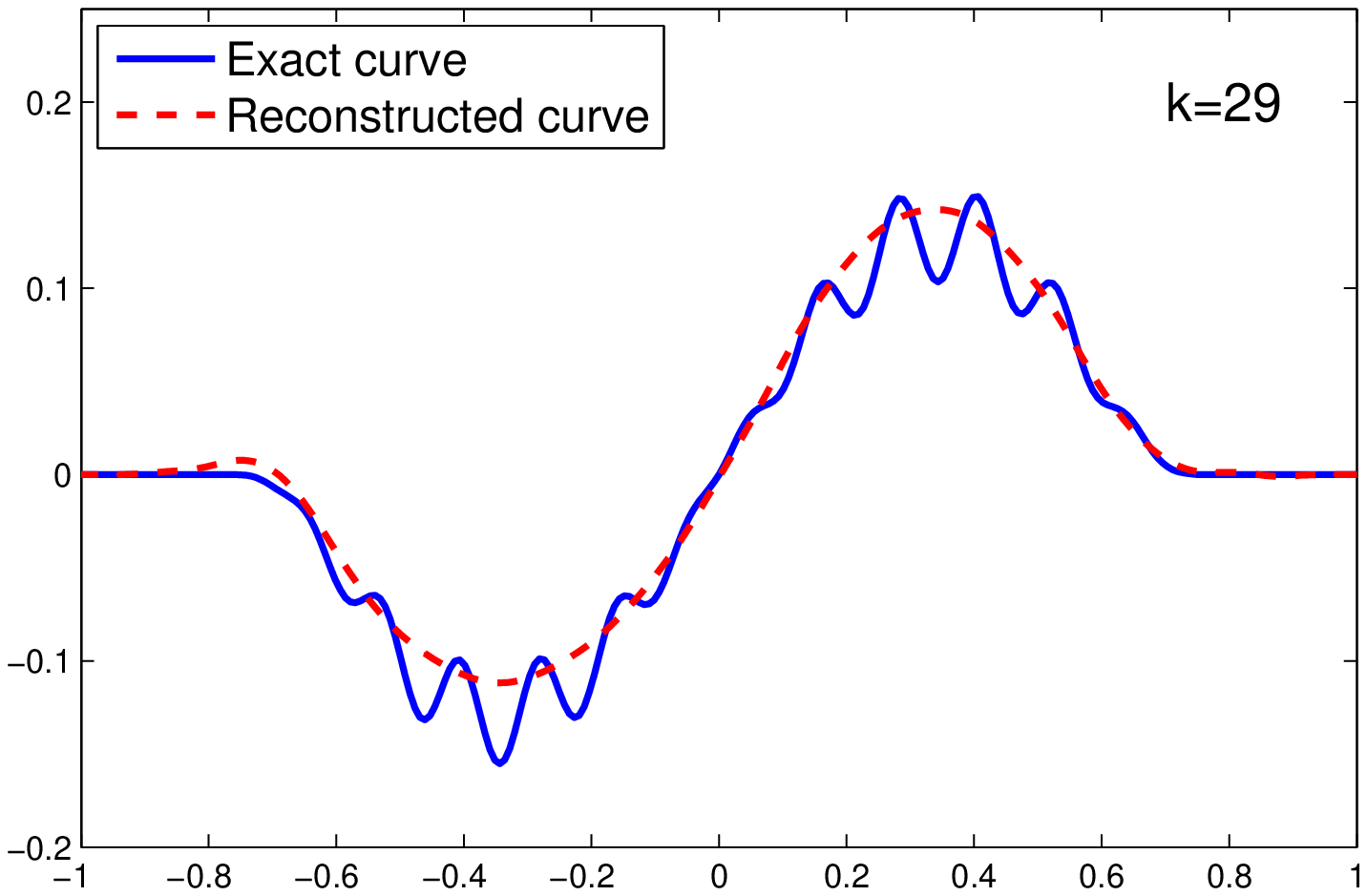}}
  \hspace{-0.35in}
  \vspace{0in}
  \subfigure{\includegraphics[width=3in]{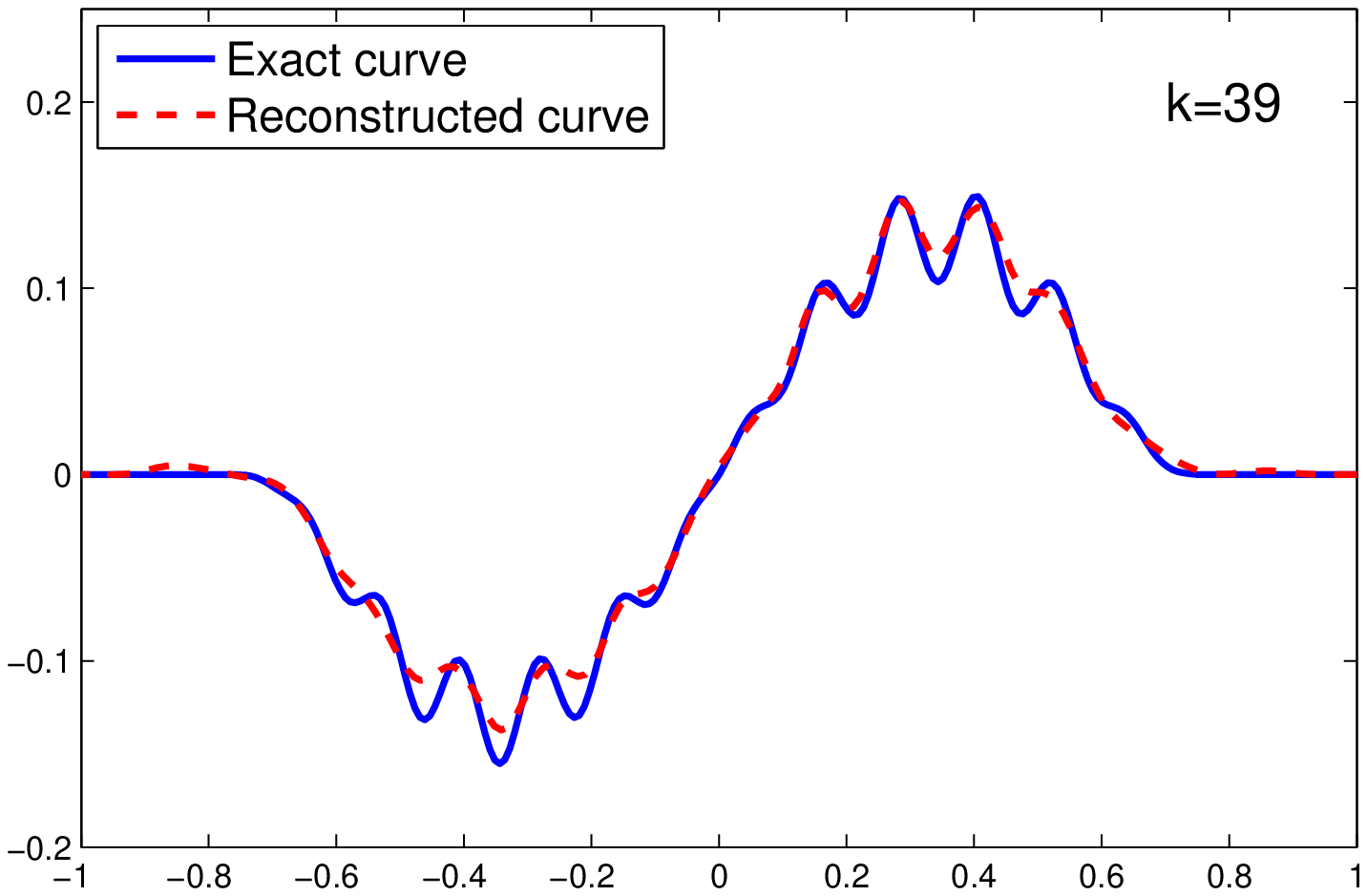}}
  \vspace{0in}
  \hspace{-0.35in}
  \subfigure{\includegraphics[width=3in]{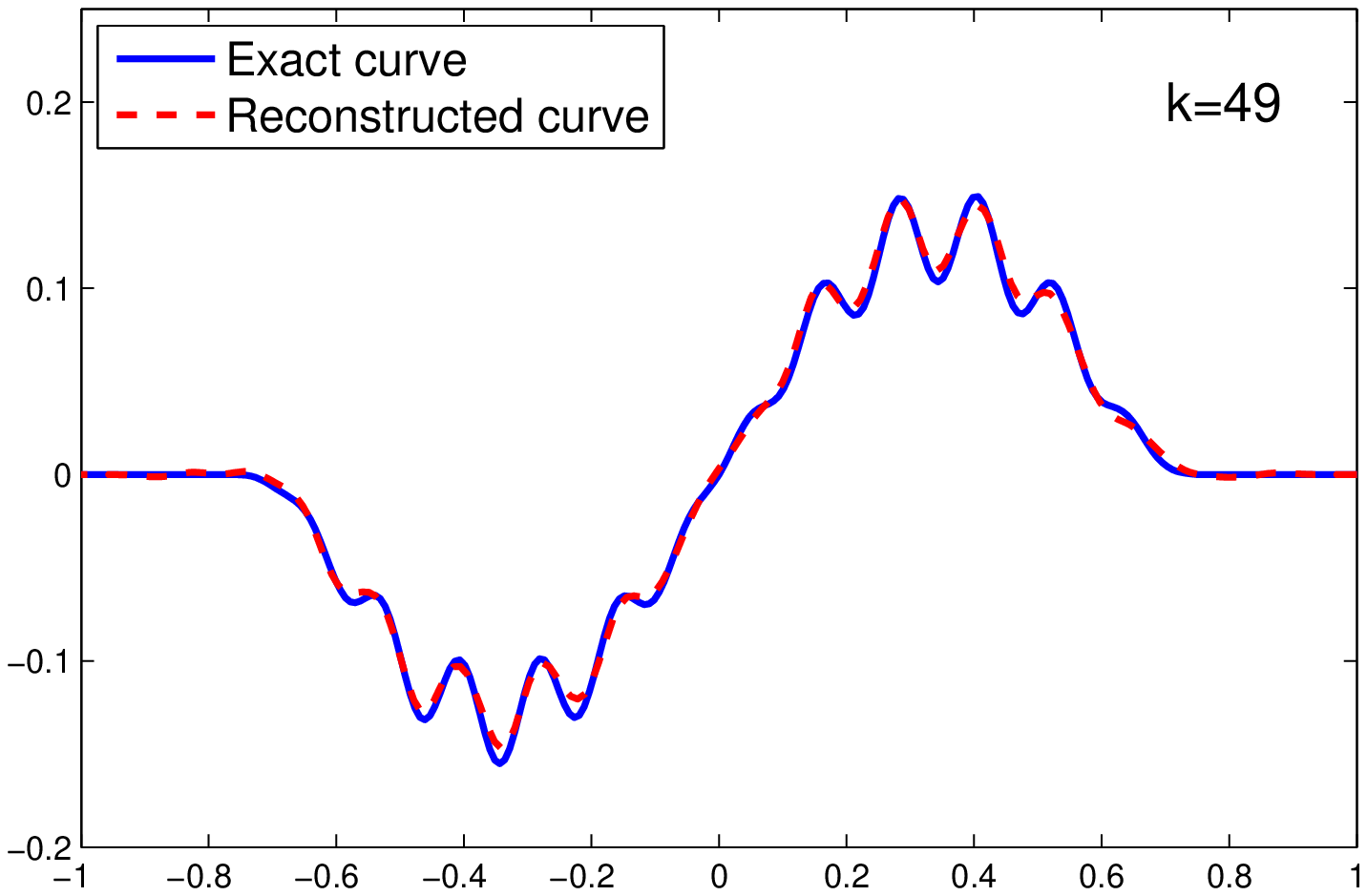}}
  \vspace{0in}
  \hspace{-0.35in}
  \subfigure{\includegraphics[width=3in]{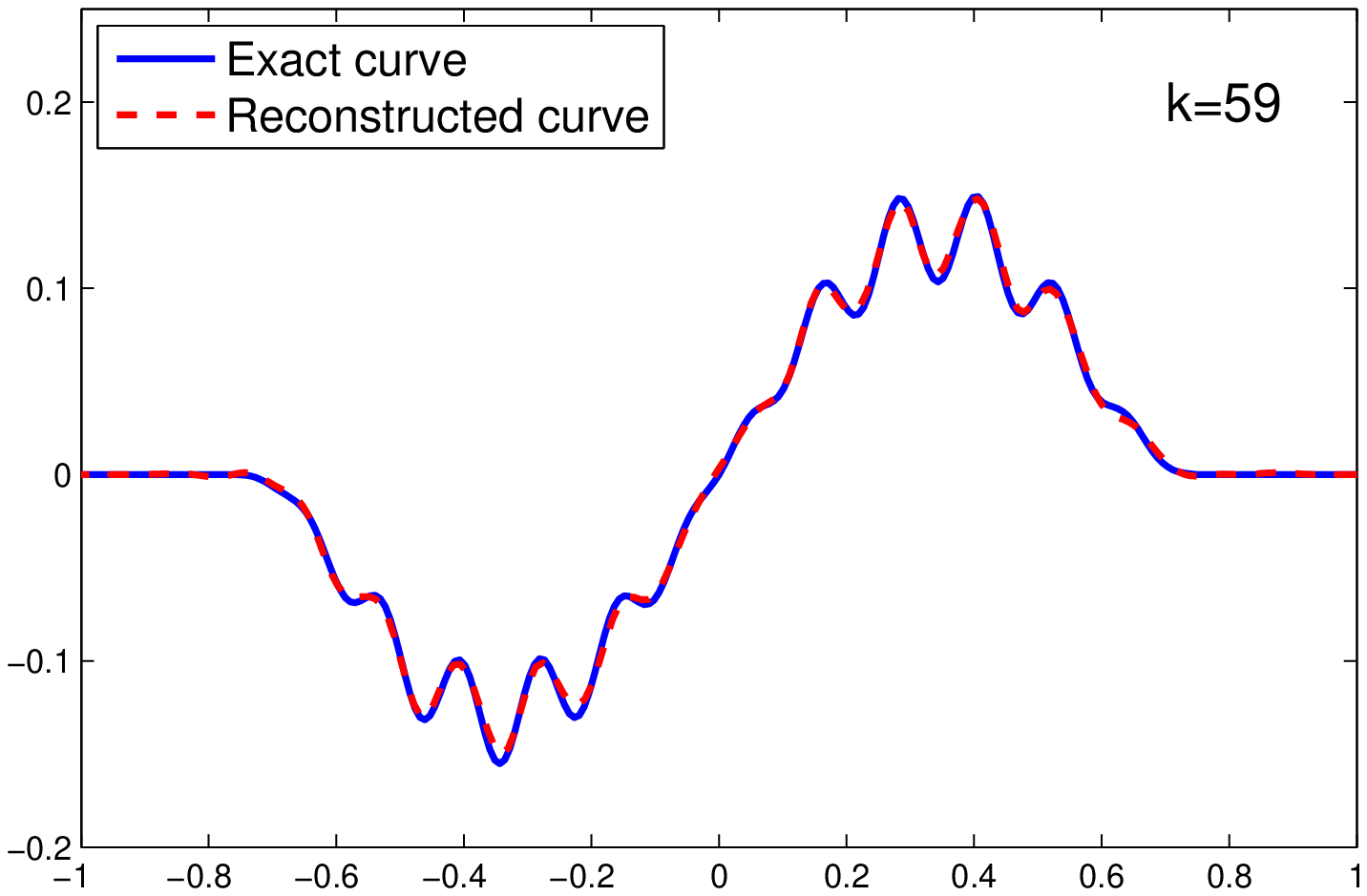}}
  \vspace{0in}
\caption{Reconstruction of a multi-scale surface profile from $5\%$ noisy intensity near-field data
generated by two incident plane waves $u^i=u^i(x;d_l,k),l=1,2$ with different directions
$d_1=(\sin(-\pi/6),-\cos(-\pi/6))$ and $d_2=(\sin(\pi/6),-\cos(\pi/6))$.
Here, the initial and reconstructed curves are presented at $k=13,29,39,49,59.$
}\label{fig4-5}
\end{figure}

\section{Conclusion}\label{se5}

In this paper, it is proved that the translation invariance along the $x_1$-direction of the phaseless far-field pattern
can be broken down if superpositions of two plane waves with different directions are used as the incident fields.
An efficient multi-frequency Newton iterative algorithm is then proposed for reconstructing the locally rough surface profile
from the intensity of the far-field patterns (called the intensity-only far-field data) generated by such incident fields.
The algorithm can also be applied to recover the locally rough surface profile from the intensity of the scattering waves
measured at a straight line segment with a constant height above the surface (called the intensity-only near-field data),
generated by one incident plane wave.
At each iteration, an efficient integral equation solver is used to solve the direct scattering problem which
was proposed previously in \cite{ZhangZhang2013}, so multiple scattering is considered.
Several numerical examples presented here indicate that our inversion algorithm with multi-frequency phaseless far-field and
near-field data gives a stable and accurate reconstruction of the local perturbation of the infinite plane,
even in the multi-scale case.
The reconstruction results are similar to those obtained by the inversion algorithms in \cite{BaoLin11,BaoLin13,ZhangZhang2013}
with using the full far-field and near-field data (including the phase information).
The main reason for this, we think, is that our inversion algorithm considers both multiple scattering (through the use of
the fast integral equation solver for the direct problem) and multiple frequency data.
It was indicated in \cite{pre06,josa08} that multiple scattering plays a key role in subwavelength imaging from far-field data.
Our reconstruction results here together with those given in \cite{BaoLin11,BaoLin13,ZhangZhang2013}
illustrate that multiple scattering in conjunction with using multiple frequency data plays an essential role in
subwavelength imaging from both far-field and near-field data.

\section*{Acknowledgements}

This work was partly supported by the NNSF of China grants 61379093, 91430102 and 11501558.

\end{document}